%% file: Examples5.tex
\documentclass[reqno]{amsart}
\usepackage{hyperref}
\usepackage{cleveref}
\usepackage{epsfig,latexsym,amsfonts,amssymb,amsmath,amscd}
\usepackage{amsthm}
\usepackage[mathscr]{eucal}
\usepackage{mathrsfs}
\usepackage{mathbbol}
\DeclareSymbolFont{yhlargesymbols}{OMX}{yhex}{m}{n}
\DeclareMathAccent{\wideparen}{\mathord}{yhlargesymbols}{"F3}
\newcommand{\myarc}[1]{\wideparen{#1}}
\usepackage{oldgerm,units,color}

\def\mcA{\mathcal{A}}
\newcommand{\val}{\operatorname{val}}
\usepackage{pst-node}
\usepackage{graphicx}
\usepackage[colorinlistoftodos,bordercolor=orange,backgroundcolor=orange!20,linecolor=orange,textsize=scriptsize]{todonotes}
\renewcommand{\theenumi}{\roman{enumi}}%
\renewcommand{\labelenumi}{\theenumi}

\usepackage{eepic, epic}

\usepackage{ifthen}
\newcommand{\sumnabla}{\boxplus_{\nabla}}

\def\Null{{\operatorname{Null}}}

\def\sgn{{\operatorname{sgn}}}

\renewcommand{\geq}{\geqslant}
\renewcommand{\leq}{\leqslant}
\renewcommand{\ge}{\geqslant}
\renewcommand{\le}{\leqslant}
\renewcommand{\preceq}{\preccurlyeq}
\renewcommand{\succeq}{\succcurlyeq}
\newcommand{\phase}{\mathsf{Ph}}
\newcommand{\hgroup}{\mathcal G}
\newcommand{\kfield}{\mathcal K}
\newcommand{\ring}{\mathcal R}
\usepackage{tikz}
\input tropMacro1.tex
\definecolor{lgray}{gray}{0.90}

\def\vep{\varepsilon}

\def\mcI{\mathcal I}

 \def\mfrak#1{\mathsf{#1}}
 \def\mfraka{\mathfrak}
 \def\MMod{\mfrak{M}\text{-}\mfrak{Mod}}
 \def\WMod{\mfrak{wM}\text{-}\mfrak{Mod}}
 \def\MModT{\mfrak{M}\text{-}\mfrak{ModT}}
 
  \def\MHModT{\mfrak{M}\text{-}\mfrak{HModT}}
 \def\MHSrT{\mfrak{M}\text{-}\mfrak{HSrT}}
 \def\FHSrT{\mfrak{G}\text{-}\mfrak{HSfT}}
\def\WMModT{\mfrak{wM}\text{-}\mfrak{ModT}}
\def\MB{\mfrak{M}\text{-}\mfrak{nSr}}
\def\MBT{\mfrak{M}\text{-}\mfrak{nSrT}}

\def\WMBT{\mfrak{wM}\text{-}\mfrak{nSrT}}
\def\MSr{\mfrak{M}\text{-}\mfrak{Sr}}

\def\MSrT{\mfrak{M}\text{-}\mfrak{SrT}}

\def\TBS{\mfrak{TBSr}}

\def\distributed{}

\def\ctw{\cdot_{\operatorname{tw}}}
\def\({\left(}
\def\){\right)}
\def\Z{{\mathbb Z}}
\def\Q{{\mathbb Q}}
\def\C{{\mathbb C}}

\def\pipe{{\underset{{\ \, }}{\mid}}}

\def\vsemifield0{$\nu$-semifield$^\dagger$}
\def\vsemiring0{$\nu$-semiring$^\dagger$}

\def\pipe1{{\underset{{1}}{\mid}}}
\def\lmod1{\mathrel  \pipe1  \joinrel \joinrel =}

\def\CFunFF1{\operatorname{CFun} (F,F)}

\def\semiring0{semiring$^{\dagger}$}
\def\Semiring0{Semiring$^{\dagger}$}
\def\Semirings0{Semirings$^{\dagger}$}
\def\semidomain0{semidomain$^{\dagger}$}
\def\semifield0{semifield$^{\dagger}$}
\def\semifields0{semifields$^{\dagger}$}
\def\vsemifields0{$\nu$-semifields$^{\dagger}$}
\def\domain0{domain$^{\dagger}$}
\def\predomain0{pre-domain$^{\dagger}$}
\def\predomains0{pre-domains$^{\dagger}$}
\def\domains0{domains$^{\dagger}$}
\def\vdomains0{$\nu$-domains$^{\dagger}$}

\def\domains0{domains$^\dagger$}

\newcommand{\etype}[1]{\renewcommand{\labelenumi}{(#1{enumi})}}

\etype{\roman}

\def\pipe{{\underset{{\tG}}{\mid}}}

\def\lmod{\mathrel  \pipe \joinrel \joinrel =}
\def\pipe{{\underset{{\tG}}{\mid}}}

\def\ghost0{\operatorname{ghost}}

\newcommand{\nsets}{{\mathcal P}^*}

\theoremstyle{plain} 

\newtheorem{propt}[theorem]{Properties}
\newtheorem*{thm*}{Theorem}
\theoremstyle{plain}
\newtheorem{lem}[theorem]{Lemma}
\theoremstyle{definition}
\theoremstyle{remark}
\newtheorem{rem}[theorem]{Remark}
\newtheorem*{examp*}{Example}
\newtheorem*{examples*}{Examples}
\newtheorem*{remark*}{Remark}
\theoremstyle{definition}
\newtheorem*{defn*}{Definition}
\newtheorem*{note*}{Note}
\newtheorem{construction}[theorem]{Construction}
\newtheorem*{dig*}{Digression}

\def\R{\Real}
\def\la{\lambda}

\def\tT{\mathcal T}

\def\tTz{\tT_\zero}
\def\Hy{\mathcal H}
\def\hyzero{0}

\numberwithin{equation}{section}

\def\M0{M_{\zero}}

\def\supp{\operatorname{supp}}

\def\PS{P}

\def\Equiv{\Longleftrightarrow}

\def\semirings0{semirings$^\dagger$}

\newcommand{\nPS}[1]{\PS_{(!#1)}}
\newcommand{\nPSo}[1]{\nPS{\one}}

\begin{document}


\title {Semiring systems arising from hyperrings}


\author[M.~Akian ]{Marianne Akian }
\address{Marianne Akian, INRIA and CMAP, \'Ecole polytechnique, IP Paris, CNRS. Address:
CMAP, \'Ecole Polytechnique, Route de Saclay, 91128 Palaiseau Cedex,
France} \email{marianne.akian@inria.fr }

\author[S.~Gaubert]{Stephane~Gaubert}
\address{Stephane~Gaubert, INRIA and CMAP, \'Ecole polytechnique, IP Paris, CNRS. Address:
CMAP, \'Ecole Polytechnique, Route de Saclay, 91128 Palaiseau Cedex,
France}
 \email{stephane.gaubert@inria.fr }

 \author[L.~Rowen]{Louis Rowen}
\address{Louis Rowen,
Department of Mathematics, Bar-Ilan University, Ramat-Gan 52900,
Israel} \email{rowen@math.biu.ac.il}


\date{\today}



\thanks{The research of the third author was supported by the ISF grant 1994/20 and the Anshel Peffer Chair.}
\thanks{\noindent We thank O.~Lorscheid and the referee for helpful comments. }
\thanks{\noindent \underline{\hskip 3cm } \\ File name: \jobname}


\begin{abstract}
  Hyperfields and systems are two algebraic frameworks which have been
  developed to provide a unified approach to classical and tropical structures.
  All hyperfields, and  more generally hyperrings, can be represented by systems. Conversely, we show that the systems arising in this way, called {\it hypersystems},   are characterized by certain elimination axioms.  Systems are preserved
  under standard algebraic constructions; for instance matrices and polynomials over hypersystems are systems, but not hypersystems.
  We illustrate these results by discussing several examples
  of systems and hyperfields, and constructions like matroids over systems.

\end{abstract}

\maketitle


{\small \tableofcontents}



\section{Introduction}


\numberwithin{equation}{section}

\subsection{Overview}$ $

The original algebraic structure of tropical mathematics is the max-plus
semifield, which is the set $\R\cup\{-\infty\}$, with $\max$ under the usual order as addition, and
classical addition as multiplication. This structure appears  when considering non-archimedean valued fields, since
a non-archimedean valuation may be thought of as a weak kind of morphism from
a non-archimedean field to the max-plus semifield. This observation
is at the basis of tropical geometry, see~\cite{MaclaganSturmfels} for background, and this is one of the main motivations for the study
of tropical algebraic structures.

In an attempt to bypass the lack of negation in the tropical world, several approaches have been introduced.

A first approach is to extend the tropical structure with a ``formal negation map''.
This led to the introduction of the symmetrized max-plus semiring in~\cite{Pl,Ga}, see also~\cite{AGG2} for a more recent presentation.
This allows one to solve linear systems by a tropical analogue of the Cramer formula, and to define appropriate notions of matrix
rank~\cite{Ga,AGG1}, which turn out to be more degenerate than their classical analogues. The symmetrized tropical semifield  arises when considering images of ordered nonarchimedean fields by the map which keeps track of the valuation and sign, i.e.,
it is well adapted to the study of real tropicalizations,
which goes back to the work of Viro, see~\cite{viro,viro2001dequantization}

Another extension consists of ``supertropical algebra'', introduced by Izhakian in his dissertation~\cite{izhakianthesis}, see also~\cite{IR}. This algebraic structure allows one to encode the fact that a maximum is achieved twice at least, a condition which appears  when considering ``complex'' tropicalization instead of ``real'' ones, i.e., the tropicalizations of algebraic sets over algebraically closed fields.
Linear algebra, and in particular the notions of rank, have been
extensively studied in the supertropical setting~\cite{IzhakianRowen2009TropicalRank}.

Another approach consist in using hyperstructures, i.e.,
hyperfields and hyperrings, initially developed by Krasner~\cite{krasner}, with motivation from number theory. In these structures, the addition becomes multivalued,
allowing one to deal with the absence of opposite elements
in the usual sense.
 Several important
 examples of hyperfields, and, more generally, hyperrings, can be defined in a natural
 manner, by modding out fields or rings by multiplicative subgroups~\cite{krasner,CC} (see also \Cref{prop-CC}).
 Hyperfields have been recently studied with motivations
 from tropical
 geometry~\cite{Vi,Ju,BB2} and arithmetic~\cite{CC}.
Moreover, the notions of  multirings
and multifields were introduced in~\cite{M}: multifields are the same as  hyperfields, whereas multirings are slightly more general than hyperrings~\cite{Vi}.
As we shall discuss below, hyperfields
are also understood in terms of their power sets, which have addition and multiplication, but need not be a semiring because a hyperfield
need not be {doubly distributive}.

Another  structure is fuzzy rings \cite{Dr,DrW},   introduced by Dress 
 \cite{Dr,DrW} as a
vehicle for unifying and studying matroids. Even more general structures include the notion of tracts \cite{BB2} and blueprints~\cite{Lor1,Lor2}.

Our original motivation, inspired by an oral question of Baker, was to understand exactly what was needed to carry out the proofs in tropical linear algebra, and to build a more general theory of linear algebra over suitable tropical-like structures, including hyperfields.

In 2016, the third author developed a general algebraic
framework called {\em system}, cf.~\cite{Row21},
to provide an  umbrella structure encompassing all of these theories, especially supertropical algebra, as well as ``classical'' algebra. The  structure theory specializes to each of these, with unified proofs.


A system consists of a set $\tTz$ (especially a monoid)  embedded into a set $\mcA$ with more structure, including a ``surpassing relation'' and a ``negation map.'' Such a phenomenon is found in classical algebra by taking $\mcA$ to be an associative algebra graded over a monoid, and $\tTz$ to be the submonoid of homogeneous elements.
In particular one could embed a monoid into a monoid semialgebra, as in \cite{Lor1}.


The ``surpassing relation,''  a pre-order which generalizes equality, allows one to formulate the appropriate notion of ``equation'' in a system,
replacing equality by the balance relation. The surpassing relation has two types: ``$\circ$-type'', motivated by supertropical algebra and symmerized max-plus semiring, and
``hyper-type,'' motivated by  {\it hyperfields} and
 {\it hyperrings}.

Hyperrings provide a special sort of  system,
called \textbf{hypersystems}, in which the set $\tTz$   is precisely the set  of
elements of a hyperring, 
as shown in~\cite[Theorem~3.21]{Row21}, and $G$-\textbf{hypersystems} when $\tT= \tTz\setminus \{\zero\}$ is a group.    Doubly distributive hypersystems give rise to semiring systems, see \Cref{hypersys}.
The structure of a semiring system  is well suited
 to algebraic constructions, like passing from
 scalars to matrices or polynomials. In contrast,
hyperrings are less amenable:  even over a doubly distributive
hyperfield, the multiplication
of polynomials and matrices is not univalent,
because it involves the multi-valued addition in the original hyperfield. Hence,
polynomials and matrices over hyperfields do not constitute hyperrings. Systems can be studied via two main classes of morphisms: ``Weak morphisms'' (\Cref{weam}) and $\preceq$-morphisms (\Cref{mo}).

In this paper, our
first main result, \Cref{func1}, shows how ``doubling'' 
create a formal ``negation map'', which is obtained as for the construction of negative integers.

Next, in \S\ref{hypersy}, we compare the category of systems with the
category of hyperrings.
Using \cite[Theorem~3.21]{Row21}, recalled above,
we embed the category of hyperrings into the category of systems,
which provides the subcategory of hypersystems.

 Conversely, one may ask which semiring systems can be realized as hypersystems.
 In particular, given a system $\mcA$ associated to a monoid $\tTz$, one can 
 ask when $\tTz$ constitutes a hyperring.
We observe that the addition in systems generally differs from the one in hyperrings.
  Nonetheless, we identify
 a condition under which a system corresponds to a hyperring,
cf.~\Cref{hyprec}. This condition involves
elimination axioms. It is stated in terms of the {\em balance} relation,
defined from the ``surpassing relation'' and ``negation map''. 
This result covers many key examples of structures
 used in tropical geometry, involving tropical
 numbers equipped with a multiplicity or a sign information,
which can be represented both as systems and as hyperfields.

In \Cref{Other}, we provide
 functors to categories of fuzzy rings, and blueprints.
 Fuzzy rings  also have two classes of morphisms, which correspond to the corresponding classes of morphisms of systems (of hyper-type).
 The more recent notion of tracts \cite{BB2} also is
encompassed in systems, see \Cref{trac1}.

We also are interested in the diversity of applications, and with related algebraic structures. In
\Cref{Induced}, we provide additional examples and
algebraic constructions that
 arise in the theory of systems.
 In \Cref{layered} we introduce  a
graded construction generalizing the ``layered algebras'' of~\cite{IKR0}
and the ``semidirect products'' of \cite{AGG2}, and which leads to different
complexifications of the tropical numbers.

In \Cref{sec-mat-sys}, we introduce matroids over systems, which generalize valuated matroids over fuzzy rings, allowing ``singular'', that is, non-invertible, values (or equivalently multivalues in hyperrings). This allows one in particular to increase the set of representable matroids, and, as we shall see, to encode degenerate configurations of tropical vectors.

 In   the appendix we  show more generally how to
obtain systems from abstract $\tT$-modules.

Another application  of systems, developed in \cite{GaR}, is to a version of the Grassman semialgebra and a sweeping version of the Cayley-Hamilton Theorem; this has been extended more recently to Clifford semialgebras in \cite{CGR}. Also projective modules in systems were investigated in \cite{JMR,JMR1}.

In another direction, roots and factorization of univariate polynomials
over systems have been studied in
\cite{AGT}, using in particular a notion of multiplicities introduced by Baker and Lorscheid for hyperfields. These questions has also been investigated by Gunn using the framework of hyperfields and idylls \cite{gunn,gunn2}.

We are continuing this study in a paper  \cite{AGR1} that deals especially with linear algebra over
systems  in a more general setting.

\subsection{Our main concepts}\label{Prel1}$ $

Before stating our main results, we review some concepts.
We deal with   special kinds of structures,   defined as follows.
An \textbf{nd-semiring} $(\mathcal A,  +, \cdot, \zero,  \one )$  satisfies all the properties of a
ring, except for negation and distributivity. More precisely, $\mcA$ is endowed with an addition $+$ and a multiplication $\cdot$, also denoted by concatenation, such that $(\mcA, +,\zero)$ and $(\mcA, \cdot,\one)$ are monoids,
and
 $\zero$ is a multiplicatively absorbing element.
$\mathcal A^\times$ denotes the set of invertible elements of
$\mathcal A.$
A \textbf{semiring} \cite{golan92}  is an nd-semiring for which multiplication distributes over addition.
A \textbf{semifield} is a semiring $\mcA$ whose non-zero elements have multiplicative inverses, i.e., $\mcA = \mcA^\times \cup \{\zero\}$.

\begin{definition}\label{distt}\
 \begin{enumerate}  \item A \distributed  left module additively generated by a multiplicative monoid, written \textbf{left mgen module} for short,  is  an additive monoid $(\mcA,+,\zero)$ together with a subset $\tT\subset \mcA$,
  satisfying the following properties:

\begin{enumerate}

    \item  $\tTz: = \tT\cup \{\zero\}$ is a multiplicative monoid with unit $\one\in \tT$, and which additively generates $(\mcA,+,\zero)$.

       \item There  is a multiplication $\tTz \times \mcA \to \mcA,$ satisfying
      $(a_1 a_2)b = a_1(a_2b)$, $\zero b=\zero$,
     $\one b =b$, $a \one =a$ for all $a,a_i\in \tTz$ and $b\in \mcA$.

        \item Multiplication distributes over addition of $\mcA$, i.e., $a\sum a_i = \sum a a_i$ and $a\zero=\zero $
        for all $a,a_i\in \tTz.$ (This  also holds for $a_i\in \mcA$, by (b).)
        \end{enumerate}

 The elements
of $\tT$ are called {\em tangible} and when $\tT$ is fixed, $\mcA$ is also called a \textbf{left $\tT$-gen module}.
        Right mgen and $\tT$-gen modules are defined analogously.
A \textbf{$\tT$-gen bimodule} is a left and right $\tT$-gen module and is also called an
mgen bimodule.

         \item A subset $\mcA'$ of $\mcA$ is a \textbf{left $\tT$-submodule} if it is an additive submonoid and  $ab \in \mcA'$ for all $a\in\tT$ and $b\in \mcA'$.
        Right $\tT$-submodules or $\tT$-sub-bimodules are defined analogously.
    \item        A left $\tT$-gen module $\mcA$ is  {\it left cancellative} if $ab_1 = ab_2$ implies $b_1=b_2$ for $a\in \tT,$
$b_i\in \mcA.$ Likewise for {\it right cancellative right $\tT$-gen modules}. A $\tT$-gen bimodule is \textbf{cancellative} if it is   both   left and right cancellative.

\item A \textbf{\distributed $\tT$-gen nd-semiring}, also called \textbf{mgen nd-semiring} when $\tT$ is understood, is a \distributed   $\tT$-gen  bimodule which is also a multiplicative monoid   $(\mcA,\cdot,\one),$
         for which the $\tTz$ actions coincide with the multiplication of $\mcA$ restricted to the sets
         $\tTz\times \mcA$ and $\mcA\times \tTz$.

   \item  An \textbf{mgen semiring} is an \distributed mgen nd-semiring which is  \textbf{doubly distributive}, meaning
$$(a+b)(c+d)= ac+ad+bd+bd\enspace \forall a,b,c,d\in \mcA.$$
\end{enumerate}
\end{definition}

Although most of our semirings are commutative, we do not require commutativity, in order to allow examples like the semiring $M_n(\mathcal A)$ of $n\times n$
  matrices with entries in $\mathcal A$.

 \subsubsection{Negation maps, triples, and the balance relation} \label{Prel}$ $

\begin{definition}\label{negmap}$ $
    A \textbf{negation map} on a \distributed left $\tT$-gen module $\mcA$ is a
 semigroup isomorphism
$(-) :\mathcal A \to \mathcal A$ written
$a\mapsto (-)a$, that is of order~$\le 2,$  satisfying the following properties:
   \begin{itemize}
          \item      $(-)\one \in \tT.$
          \item  $((-)a)b = (-)(ab) = a ((-)b)$ for $a \in \tT$, and $b \in \mathcal A.$
         \end{itemize}
 We  then write $a(-)b$ for $a + ((-)b),$
and $a = (\pm ) b$ when $ a= b$ or $a = (-)b.$

A negation map of a right $\tT$-gen module,
resp.~a $\tT$-gen bimodule, is defined analogously.
\end{definition}

A negation map is called a ``symmetry'' in \cite[Definition~2.3]{AGG2}.
We cope with the lack of negation by introducing the above weaker negation map.
The identity map is an example of a negation map, and often is used in tropical mathematics. However,
  a subtler use of negation map in semiring theory is provided by symmetrization, to be discussed in \Cref{semidir38,semidir39}.

The negation map $(-)$ satisfies the usual rules of a minus sign, except that
      we need not have $b(-)b = \zero.$

$b^\circ := b (-)b$ will be called a {\em quasi-zero} for all $b\in \mcA$. (a quasi-zero is
called a \textbf{balanced element} in \cite[Definition~2.6]{AGG2}.)

An important
$\tT$-submodule of $\mathcal A$ is:
$$\mathcal A^\circ = \{ b^\circ: b \in \mathcal A\},$$
which takes the place of $\zero$ in much of the theory.

A \textbf{quasi-negative} of $b \in
\mathcal A$ is an element $b'$ such that $b+b' \in \mathcal A^\circ.$
\begin{propt}\label{negm}$ $
\begin{enumerate}
    \item $(-)b = (-)(\one b) = ((-)\one)b$ for all $b\in \mcA.$
     \item $(-)\zero = ((-)\one)\zero =\zero.$
     \item $((-)\one)^2 = (-)((-)\one) = \one.$
\item $(-)\tT = ((-)\one) \tT = \tT.$
\item Any (left or right) $\tT$-submodule $\mathcal B$ of $\mathcal A$ is stable under
the negation map (for all $b\in \mathcal B$, since
$(-) b= (-) (\one b)= ((-) \one) b\in \mathcal B$.

\item $(-)b^\circ = (-)b + ((-)(-)b) = b^\circ .$
\item Any multiplicative monoid homomorphism $f:\tT\to \tT'$ satisfies $f((-)a) =f((-)\one))f(a) = (-)f(a),$ and $f((-)\one))^2 = \one'.$\end{enumerate}
\end{propt}

\begin{definition}$ $\label{def-tripbal}
    \begin{enumerate}
 \item  A \textbf{triple}  $(\mathcal A, \tT, (-))$ is a \distributed  $\tT$-gen bimodule
$\mathcal A$ together with a negation map $(-)$
   for which
        $\mathcal A^\circ \cap \tT= \emptyset.$

      \item A $G$-\textbf{triple} is a
 triple  $(\mathcal A, \tT, (-))$ for which $\tT= \mcA^\times.$

 \item Given a  triple  $(\mathcal A, \tT, (-))$ and a subset $\mathcal I \supseteq \mathcal A^{\circ}$ of $\mathcal A$,
  we say that $b_1$ \textbf{balances} $b_2$ or that the pair $(b_1,b_2)$ is \textbf{balanced} (relatively to $\mathcal I$),
  and we write
  $b_1\nabla_{\mathcal I}\, b_2$, if $b_1(-)b_2
  \in \mathcal I.$
  (This is completely different from the use of
$\nabla$ in \cite{IzhakianRowen2008Matrices2}.)

 \item 
 We
say that
a subset $S \subset {\mathcal A}$  is \textbf{uniquely quasi-negated over $\mathcal
I$} if $a\nabla_{\mathcal I} a'$ for $a,a' \in S$ implies $a= a'$.
 \item  The triple  $(\mathcal A, \tT, (-))$
 is \textbf{uniquely negated} if $\tT$
 is uniquely quasi-negated over ${\mathcal A}^\circ$.
    \end{enumerate}
\end{definition}

The uniquely negated property is  satisfied by  all of our examples
below (in particular \Cref{supert,semidir38,semidir39}, see also
\Cref{hypersy}) but not, ironically, by the original max-plus
semiring $\mathcal{A}=\R\cup\{-\infty\}$, with negation taken as the
identity map, in which $ \mathcal A = \mathcal A^\circ.$

\begin{rem} If $\mathcal I$ is a submodule of $\mcA$,
the relation $\nabla_{\mathcal I}$ is  stable under the negation map, and
is reflexive, symmetric, and satisfies, for elements $a_i\in \tT$, $b, b_i,b_i' \in \mathcal A$:
\begin{enumerate}
    \item
 $b \nabla_{\mathcal I}\, \zero \Leftrightarrow b\in {\mathcal I}$.
    \item $ b^\circ\nabla_{\mathcal I}\, \zero$.
\item If $b_i \nabla_{\mathcal I} b'_i$  for $i=1,2$, then $a_1b_1+a_2b_2 \,\nabla_{\mathcal I}\, a_1b_1'+ a_2b_2'$ and $b_1a_1+b_2a_2 \,\nabla_{\mathcal I}\, b_1'a_1+ b_2'a_2$ for $a_i\in \tT$.
 \end{enumerate}
\end{rem}

\subsubsection{Surpassing relations and systems}$ $

Many results weaken equality to the surpassing relation.
 Let us
recall its definition, also
cf.~\cite[Definition~2.28]{Row21}.
\begin{definition}\label{precedeq07}$ $
\begin{enumerate}
\item
By \textbf{pre-order} we mean a reflexive and transitive relation. A
\textbf{partial order} (PO) is an antisymmetric  pre-order; an
\textbf{order} is a total PO.

\item \label{precedeq07-1} A  \textbf{$\tT$-pre-order}  $(\preceq)$  on  a $\tT$-gen bimodule  $(\mathcal A, \tT)$
 is a   pre-order  satisfying the following, for elements $a_i\in \tT$, $b, b_i,b_i' \in \mathcal A$:
  \begin{enumerate}
 \item\label{surp-0} If $b\preceq \zero$ then $b=\zero.$
    \item\label{surp-1}  $\zero \preceq b^\circ$.
\item\label{surp-2}  If $b_i \preceq b'_i$  for $i=1,2$, then $a_1b_1+a_2b_2 \preceq a_1b_1'+ a_2b_2'$ and $b_1a_1+b_2a_2 \preceq b_1'a_1+ b_2'a_2$.
 \end{enumerate}

\item When $(\mathcal A, \tT)$  is a \distributed $\tT$-gen nd-semiring, we   also require
\begin{enumerate}\setcounter{enumii}{3}
\item\label{surp-22}  If $b_i \preceq b'_i$  for $i=1,2$ then $b_1b_2 \preceq b_1'b_2'$.
 \end{enumerate}
\item \label{precedeq07-2} A \textbf{surpassing relation} on a $\tT$-gen bimodule  $\mathcal A$,
denoted
  $\preceq $,  is a $\tT$-pre-order satisfying the condition
  \begin{enumerate}\setcounter{enumii}{4}
  \item   \label{surp-5} If  $a\preceq a' $ for $a,a' \in \tTz$, then $a =  a'.$
 \end{enumerate}

  \item  A \textbf{strong surpassing relation} on a triple $\mathcal A$
is a surpassing relation   satisfying the following stronger
version:
  \begin{enumerate}\setcounter{enumii}{5}
\item If $b  \preceq a $ for $a \in \tT$ and $b \in \mathcal A$,
then $b=a$. \end{enumerate}

\item  Define $
{\mathcal A}_{\Null}: = \{ b \in \mathcal A : b \succeq
\zero\},$ and write $ \nabla $ for $
\nabla_{\mathcal A_{\Null}}$.

\item A \textbf{system} $(\mathcal A, \tT, (-),\preceq)$ is a
  triple $(\mathcal A, \tT, (-))$ with a surpassing relation $\preceq,$
such that $\tT$ is uniquely quasi-negated over  the $\tT$-sub-bimodule  ${\mathcal A}_{\Null}.$

\item A \textbf{nd-semiring system}  is a  system $(\mathcal A, \tT, (-),\preceq)$  for which
  $\mathcal A$  is a $\tT$-gen nd-semiring.

  \item A \textbf{$G$-system}  is an nd-semiring  system $(\mathcal A, \tT, (-),\preceq)$  for which
 $\tT = \mcA^\times$.

\item A \textbf{semiring system}  is a nd-semiring system $(\mathcal A, \tT, (-),\preceq)$ for which
  $\mcA$ is an mgen semiring.
\end{enumerate}
\end{definition}

\begin{rem}$ $\begin{enumerate}
    \item
    In any nd-semiring system,
${\mathcal A}_{\Null}$ is an ideal (i.e., absorbing multiplication).

    \item
For any  surpassing relation on a triple, if $b \preceq b'$ then $(-)b = ((-)\one )b\preceq ((-)\one) b' =(-)b'$.

    \item
We have $b_1 \preceq b_2 \Rightarrow b_1 \nabla b_2,$
 since $\zero \preceq b_1 (-)b_1 \preceq b_2 (-) b_1.$

 \item The   unique negation in systems is slightly stronger than unique negation with respect to ${\mathcal A}^\circ$, since
 ${\mathcal A}_{\Null}$  may not coincide with $\mathcal A^\circ$.
\end{enumerate}\end{rem}

We are interested especially in three types of systems.

\begin{definition}\label{precex}$ $

 \begin{enumerate}  \item \label{precex-i}(Compare with \cite[Definition~2.8]{AGG2})
The \textbf{$\circ$-surpassing relation} on  a uniquely negated triple  $(\mathcal A, \tT, (-))$ is given by
 $b \preceq_\circ b'$
  whenever $b + c = b'$ for some $c\in \mathcal A^\circ$.
  \item A surpassing relation is of   $\circ$-\textbf{type} if  $  {\mathcal A}_{\Null}= \mathcal A ^\circ.$

\item
A system $(\mathcal A, \tT, (-), \preceq)$ is of \textbf{weak hyper-type} if
for every $b_i\in \mcA$ with $\sum _{i=1}^n b_i \in {\mathcal A}_{\Null}$ there are $a_i\in \tTz$ such that $a_i\preceq b_i$ and $\sum _{i=1}^n a_i  \in {\mathcal A}_{\Null}$.

\item
An nd-semiring system $(\mathcal A, \tT, (-), \preceq)$ is of \textbf{hyper-type} if   for every $b_i$, $b'_i\in \mcA$ with $\sum _i b_i b'_i \in {\mathcal A}_{\Null}$ there are $a_i, a_i'\in \tTz$ such that $a_i\preceq b_i$, $a'_i\preceq b'_i$ and $\sum _i a_i a'_i \in {\mathcal A}_{\Null}$.
 \end{enumerate}\end{definition}

 \begin{rem}$ $
     \begin{enumerate}
         \item The system
 $(\mathcal A, \tT, (-), \preceq_\circ )$ is  of   $\circ$-{type}.

\item In \Cref{hypersy} we shall see that systems of hyperrings are of hyper-type. Taking $b_i' =\one$ in (iv), we see that   hyper-type implies weak hyper-type; taking $b_i' =\zero,$  we see that for every $b\in \mcA$ there is  $a\in \tT$ such that $a\preceq b$.
     \end{enumerate}
 \end{rem}

 \begin{note}
    The applications in the literature often are $G$-systems, although the theory does not require the elements of $\tT$ to be invertible.
 \end{note}

\subsection{Summary of our main results}$ $

Our results are formulated in terms of the functors between the categories defined below in \Cref{cat1}, especially reflective subcategories and dominant functors.

\begin{enumerate}
    \item
  \Cref{func1}. (generalized in \Cref{dou}) The  {\it doubling functor}  sends the category of \distributed mgen bimodules to  the category of
triples; hence  the category of
triples is a reflective subcategory of the category of \distributed mgen bimodules. (Likewise for hyper-type.)

\item  \Cref{hypersys}. There is a faithful functor in which every     hyperring (resp.~hyperfield) gives rise to a corresponding hypersystem (resp.~$G$-hypersystem).

\item  \Cref{hyprec}.  The category of hyperrings  is a reflective subcategory of the category of tangibly balanced
 systems $(\mathcal A,  \tT, (-),\preceq)$ satisfying tangible balance elimination, and $\nabla$-inversion.

 \item  \Cref{sysassoc}. The analogy of \Cref{hyprec}, using  $\preceq$ instead of $\nabla.$

\item  \Cref{fuzz07},\Cref{fuzz071}. (Also see  \Cref{fuzz072}).
There is a  functor (resp.~reflective functor) from the category of $G$-systems of hyper-type, with $\preceq$-morphisms (resp.~weak $\preceq$-morphisms), to the category of fuzzy rings with fuzzy morphisms (resp.~weak $\preceq$-morphisms).

    \item  \Cref{trac1}. Any tract with $N_G$   a $G$-module gives rise to a
  triple of the form $( \mathcal A = \mathbb N[G] , G, (-))$
with a group $\tT = G$, satisfying  unique negation, and to a system
$( \mathcal A = \mathbb N[G] , G, (-),\preceq_{\mathcal I})$ where
$\mathcal I = N_G$. This gives a faithful functor from such tracts to semiring systems.

\item   \Cref{Lap700}  and \Cref{Lap7} show that
matroids over systems extend valuated matroids and matroids over hyperfields. In particular they allow to represent "tropically degenerate" configurations.
\end{enumerate}

 \section{Relevant categories for this paper}\label{cat1}$ $

 Let us view these notions categorically, following
\cite[Definition~4.1]{Row21} and \cite[Section~7]{JuRo}.

\subsection{Categorical reformulation of \Cref{distt}} $ $

\begin{definition}[See for instance \cite{maclane} and \cite{bruguiere-dominant}]$ $ \begin{enumerate}
    \item A  subcategory $\mfraka{C}$ of a category $\mfraka D$ is said to be \textbf{reflective}\footnote{Usually in the literature reflective categories are required also to be full, but we have the key instance
    in doubling, given in \Cref{func1}.} in $\mfraka D$ if for each $\mfraka D$-object~$D$ there exists an $\mfraka C$-object $C_D$ and a $\mfraka D$-morphism ${\displaystyle r_{D}\colon  D\to  C _{D}}$ such that for each $ \mfraka D$-morphism ${\displaystyle f\colon D\to  C}$   to a $\mfraka  C$-object   $C$, there exists a unique $\mfraka  C$-morphism ${\displaystyle {\overline {f}}\colon   C_{D}\to C}$ with ${\displaystyle {\overline {f}}\circ r_{D}=f}.$

    \item An object $  C$ in a category is called a {\em retract} of an object $ D $ if there are morphisms $i: C\to  D$ and $r: D\to  C $ such that $r\circ i=1_{ C}.$

\item A \textbf{dominant functor} is a functor $\Psi : \mfraka C \to \mfraka D$ in which every object of $\mfraka D$ is a retract of an object of the form $\Psi(X)$ for some object $X$ of $\mfraka C$.
 \end{enumerate}
 \end{definition}

 \begin{definition}\label{def7}$ $
 Since in all categories in this paper, the objects under consideration include an underlying set $\tT$ such that  $\tTz$ is a monoid, which varies, we call them \textbf{M-categories}. (In \cite{AGR1} we consider \textbf{ $\tT$-categories}, in which $\tT$ is fixed.)
 \begin{enumerate}

    \item
       Our main  category $\MMod$, expressed in terms of universal algebra, has as a typical object
an mgen bimodule, that is a couple  $(\tT,\mcA)$
satisfying the relations of
\Cref{distt}(i) and (ii).

  \item
 A  \textbf{bimodule multiplicative map} is a map $f$ sending a  \distributed  $\tT$-gen bimodule $\mcA$ to  a  \distributed   $\tT'$-gen bimodule $\mcA'$,  such that  $f(\tTz)\subseteq \tTz'$, $f(\zero) = \zero',$ $f(\one) = \one',$ and
$$ f (ab) = f (a)f (b),  \qquad f (ba) = f (b)f (a), \quad\text{for all}\; a\in \tT,\; b\in \mcA\enspace .$$

 \item    A \textbf{bimodule homomorphism}  from a  \distributed  mgen bimodule $\mcA$ to  a  \distributed   mgen bimodule $\mcA'$ is a bimodule multiplicative  map  $f$ also preserving addition
 from $ (\mcA,+,\zero) $ to $(\mcA',+,\zero') $.
 As usual, an {\it isomorphism} is a homomorphism which is 1:1 and onto, so its inverse map is also a  homomorphism.

Our morphisms in  $\MMod$ are the bimodule homomorphisms.

    \item
The  category $\MB$  is   the  subcategory in which $\mcA$ now is an \distributed mgen nd-semiring. Here the
  homomorphisms also must preserve nd-semiring multiplication, i.e., $f(b_1b_2)=f(b_1)f(b_2)$ for $b_1,b_2\in \mcA.$

    \item
The  category  $\MSr$  is   the full subcategory of $\MB$,  in which $\mcA$ is an \distributed mgen semiring.

    \item
The  category $\MModT$ is   the   subcategory of $\MMod$  whose objects are triples,  and whose morphisms also satisfy $f((-)\one)= (-)\one' $.

    \item
The  category $\MBT$ is  the  subcategory of $\MB$ whose objects are nd-semiring  triples, and whose morphisms also satisfy $f((-)\one)= (-)\one' $.

\item
The  category $\MSrT$ ($\mathfrak{SDM}$ of \cite[Section~7]{JuRo})  is   the  full subcategory of $\MBT$,  whose objects are semiring  triples.
 \end{enumerate}
 \end{definition}

\begin{rem}$ $\begin{enumerate}
    \item
    The condition $f((-)\one)= (-)\one' $ is automatic for additively  cancellative mgen bimodules, but otherwise this stipulation is an obstruction to the subcategory $\MModT$
    (resp.\ $\MBT$, $\MSrT$) of $\MMod$ (resp.\ $\MB$, $\MSr$) being full.
    \item The negation map was viewed as an endofunctor in \cite{JuRo}. Although this is fine for  $\MModT$, it does not apply to $\MBT$, since for a multiplicative morphism $f$, $(-)f(b_1b_2) \ne ((-)f(b_1))(-)f(b_2)$ for $(-)$ of the second kind.
\end{enumerate}
  \end{rem}

\subsection{Weak morphisms }$ $

Following \cite{DrW}, \cite{GJL} also brings in the notion of {\it weak morphism}.

\begin{definition}\label{weam}

  A \textbf{weak morphism} $f$ of  systems from $(\mathcal A, \tT, (-),\preceq)$ to $(\mathcal A', \tT', (-)',\preceq')$
   is a  multiplicative map $f: \tT\to \tT'$ satisfying merely the property that $\sum a_i \in {\mathcal A}_{\Null}$ for $a_i\in \tT$ implies
  $\sum f (a_i) \in {\mathcal A'}_{\Null}.$
  \end{definition}

\begin{note} Although this paper deals with systems,
    we could use any ideal $\mathcal{I}$ instead of ${\mathcal A}_{\Null},$ in which case we do not strictly need the surpassing  relation.
    This avenue is pursued in \cite{AGR1}.
\end{note}

\begin{definition}
  $ $\begin{enumerate}
  \item   $\WMModT$ is the category with the same objects as  $\MModT$  but whose morphisms are {\it weak morphisms}.

\item Likewise  for $\WMBT$.
\end{enumerate}
\end{definition}

\subsection{$\preceq$- morphisms }$ $

Homomorphisms often are too strong for a general systemic theory.
Once we introduce the surpassing relation, we may enter it into the definition of morphism:

\begin{definition}\label{mo}
A $\preceq$-\textbf{morphism} $f: (\mcA,\tT,(-),\preceq) \to (\mcA',\tT',(-),\preceq')$  of systems is a map   $f :\mcA \to \mcA'$ which restricts to a  monoid homomorphism   $ (\tT,\cdot,\one) \to (\tT',\cdot,\one')$, for which
 \begin{enumerate}
 \item $f(\zero) =  \zero' ,$
\item $f (ab) = f (a)f (b)  $ and $f (ba) = f (b)f (a)  $ for all $a\in \tT,$ $b\in \mcA,$
\item $b_1\preceq b_2$ implies
 $f(b_1)\preceq' f( b_2)$, for all $b_1,b_2 \in
\mathcal A,$
\item
$f (b_1 + b_2)\preceq f(b_1)+f(b_2)$ for all $b_i\in \mcA.$
 \end{enumerate}

\end{definition}

The category $(\MSrT;\preceq)$ of semiring systems and its functors to other categories was studied in depth in \cite{Row21,JuRo}, but we are interested here in more general situations, to show how  $(\MModT,\preceq)$ encompasses tropical systems, hypersystems, doubling, and fuzzy systems. The more general approach in~the appendix gives rise to further interesting connections, to be discussed there.

Let us see why the annoying condition $f((-)\one)= (-)\one' $ of \Cref{def7} can be bypassed in \Cref{mo}, yielding all the properties of \cite[Definition~4.1]{Row21}.

\begin{lem}\label{ful1}
   Any $\preceq$-morphism $f$ of  systems satisfies the following:
   \begin{enumerate}
   \item $f(\mcA_{\Null}) \subseteq (\mcA')_{\Null};$
\item $f((-)\one) =(-)\one' ;$
       \item $f((-)a) =(-)f(a),$  for all $a\in \mcA$.
      \end{enumerate}
\end{lem}
\begin{proof}
(i) $\zero \preceq b$ implies $\zero' =f(\zero)\preceq' f(b).$

(ii) Since $\one +(-)\one \in \mcA_{\Null}$, we get that
$\zero'\preceq' f(\one +(-)\one)\preceq' f(\one) +f((-)\one) $.
This implies $f((-)\one) = (-)f(\one)=(-)\one'$ since any  system is uniquely negated.

(iii) $f((-)a) = f((-)\one)a) = f((-)\one)f(a) = ((-)\one') f(a) = (-)f(a).$

\end{proof}

\begin{example}
Tropicalization is a $\preceq$-morphism, as explained in \cite[Proposition~9.2]{Row21}
Another example, closely related to this paper, is \cite[Theorem~7.4]{Row21}.
\end{example}

$(\mfrak{DModT},\preceq)$, $(\MBT,\preceq)$, and
$(\MSrT,\preceq)$  are the appropriate categories of systems,
whose morphisms now are the $\preceq$-morphisms.

 The  category $(\MModT,\preceq_{\operatorname{hyp}})$ is the full
 subcategory of $(\MModT,\preceq)$ whose objects $(\mathcal A, \tT, (-),\subseteq)$ are systems  of weak hyper-type. This has the subcategory $(\MBT,\preceq_{\operatorname{hyp}})$ of nd-semiring systems of hyper-type, which in turn has the  subcategory $(\MSrT,\preceq_{\operatorname{hyp}})$
of semiring systems of hyper-type.

The  category $(\MModT,\preceq_{\operatorname{Null}})$ is the full
 subcategory of $(\MModT,\preceq)$ whose objects  $(\mathcal A, \tT, (-),\subseteq)$ have surpassing relations of $\circ$-type.

 In both cases the morphisms are $\preceq$-morphisms, although in the latter case one could enlarge this to the set of weak morphisms.

\section{Preliminary results}\label{Prel12}

 $\Net$ denotes the
semiring of natural numbers, including~$0$, and $\Net_{>0}:=\Net\setminus\{0\}$.
The monoid $(\Net_{>0},\times )$ acts on any semigroup $(\mathcal A,+) $ via
$mb := b+\cdots + b$, where the sum is taken $m$ times.

\begin{proposition}\label{dc} Suppose that
$\mathcal A$ is a $\tT$-gen bimodule. 
We can define a  doubly distributive multiplication on $\mathcal A$,
extending the given multiplications $\tT \times \mcA\to \mcA$ and $\mcA\times \tT\to \mcA$,
 defined  by  \begin{equation}\label{eq2.17}
\left(\sum_i
 a_i \right)\left(\sum_j a_j'\right) = \sum_{i,j}
 (a_ia_j'),\end{equation}
 with respect to which $\mcA$ becomes a $\tT$-gen semiring. 

 If $\mathcal A$ already is an mgen semiring,
then the multiplication in $\mathcal A$ coincides with \eqref{eq2.17}.
\end{proposition}
\begin{proof} First notice that the formula is well-defined, since if $\sum_j a'_j  =
\sum_k a''_k $ then $$\sum_{i,j} (a_i a'_j)   =
\sum_{i,j} a_i (a'_j )
=\sum_i a_i
\left(\sum_j  a'_j   \right)
=\sum_i a_i
\left(\sum_k  a''_k   \right)
%
=
\sum_{i,k} (a_i a''_k) ,$$ and likewise if we change the $a_i$.
Now, (double) distributivity is clear.

 Next, $\one$ remains the unit element of $\mcA.$
  Associativity follows from \eqref{eq2.17} and
the associativity of multiplication in $\tT$.

To show the second part of the proposition, we assume that $\mathcal A$ is a semiring with a multiplication~$\times$
consistent with the action of $\tT$,  and unit $\one$.
Let $a,a'\in \tT$.
From the consistency of multiplication with the action, we have $(a \one)\times (a'\one)=a (\one\times (\one a'))=a ((\one\times \one) a')$.
 We have $\one\times \one=\one$, which leads to
$(a \one)\times (a'\one)= a (\one a')= a (a'\one)=(aa') \one$.
Then,  from the distributivity of $\times$ over addition, we
obtain that
$\left(\sum_i  a_i \one\right)\times \left(\sum_j a_j'\one\right) = \sum_{i,j}
(a_i \one)\times (a_j' \one)= \sum_{i,j} (a_i a'_j)\one$, which means that
the multiplication in $\mathcal A$ coincides with \eqref{eq2.17}.
\end{proof}

(This also follows from \cite[Theorem~2.45]{Row21}.) Nonetheless, we do not necessarily want to consider~$\mathcal A$ with this new semiring structure, as we shall see in the treatment of hyperrings.


\begin{table}

\begin{tikzpicture}[scale=0.85]
 \node (top) at (0,0) {\begin{tabular}{c}Triples $(\mathcal{A},\mathcal{T},(-))$\end{tabular}};
 \node (topa) at (0,1.5)
       {\begin{tabular}{c} $\mathcal{T}$-gen modules $\mathcal{A}$\end{tabular}} edge[thick,bend left=45,->] (top.east);
       \node at (4.5,0.75) (doubling) {doubling, Th.~\ref{func1}, (i)};
    \draw[->] (top) -- (topa);

    \node at (0*1.5,-1.5) (btop) {Systems $(\mathcal{A},\mathcal{T},(-),\preceq )$};

    \node[fill=white] at (5.5,-1.5) (fuzzyring) {Coherent Fuzzy rings};
    \draw[->] (fuzzyring) --(btop);
    \node[fill=white] at (3.5,-3) (hypertype) {Hyper-type};
   \node[fill=white] at (3.5,-4.5) (tbs) {$(\TBS,\preceq)$};
   \draw[->] (hypertype) -- (btop);
   \draw[<-] (top) -- (btop);
    \draw[->] (tbs) -- (hypertype);
   \node[fill=white] at (-4,-3) (metat) {Metatangible};
   \draw[->] (metat) -- (btop);
   \node[fill=white] at (-4,-4.5) (bipot) {$(-)$-Bipotent};
   \draw[->] (bipot) -- (metat);
   \node[fill=white] at (3.5,-6) (hsys) {Hypersystems} edge[thick,bend left=45,<-] (tbs.west);
 \node[fill=white] at (0,-7.5) (shsys) {Semiring hypersystems};
   \node at (1,-5) (labelth) {Th.~\ref{hyprec}};
   \draw[->] (hsys) -- (tbs);
   \node[fill=white] at (7,-6) (hrings) {Hyperrings};
   \draw[thick,<->] (hsys) -- (hrings);
    \node at (5.5,-5.5) (equiv) {Th.~\ref{hypersys}, \eqref{hypersys-i}};
   \node[fill=white] at (7,-7.5) (regular) {Doubly distributive hyperrings};
   \draw[->] (regular) -- (hrings);
    \draw[<->] (shsys) -- (regular);
\draw[->] (shsys) --(hsys);
\node at (3.5,-7) (labelthnew) {Th.~\ref{hypersys}, \eqref{hypersys-iii}};
\end{tikzpicture}

  \caption{Triples, systems and hypersystems.}
\end{table}


\begin{definition}
 We define the \textbf{height} of an element $c$ in a
triple $(\mathcal A, \tT,(-))$ as the minimal $t$ such that $c =
\sum_{i=1}^t a_i$ with each $a_i \in \tT.$ (We say that $\zero$ has
height 0.) The \textbf{height} of $\mathcal A$ (if it exists) is the
maximal height of its elements.
\end{definition}
If $(\mathcal A, \tT,(-))$ is a triple, we can define
$\tT^\circ = \{ b^\circ: b \in \tT\}$, and
$\tTz^\circ= \tT^\circ\cup\{\zero\}$.
Then $\mathcal{A}^\circ\setminus{\zero} $
 is additively spanned by $\tT^\circ$, and
$\mathcal{A}^\circ $ is additively spanned by $\tTz^\circ$.

 A triple $(\mathcal A, \tT, (-))$ is called \textbf{shallow} if
 $ \mathcal A =  \tT ^\circ \cup \tTz$. In this case $\mcA$ has height 2.

When $\mathcal A$ is a semiring, we call $(\mathcal A, \tT, (-))$  a \textbf{semiring triple}.
Then $A^\circ$ is clearly an ideal.

Note that any additively indecomposable element of $\mathcal A$ must be tangible.

A  triple $(\mathcal A, \tT, (-))$ is $\tT$-\textbf{trivial} if $\tT=\{\one\}$,
and  is $(\tT,(-))$-\textbf{trivial} if $\tT
= \{(\pm) \one \}.$  An interesting example is the ``characteristic
triple'' where $\tT = \{(\pm) \one \}$ and $\mathcal A = \tT \cup
\{\zero, \one^\circ\} = \tTz \cup \{ \one^\circ\}$, which  is shallow; here $\mathcal A^\circ = \{\zero, \one^\circ\}$.




The   various negation maps are treated in a unified manner. The
negation map $(-)$ is said to be of the \textbf{first kind} if
$(-)\one = \one$ (and thus $(-)$ is the identity), and of the
\textbf{second kind} if $(-)a \ne a$ for all $a \in \tT$.

\begin{rem}
Any negation map on a cancellable \distributed mgen bimodule $\mcA$ is of first or second kind,
and  it is enough to check
whether or not $(-)\one = \one$ to determine the kind of the
negation map.
(Indeed, if $(-)a = a,$ then $a((-)\one)  = a\one, $ implying $(-)\one =\one.$)
\end{rem}

Negation
maps of the first kind are used in the supertropical system, see
\Cref{supert}, and negation maps of the second kind are
used in the symmetrized systems, see \Cref{semidir38,semidir39}.

\subsection{Metatangible and $(-)$-bipotent triples}$ $

Recall from \cite[Definition~2.25 and Lemma~6.1]{Row21} that a
\textbf{metatangible triple}  is a uniquely negated triple
 for which $a+b \in \tT$ for any $a \ne (-) b$
in~$\tT.$ Metatangible triples are the mainstay of \cite{Row21}.

We strengthen metatangibility a bit.

\begin{definition}\label{bipot}
 A triple $(\mathcal
A, \tT, (-))$ is \textbf{$(-)$-bipotent} if $a + a' \in \{a,a'\}$
whenever $a, a' \in \tT$ with  $a' \neq (-) a.$  \end{definition}
This is a relaxed version of the classical notion of bipotence of a semigroup $(S,+)$,
which requires that $a+a'\in \{a,a'\}$ for all $a,a'\in S$.

 As
discussed in~\cite[\S 3.3.1]{Row21}, the triples used in tropicalization
(related to the max-plus algebra) are all $(-)$-bipotent, thereby
motivating us to develop the algebraic theory of such triples. Also see \Cref{semidir39}.

Examples of non-$(-)$-bipotent, metatangible triples are $(\mathbb
Z,+),$ taking $\tT = \mathbb Z \setminus \{ \zero\},$ or $(\mathbb
Z/n,+),$ but they  have a classical flavor.


\begin{proposition}[{\cite[Theorem 6.28]{Row21}}]\label{circint3}
  Assume that the triple $(\mathcal{A},\tT,(-))$ is metatangible and
cancellative,  and that $b\in\mathcal{A}$.
Then, $b$ has a \textbf{uniform presentation} as follows:  for $m\ne 2$, $b=mb_{\tT}$ where $b_\tT \in \tTz$, and $b = b_\tT^\circ$ for $m=2.$ Furthermore if $m$ is minimal then $b$  has height $m$. \end{proposition}



%

Define
\begin{equation}\label{eq:al} e= \one^\circ  = \one (-) \one , \qquad e' = e + \one =  \one (-) \one + \one  .\end{equation}

We handle the   case $m=2$
as follows,
for $(-)$-bipotent triples.

\begin{proposition}\label{prop-uni} In a $(-)$-bipotent cancellative  triple,
 \begin{enumerate}
 \item \label{prop-uni-i}
 If $a(-)a = a'(-)a'$ for $a,a'$
tangible then $a' = (\pm) a.$
 \item \label{prop-uni-ii} The uniform presentation $b=mb_{\tT}$ of  an element $b$ is
 unique, with the only exceptional case of $b={b'_\tT}^ \circ = b_\tT^ \circ$
 with $b'_\tT  = (\pm) b_\tT.$
 In particular $m$ is the height of $b$.
\end{enumerate}
 \end{proposition}
\begin{proof}
 \eqref{prop-uni-i} Assume on the contrary  that $a' \neq (\pm) a$. Then,
we may assume for instance that $a'+a = a'$.
By \cite[Lemma~2.15~(i)]{Row21}, this implies $a'(-)a = a'$. So
$a'(-)a'+ a' = a(-)a + a' = a',$ implying $e'=\one,$ but $a(-)a + a=
a'(-)a' + a = a'(-)a' = a(-)a,$ implying $e'=e,$ so $\one=e$.
Since $a(-)a = a'(-)a'$, this implies $a=a(-)a = a'(-)a'=a'$,
which contradicts $a' \neq (\pm) a$.

\eqref{prop-uni-ii} \cite[Theorem~6.28]{Row21} yields uniqueness except in the case
 $b={b'_\tT}^ \circ = b_\tT^ \circ$,  and then, by \eqref{prop-uni-i}, $b'_\tT  = (\pm) b_\tT.$
 \end{proof}

\begin{proposition}\label{neg3} Suppose the triple   $(\mathcal
A, \tT, (-))$ is cancellative and $(-)$-bipotent, and let $b, b' \in \mathcal A$ with
$b'_\tT \ne (\pm) b_\tT.$ Then one of the following holds:
\begin{enumerate}
\item 
$b_\tT+b'_\tT = b_\tT$,
$b_\tT(-)b'_\tT = b_\tT$,
$b+b' = b$ and
 $b(-) b' = b,$
 \item 
 $b_\tT+b'_\tT = b'_\tT$,
 $b'_\tT(-)b_\tT = b'_\tT$,
 $b+b' = b'$ and   $b'(-) b = b'$.
 \end{enumerate}
 \end{proposition}
\begin{proof} We start with \cite[Lemma~2.15~(i)]{Row21}, that $a+a' = a$ for $a\neq a'$ in $\tT$ implies $a(-)a' = a.$
Thus, if $b_\tT \ne (\pm) b'_\tT$ then $b_\tT + b'_\tT $ is either $b_\tT$ or $b'_\tT $.
In the first case  $b_\tT (-) b'_\tT = b_\tT$ and thus
since $b=c+b_\tT$ for some $c\in \mathcal A$, we get
$b (\pm) b' = c+ b_\tT  (\pm) b'_\tT  \dots (\pm) b'_\tT = c+b_\tT=b $ since the summand $b_\tT $ absorbs all the $b'_\tT $ and $(-) b'_\tT $.
In the second case, by the same reasoning, we get that
$b' +b = b'$ and $b'(-) b= b'$.
 \end{proof}


\begin{definition}\label{circidem} A  triple  $\mathcal A$ is
$\circ$-\textbf{idempotent} if $e^\circ  = e $.

 A
 \textbf{strongly $(-)$-bipotent} triple is a
$(-)$-bipotent
 triple for which  $e'=e$.
\end{definition}

The condition $e'=e$ holds in the supertropical system,
as recalled below (see \Cref{supert}), and is very
useful in computations.

\begin{lem}\label{ht2}$ $  \begin{enumerate}
 \item  \label{ht2-i}$a^\circ + a^\circ = a^\circ $ for
all $a\in \tT,$ in any $\circ$-idempotent triple.
    \item \label{ht2-ii} Any $\circ$-idempotent $(-)$-bipotent
 triple $\mathcal A$ has height $\le 3$. All elements of  height $ 3$
 are $ae'$ where $(-)\one = \one,$ and all elements of  height $ 2$ are $ae.$
  \item \label{ht2-iii} Any $(-)$-bipotent
 triple satisfying $e'=e$ or $e' = \one$ is also $\circ$-idempotent
and shallow.

\end{enumerate}
 \end{lem}
\begin{proof} \eqref{ht2-i} $a^\circ + a^\circ = a(e+e) = a e^\circ = a e=a^\circ .$

\eqref{ht2-ii} The sum of four elements $a_i$ of $\tT$ is reduced to a
sum of three elements of $\tT$. This is immediate when $a_i\neq (-) a_j$ for some $i\neq j$.
Otherwise, $a_i=(-)a_j$ for all $i\neq j$, in particular $a_1=(-)a_2=a_3=(-)a_4$,
so  $\sum_i a_i=(a_1^\circ)^\circ=a_1((e^\circ)^\circ) = a_1e^\circ = a_1^\circ$, which is the sum of two elements of
$\tT$. Hence, by induction, all the elements of $\mathcal A$
have  height~$\le 3$.
The same arguments also show that elements
of height $3$ are of the form
$ae'$ with $(-)a=a$, and elements of height~ $2$ are of the form $a^\circ$.
Also,
$a^\circ+ b$ is either $a^\circ,$ $b$, or (in case $a = (\pm b)$) $(\pm) a e'$.

\eqref{ht2-iii} If $e'=e$ or $e' = \one$, then
either respectively $e^\circ = e' (-) \one = e (-) \one = (-)e
= e $ or $e^\circ = e' (-) \one = \one (-) \one =  e $, yielding
$\circ$-idempotence.
Moreover, by the previous item, there are no
elements of $\mathcal A$ of height~3, so $\mathcal A$ is shallow.
\end{proof}

%


\subsection{More about surpassing relations and systems}\label{mex77}$ $

\begin{definition}[{\cite[Definition 2.36]{Row21}}]\label{syst}
 A  system is \textbf{strong} if $ \preceq$ is a strong surpassing relation.
\end{definition}

\begin{example}\label{ex-acirc}
Non-hyperring examples of surpassing POs on
a triple with  ${\mathcal A}_{\Null}\ne \mathcal A^\circ$:
 \begin{enumerate}
 \item  This example is in the spirit of hyperrings (\Cref{hyp}), but lacking  uniqueness
  of hypernegatives. $ \mathcal A$ is the set of closed intervals $\mfrak b$ on the real
 line, and   $\tT = \{  [a,a]: a \in \Real\}.$ Put $\zero = [0,0]$, and addition is given by $ \mfrak b + \zero = \zero +\mfrak b =\mfrak
 b$,
  and for $ \mfrak b_1, \mfrak b_2\ne \zero $, $ \mfrak b_1 + \mfrak b_2$ is
  the convex hull of $ \mfrak b_1 \cup \mfrak b_1,$ in analogy to hyperfields.  This is
  associative, since associativity holds outside $\zero$, and $\zero$ cannot be written as a proper sum. The negation map $(-)$ is given by
  the classical negation of the interval. We
  say that $ \mfrak b_1 \preceq  \mfrak b_2$ if $ \mfrak b_1 \subseteq \mfrak
  b_2.$ Then $ \mathcal A^\circ $ is the set of intervals $[-a,a],$
  for $a \ge 0,$ and for any $a \in \Real$, we have $[a_1,a_1] +
  [a_2,a_2] = [-a,a]$  only when $a_1 = -a$ and  $a_2 =  a$ or visa
  versa.
 But any interval $[a_1,a_2]$ containing 0 succeeds 0 by
  definition, and  need not be in $ \mathcal A^\circ $.
  \item We take $\tT = (\mathbb Z,+)$ and  $ \mathcal
  A$ to be the subset of the power set of $\tT$ spanned by all sums
  under the action:
  $$ a + \zero = \zero + a =a, \qquad a_1 + a_2 = \{ 0, \ a_1, \ a_2,
  \ a_1+a_2\},\quad \forall a_1 , a_2 \ne \zero.$$ Addition in $ \mathcal
  A$ of sets
  is defined element wise, so in particular $ a_1 + a_2 =  \{ 0, a_1\} + \{ 0, a_2
  \}.$ Associativity follows easily from element wise associativity.
 Now $a (-)a = \{ 0, \pm a\}$ and has only three
  elements, and $ a_1+a_2 \in  \mathcal A^\circ $  only when $ a_1 = -a_2.$
  Since $ A^\circ$ is
  symmetric under negation,
 the set $1+2= \{ 0, 1, 2, 3 \}$ cannot be in  $ \mathcal
  A^\circ.$
 \end{enumerate}
\end{example}

 The surpassing relation is very important,
providing ${\mathcal A}_{\Null} $ which   plays a
critical role; in a system it is disjoint from $\tT$.
  In most of our theorems generalizing classical algebra,
 $\preceq$ replaces equality and
 $\mathcal A_{\operatorname{Null}}$ replaces $\zero$.

\begin{rem}$ $
    \begin{enumerate}
        \item When   $(\mathcal A, \tT, (-))$ is
a shallow triple we usually take $\preceq= \preceq_\circ $.

  \item The  $\circ$-surpassing relation is  a partial order precisely when $\mathcal A$ satisfies the property that

  $a + c_1^\circ  + c_2^\circ = a$ implies   $a + c_1^\circ = a$.

    \end{enumerate}
\end{rem}

\begin{example}\label{supert}$ $
\begin{enumerate}  \item
The main triple of tropical algebra is $(\mathcal A, \tT, (-))$,
where $\mathcal A$ is the supertropical semiring of~$\tT$,
according to the terminology of \cite[Definition~3.4]{IR},  $\tT$ is
any ordered abelian group, called the set of ``tangible elements'',
and the negation map $(-)$ is the identity map. Then, $a^\circ = a+a$ is written
as the ``ghost'' element $a^\nu$, $\tT^\nu$ is a copy of $\tT$, $\tG
= \tT^\nu\cup\{\zero\}$ is the ``ghost ideal'' which is the same as
${\mathcal A}_\Null$, and $\mathcal A =  \tT \cup \tG$.
It is a strongly $(-)$-bipotent  triple of the first kind.

Moreover $(\mathcal A, \tT, (-))$
becomes a strong system, where the
surpassing partial order $ \preceq_\circ $ is called ``ghost surpassing'' in
\cite{IR}; here $a \preceq_\circ b$
  whenever $a + c = b$ for some $c \in \tG$.

 We shall call $(\mathcal A, \tT, (-), \preceq_\circ )$  the
 \textbf{supertropical triple} or \textbf{system}.

A $(-)$-bipotent system with a $\circ$-surpassing relation is {\it not}
 of weak hyper-type, since one can take $b_1+b_2 = b_2$ whereas $b_1^\circ + b_2^\circ \in \mcA^\circ.$
 \item  The
\textbf{symmetrized system} is
$(\widehat{\mathcal A}, \tT_{\widehat{\mathcal A}},(-), \preceq_\circ) $,
as in \Cref{semidir38}, when  $\tT$ additively generates
$\mathcal A \setminus \zero$,
or $(\widehat{\mathcal A}', \tT_{\widehat{\mathcal A}},(-), \preceq_\circ) $,
as in \Cref{semidir39}, when $\tT$ is a bipotent subsemigroup of
$({\mathcal A},+)$. 

%

\end{enumerate}\end{example}


The supertropical system and
the symmetrized system are strong.
As we shall see in next section, hypersystems are also strong.

\subsubsection{Systems obtained from a  $\tT$-sub-bimodule $\mathcal I \supseteq \mathcal A_{\circ}$}\label{mex17}$ $

Although the $\circ$-surpassing relation is the natural surpassing
relation that comes up in tropical mathematics, there are occasions
  in which we want a more general formulation
which enables us to choose ${\mathcal A}_{\Null}$ to be any chosen $\tT$-sub-bimodule $\mathcal I \supseteq \mathcal A^{\circ}$  of $\mathcal A$ disjoint from $\tT$
(recall that $\mathcal I$ is then necessarily stable under
the negation map).

 \begin{definition}
\label{precex-2}
Given a triple  $(\mathcal A, \tT, (-))$ and a
  $\tT$-sub-bimodule $\mathcal I \supseteq \mathcal A^{\circ}$ of $\mathcal A$, we define the  $\mathcal I$-\textbf{pre-order}  $ \preceq_{\mathcal I}
  $ by
    $b_1  \preceq_{\mathcal I} b_2$ if  $b_2  = b_1 +c$ for some $c \in \mathcal
    I.$


\end{definition}

  As a special case of \Cref{precex-2}, $ \preceq_{\mathcal I}$ is $ \preceq_{\circ}$
      when $  \mathcal I={\mathcal A}^\circ.$

   \begin{lem}\label{newsys}$ $  \begin{enumerate}
 \item\label{newsys-1}  The $\mathcal I$-pre-order   $ \preceq_{\mathcal I}$ is a $\tT$-pre-order. 

\item \label{heigh1}   The  relation
  $ \preceq_{\mathcal I}$    is a surpassing relation if
  one of the following equivalent conditions holds:
  \begin{enumerate}
    \item\label{newsys-1a}     $\tTz $   is uniquely quasi-negated over $\mathcal I$;
  \item\label{newsys-1b}  $ \mathcal  I \cap \tT= \emptyset,$ and $\tT$   is uniquely quasi-negated over $\mathcal I$.
\end{enumerate}
\item \label{newsys-4} In case the previous conditions hold, we obtain the system $(\mathcal A, \tT, (-), \preceq_{\mathcal
 I})$,
for which ${\mathcal A}_{\Null} = \mathcal I.$
\end{enumerate}
 \end{lem}
\begin{proof} \eqref{newsys-1}:   If    $b_1 \preceq b_2$ and $b_2 \preceq b_3$ then
   $b _i + c_i = b_{i+1}$  for $i= 1,2$, with $c_i\in{\mathcal I}$,
   so   $b_1 + c_1+c_2 = b_{3}$ and $b_1 \preceq b_3$, since $\mathcal I$ is a module. So $\preceq_{\mathcal I}$ is a $\tT$-pre-order.

Similarly, if    $b_1 \preceq b_1'$ and $b_2 \preceq b_2'$ then
   $b _i + c_i = b_i'$  for $i= 1,2$, with $c_i\in{\mathcal I}$,
   so   $b_1 + b_2 \preceq
        b_1'
   + b_2'$, since $\mathcal I$ is a module.
   Moreover, for $a\in \tT$, we have
   $ab_1 + ac = ab'_1$, with $ac\in {\mathcal I}$ since
   $\mathcal I$ is a $\tT$-gen bimodule.
This implies  that $ \preceq_{\mathcal I}$ is a $\tT$-pre-order.

Condition \eqref{surp-1} of \Cref{precedeq07} is given.

 \eqref{heigh1}:
 The verification of  \eqref{precedeq07-2}  of  \Cref{precedeq07}  is in  \cite[Proposition~2.31]{Row21}. When $\tTz $   is uniquely quasi-negated over $\mathcal I$, if
  $a \preceq_{ \mathcal
    I} a' $ for $a,a' \in
\tTz,$ then writing $a' = a+ c$ with
 $c \in \mathcal
I,$ we have $a'(-)a \in \mathcal I$, so $a' = a$.
This shows that  $ \preceq_{\mathcal I}$    is a surpassing relation.

Moreover, $\tTz $ is uniquely quasi-negated over $\mathcal I$,
if and only if for all $a,a'\in \tTz$, $a(-)a' \in \mathcal I$ implies that
$a=a'$. Taking   $a,a'\in \tT$, we get that
 $\tT$ is uniquely quasi-negated over $\mathcal I$, and taking
$a\in\tT$ and $a'=\zero$, we get that $ \mathcal  I \cap \tT= \emptyset$.
Conversely if   \eqref{newsys-1b} holds, then  \eqref{newsys-1a} holds, seen
by checking the four cases: $a\in \tT$ or $a=\zero$ and
$a'\in \tT$ or $a'=\zero$ and using that $\mathcal I$ and $\tT$ are
closed by $(-)$.

    \eqref{newsys-4} ${\mathcal A}_{\Null} = \{ b: \zero \preceq_{\mathcal I} b\} = {\mathcal I}.$
\end{proof}
 \begin{rem}
     In
the other direction, given a  system $(\mathcal A, \tT, (-),
\preceq)$, the
 surpassing relation $\preceq_\Null$    is given by
 $ \preceq_{\mathcal I}$ for $\mathcal I = {\mathcal A}_{\Null}.$
It is stronger than the given surpassing relation and weaker than
the $\circ$-surpassing relation; i.e., $ b \preceq_\circ b' \Rightarrow
b\preceq_\Null b'\Rightarrow b\preceq b'$.
 \end{rem}

There could be other such systems, obtained by extending $\preceq_
{\mathcal I}$ to include comparisons other than with $\zero.$
Nevertheless, $\preceq_ {\mathcal I}$ will be one of the two main
surpassing relations that we consider, together with set inclusion in
\Cref{hypersys}. It respects products in the following sense:

\begin{lem}\label{biphom1} If  $(\mathcal A, \tT, (-),\preceq_ {\mathcal I})$ is a semiring system for a $\tT$-sub-bimodule $\mathcal I$ of $\mathcal A$, and
$a\preceq_{\mathcal I}c$ and $a'\preceq_ {\mathcal I}
 c'$, then $aa' \preceq_{\mathcal I} cc'.$
\end{lem}
\begin{proof} Write $a + b = c$ and $a' + b' = c'$ for $b,b' \in \mathcal I.$ Then $aa'+ ( ba' +ab' + bb') = (a + b )(a' + b' ) = cc'.$
\end{proof}


\section{The doubling functor}\label{do}$ $

The goal of this section is to create a negation map (of second kind), reminiscent of constructing $\Z$ from $\Net.$

\begin{example}\label{semidir38} An example of
considerable interest, based on \cite{Ga,AGG2})  and to be used extensively in
\cite{AGR1}, is  \textbf{symmetrization}.

\begin{enumerate}
    \item Define $\widehat{\mathcal A} := \mathcal A
\times \mathcal A$ with the pointwise addition. We think of the first component as a positive copy of $\mcA,$ and the second component as a negative copy of $\mcA.$

\item Define $ \tT_{\widehat{\mathcal A}} = (\tT \times \{ \zero \})
\cup (\{\zero \}\times \tT)
  ,$ and $\widehat{\mathcal A}_\zero = \{(b_0,b_1): b_0+b_1 \in \mcA_0\}.$

 \item The \textbf{twist action}
on $\widehat{\mathcal A}$ over $ \tT_{\widehat{\mathcal A}}$
 is defined as follows:
 \begin{equation}\label{twi} (a_0,a_1)\ctw (b_0,b_1) =
 (a_0b_0 + a_1 b_1, a_0 b_1 + a_1 b_0),\ a_i \in \mathcal T, b_i \in \mathcal A.  \end{equation}

 \item The \textbf{twist product}
on $\widehat{\mathcal A}$ is defined also as $\eqref{twi}$, but now for $a_i,b_i\in \mcA.$


\item $\widehat { \mathcal A }$   has the ``switch''  map given
 by $(-)(b_0,b_1)  = (b_1,b_0)$.
Thus, on $ \tT_{\widehat{\mathcal A}}$, we have $(-)(a,0) = (0,a)$, and we have
 unique negation  of
the second kind. The quasi-zeros have the form
 $(a,a)$ with $a\in {\mathcal A}$

 If $f,g: (\mcA,\tT)\to (\mcA',\tT')$ are maps, then  define $(f,g):  \widehat{\mathcal A}\to   \widehat{\mathcal A}'$ given by $(f,g)(b_0,b_1) = (f(b_0)+g(b_1),f(b_1)+g(b_0)).$
\end{enumerate}
\end{example}

\begin{theorem}\label{func1} $ $\begin{enumerate}
    \item If $ \mcA$ is a \distributed   $\tT$-gen bimodule of height $t$, then $(  {\widehat{\mathcal A}}, \tT_{\widehat{\mathcal A}},(-))$ is a triple, of height $2t$.

\item If $f,g: (\mcA,\tT)\to (\mcA',\tT')$ are morphisms in $\MMod$, then there is a  morphism  $(f,g)\in \MModT$ given by $(f,g)(b_0,b_1) = (f(b_0)+g(b_1),f(b_1)+g(b_0)).$
(Likewise for  $\MB$ and $\MSr.$)

\item If $f$ and $g$ are weak morphisms, then so is $(f,g).$

    \item Any surpassing relation $\preceq$ on $\mcA$ induces a surpassing relation on $ \widehat{\mathcal A}$, given by $(b_0,b_1) \preceq (b_0',b_1')$ if $b_0 +b_1' \preceq b_0'+b_1.$   If $f$ and $g$ are $\preceq$-morphisms, then so is $(f,g).$

    \item
  There is a {\it doubling functor} from the category $\MMod$ to    $\MModT$, given by $(\mathcal A,\tT) \to ({\widehat{\mathcal A}}, \tT_{\widehat{\mathcal A}},(-))$ where $(-)$ is the switch  map, and sending
a morphism $f \to (f,f).$ This functor displays $\MModT$ as a reflective subcategory of $\MMod$. $(\mathcal A,\tT)$ is a retract of $(  {\widehat{\mathcal A}}, \tT_{\widehat{\mathcal A}},(-))$, seen via
the projection $ (  {\widehat{\mathcal A}}, \tT_{\widehat{\mathcal A}},(-))\to (\mathcal A,\tT) $ onto the first coordinate, which is a dominant functor.

Likewise $\WMModT$ is a  reflective subcategory of $\WMod$, and $(\MModT,\preceq)$ is a  reflective subcategory of $(\MMod,\preceq)$.
\end{enumerate}
\end{theorem}
\begin{proof}
(i),(ii) First note that $\tT_{\widehat{\mathcal A}}$ is a monoid, with unit element $(\one,\zero).$
 ${\widehat{\mathcal A}}$  is indeed a \distributed $\tT_{\widehat{\mathcal A}}$-bimodule, by a straightforward calculation:
 $$\begin{aligned} \sum _i (a_1,a_2)(b_{0,i},b_{1,i}) & = \sum (a_1 b_{0,i}  +a_2b_{1,i}, a_1 b_{1,i} + a_2 b_{0,i}) \\ & = (a_1 \sum b_{0,i} + a_2 \sum b_{1,i},\, a_1 \sum b_{1,i} + a_2 \sum b_{0,i}) \\ & = (a_1 ,a_2) \sum (b_{0,i},b_{1,i}).
 \end{aligned}$$ Obviously if $b_j$ has height $t_i$ then $(b_0,b_1)$ has height $t_1+t_2.$
 The functorial properties are easily verified.

(iii) Suppose $\sum _i (b_{0,i},b_{1,i} )\in {\widehat{\mathcal A_0}}$. Then  $\sum _i b_{0,i}+\sum _i b_{1,i} \in \mathcal A_0,$ implying $$(f,g)(\sum  (b_{0,i},b_{1,i} ) = f(\sum _i (b_{0,i})+g(\sum _i b_{1,i}, f(\sum _i b_{1,i})+g (\sum _i (b_{0,i}) \in {\mathcal A_0}\times {\mathcal A_0} ,$$
since $f(\sum _i (b_{0,i})+  f(\sum _i b_{1,i})\in {\mathcal A_0}$ and  $g(\sum _i (b_{0,i})+  g(\sum _i b_{1,i})\in {\mathcal A_0}$.

 (iv) We check the axioms of surpassing relation by components, noting that $(b_0,b_1)(-)(b_0,b_1) = (b_0+b_1,b_0+b_1) \succeq (\zero,\zero).$

 Now  \begin{equation}
     \begin{aligned}
         (f,g)(\sum  (b_{0,i},b_{1,i} ) & = \left(f(\sum _i b_{0,i})+g(\sum _i b_{1,i}), f(\sum _i b_{1,i})+g (\sum _i (b_{0,i})\right)\\& \preceq \left(\sum _i  f (b_{0,i})+\sum _i g(b_{1,i}), \sum _i  f( b_{1,i})+\sum _i g (b_{0,i}) \right) \\& = \sum (f,g)   (b_{0,i},b_{1,i} ) .
     \end{aligned}
 \end{equation}

 (v) Put together (i) thru (iv).
\end{proof}

 \begin{example}\label{netneg} $\widehat{\Net}$ is the symmetrization of $\Net$,
   according to \Cref{semidir38},
 where we write $m(-)n$ for $(m,n)$. Thus $(m_1(-)n_1)(m_2(-)n_2) = (m_1m_2 + n_1n_2)(-)(m_1n_2+n_1m_2).$

 $\widehat{\Net}$ acts on any $\tT$-bimodule $\mathcal A $ with negation via $(m(-)n)b = mb (-) n b.$
\end{example}

\begin{example}\label{semidir39} As defined in \Cref{semidir38}, the symmetrized triple
is not $(-)$-bipotent. To remedy this situation, as in \cite{Ga,Pl,AGG2},
 we can modify addition in the case that $\tT$ is a bipotent sub-semigroup
of $(\mathcal A,+)$,
by defining $$(a,\zero)+ (\zero,a')=
\begin{cases} (a,\zero) &\text{if } a+a' = a\neq a',\\
 (\zero,a') &\text{if } a+a' = a'\neq a, \\
 (a,a) &\text{if } a' = a ,\end{cases}
 \quad{and}\quad
(a,\zero)+ (a',a')=
\begin{cases} (a,\zero) &\text{if } a+a' = a\neq a',
\\  (a',a') &\text{if } a+a' = a',\end{cases} $$
for all $a,a'\in \tT$.
Now in \Cref{semidir38} take instead the $\tT_{\widehat{\mathcal A}}$-bimodule $\widehat{\mathcal A}' =\{(\zero,\zero)\}\cup \tT_{\widehat{\mathcal A}} \cup
\tT_{\widehat{\mathcal A}} ^\circ .$
Then, $(\widehat{\mathcal A}', \tT_{\widehat{\mathcal A}},(-))$ is a
$(-)$-bipotent triple  of second kind.
\end{example}

\section{Hypersystems}\label{hypersy}$ $

Our next main application   comes from hyperfields, as explained in
\cite[\S 3.4]{Row21}, which we now elaborate.
\subsection{Seeing hyperrings as systems}

 We denote by $\nsets (\Hy)$ the subset of non-empty subsets of a set $\Hy$.
  \begin{definition}\label{hyp} A monoid $(\Hy,\boxplus,\hyzero)$ is a \textbf{hyperring}
 when
\begin{enumerate}
   \item \label{hyp1}
$\boxplus$ is a commutative binary operation $\Hy \times \Hy \to
\nsets (\Hy)$ (the set of non-empty subsets of $\Hy$),
which also is associative in the sense that if we
 define
\[ a \boxplus S = S\boxplus a =\bigcup_{s \in S} \ a \boxplus s,
\]
 then $(a_1
\boxplus a_2) \boxplus a_3 = a_1 \boxplus (a_2\boxplus a_3)$ for all
$a_i$ in $\Hy.$
  \item \label{hyp2}
    $\hyzero$, called \textbf{the zero}, is the neutral element:
    $\hyzero \boxplus a =\{a\}$.
   \item \label{hyp3} every element $ a \in \Hy$ has a unique \textbf{hypernegative} $-a \in \Hy$, in the sense that $\hyzero \in a \boxplus
   (-a).$
\item  \label{hyp4} (reversibility) $\quad  a_1 \in a_2 \boxplus a_3 \quad \text{iff} \quad a_3 \in
a_1  \boxplus (-a_2).$
\end{enumerate}
We write $ S_1 \boxplus S_2 := \cup\{ s_1 \boxplus s_2: s_i \in
S_i\}.$

\end{definition}

Note that a negation map on $\Hy$ induces a negation map element wise on
$\nsets (\Hy)$, i.e., $$(-)S = \{ (-)\ell:\ell \in S \}.$$

 Connes-Consani~\cite[\S 2]{CC} and  Jun~\cite[Definition 2.1]{Ju} call this structure  a \textbf{canonical
 hypergroup}; Marshall~\cite{M} and Viro~\cite[Definition 3.1]{Vi} defined an equivalent notion,
 which they called a
 \textbf{multigroup}. A \textbf{multiring}~\cite[Definition 4.1]{Vi}
is
a hypergroup with associative multiplication such that $(\Hy,\cdot,\one)$ is a
monoid and $0$ is absorbing ($\hyzero a =a \hyzero=\hyzero$ for all $a\in
\Hy$), extended to all subsets of $\Hy$ by defining
   \begin{equation}\label{hypmul} S_1  S_2 = \{ s_1\cdot s_2: s_i \in S_i\},\end{equation}
satisfying
 \begin{equation}\label{mul} \{a \} ( \boxplus_i s_i) \subseteq  \boxplus_i (a \cdot s_i),\quad a, s_i\in \Hy.\end{equation}
In particular, if $S_1  =  \{ s_1\} $ and $S_2 =  \{ s_2\} $ are
singletons, then $S_1 S_2 = \{s_1\cdot s_2\}$ is a  singleton.
Then the elements of $\Hy$ can be identified with the singletons.
With this
multiplication, $\nsets (\Hy)$ satisfies the analogous axioms of a
semigroup, with $\boxplus$ instead of~$+$; in particular the zero
is absorbing ($\hyzero S =S \hyzero=\hyzero$, for $S\in  \nsets (\Hy)$).

A \textbf{hyperring}
is a multiring that is ``singly-distributive'', meaning that
multiplication by an element of $\Hy$ is distributive over addition
on~$\Hy$, i.e., equality holds in \eqref{mul}. In that case, $\nsets (\Hy)$ is an $\Hy$-bimodule, for the action
$a S=\{a\} S=\{as: a\in S\}$. We denote by $\overline{\Hy}$ the
additive submonoid
 of ${\nsets (\Hy)}$ generated by $\Hy$ (viewed as singletons).

A \textbf{hyperfield} (Connes-Consani~\cite[Definition~2.1]{CC}) is
a hyperring with
 $(\tT,\cdot,\one)$  an abelian group, where $\tT=\Hy\setminus\{\hyzero\}$.

As noted by Henry \cite{Henry}, the reversibility condition often is
superfluous. Let us make this precise.

\begin{lem}\label{reversibility} Assume that $\Hy$ satisfies conditions
(\ref{hyp1}--\ref{hyp3}) of \Cref{hyp}. Then:
\begin{enumerate}
\item \label{reversibility-i} $-\hyzero =  \hyzero.$
  \item   \label{reversibility-ii} $ a_1 \in a_2 \boxplus (-a_3) \quad \text{iff} \quad a_3 \in a_2
\boxplus (-a_1).$
 \item \label{reversibility2} If $\Hy$ satisfies the conditions of a hyperring,
except perhaps reversibility (\eqref{hyp4} of \Cref{hyp}), then
 $-a_1 = (-\one)a_1,$ for all $a_1\in \Hy$, and
$(-\one) S= - S:=\{-a: a \in S\}$, for all $S\in \nsets (\Hy)$,
and the map  $S\mapsto - S$ is a
 negation map on~$\overline{\Hy}$ or $\nsets (\Hy)$
(for the multiplicative action of $\Hy$). 
\end{enumerate}
\end{lem}
\begin{proof} \eqref{reversibility-i} $\hyzero\boxplus \hyzero =\{ \hyzero\},$ implying $-\hyzero =  \hyzero,$ by \eqref{hyp2} of \Cref{hyp}.

\eqref{reversibility-ii}  Suppose  $ a_1 \in a_2 \boxplus (-a_3) $. Then   $$\hyzero \in
(-a_1)\boxplus  a_1 \subset (-a_1)\boxplus  (a_2 \boxplus (-a_3)) = ((-a_1)
\boxplus a_2) \boxplus (-a_3) = \cup_{s \in (-a_1)\boxplus a_2} ( s
\boxplus (-a_3)).$$ But $\hyzero \in s \boxplus (-a_3)$ iff  $s=a_3$,
by unique hypernegatives, hence $a_3 \in (-a_1)\boxplus a_2 = a_2\boxplus (-a_1).$
Exchanging $a_1$ and $a_3$, we deduce the equivalence
$ a_1 \in a_2 \boxplus (-a_3) $ if and only if $a_3 \in a_2\boxplus (-a_1)$.

\eqref{reversibility2} From single distributivity,  we have
$a_1\boxplus  (-\one)a_1 = (\one\boxplus (-\one))a_1,$ which contains $\hyzero a_1 = \hyzero,$ since $0$ is absorbing.
So $ (-\one)a_1 = -a_1$ by unique hypernegation.
This implies that  $(-\one) S=\{-a: a \in S\}=-S$ for all $S\in {\nsets }(\Hy)$.
Using single distributivity again, we have
$$(-\one)(a_1\boxplus a_2) = ((-\one)a_1)\boxplus ((-\one)a_2) . $$
Hence, for all $S_1, S_2\in {\nsets }(\Hy)$,
we have $ (-\one ) (S_1\boxplus S_2)=\{(-\one) a: a\in a_1\boxplus a_2,\;
a_1\in S_1,\; a_2 \in S_2\}= \{a: a\in ((-\one)a_1)\boxplus ((-\one)a_2) ,\;
a_1\in S_1,\; a_2 \in S_2\}=
((-\one ) S_1)\boxplus ((-\one ) S_2)$.
This implies that  $S\mapsto (-\one ) S$ defines an additive isomorphism on
 ${\nsets }(\Hy)$ or $\overline{\Hy}$. Since it is compatible with
the multiplicative action of~$\Hy$, we also get a negation map.
\end{proof}

\begin{proposition}[\protect{cf.~\cite{M}, \cite[Theorem~3A]{Vi}}]\label{rev1}  Assume that $\Hy$ satisfies conditions
(\ref{hyp1}--\ref{hyp3}) of \Cref{hyp}.
The reversibility condition \eqref{hyp4} of \Cref{hyp}
is equivalent to  the map  $S\mapsto - S$ being an additive isomorphism on
 ${\nsets }(\Hy)$ or $\overline{\Hy}$.
\end{proposition}
\begin{proof} $(\Rightarrow)$ Assume reversibility.
We make three respective exchanges
 $$a_3\in (a_1 \boxplus a_2) \Equiv a_1\in a_3 \boxplus (-a_2)
\Equiv -a_2\in a_1 \boxplus (-a_3) \Equiv (-a_3)\in (-a_1) \boxplus
(-a_2) .$$ (We exchanged $a_1$ and $a_3,$ \ $a_1$ and $ -a_2$, \
and $-a_3$ and $-a_2$, respectively.)
Hence, for all $S_1, S_2\in {\nsets }(\Hy)$,
we have
$ - (S_1\boxplus S_2)=\{- a: a \in S_1\boxplus S_2\;
 \}= \{a: a\in (-a_1)\boxplus (-a_2) ,\;
a_1\in S_1,\; a_2 \in S_2\}=
(- S_1)\boxplus (-S_2)$.
This implies that  $S\mapsto - S$ is an additive isomorphism on
 ${\nsets }(\Hy)$ or $\overline{\Hy}$.

$(\Leftarrow)$
Assume that $S\mapsto - S$ is an additive isomorphism on
 ${\nsets }(\Hy)$ or $\overline{\Hy}$. Taking $-a_3$ instead of $a_3$ in
\Cref{reversibility2} of \Cref{reversibility}, we get $ a_1 \in a_2 \boxplus a_3$ iff $-a_3 \in a_2 \boxplus (-a_1)$.
Negating the last expression, 
this is equivalent to $a_3 \in -( a_2 \boxplus (-a_1))=
(-a_2)\boxplus a_1= a_1\boxplus (-a_2)$.
\end{proof}

The latter result shows in particular that the negation of a hyperring is a negation map iff the reversibility condition holds, and \Cref{reversibility2} of \Cref{reversibility} shows that both conditions follow from the other properties of a hyperring.
\begin{rem}
    Any multiplicative map $f:\mfrak{Hp}\to \mfrak{Hp}'$
    extends  to a bibimodule multiplicative map $ \nsets (\Hy)\to \nsets (\Hy')$, given by $f(S) = \{ f(s): s\in S\},$
    which restricts
    to a  bimodule multiplicative map    $f:\overline{\mfrak{Hp}}\to \overline{\mfrak{Hp}'}$.
    We make this identification without further mention.
\end{rem}

\begin{definition} We have two versions of hyperring categories:
\begin{enumerate}
\item The objects of the category $\mfrak{Hp}$ are hyperrings, and the morphisms are module multiplicative maps $f: \Hy \to \Hy'$ which satisfy
  $f(  \boxplus a_i) \subseteq \boxplus' f(a_i)$.

\item The category $\mfrak{wHp}$ has the same objects, hyperrings, but now with ``weak morphisms'' being bimodule multiplicative maps
$f: \Hy \to \Hy'$, satisfying $\zero\in  \boxplus _i a_i$ implies $\zero\in  \boxplus' _i f(a_i)$.

\end{enumerate}
\end{definition}

Hyperrings or hyperfields provide a special case
of systems via the following more precise statement of
\cite[Theorem~3.21]{Row21}:

 \begin{theorem}[\protect{\cite[Theorem~3.21]{Row21}}]\label{hypersys}$ $ Let   $(\MHModT,\preceq)$ denote the category of nd-semiring
 systems of hyper-type, whose morphisms are $\preceq$-morphisms,
 $(\MHSrT,\preceq)$   denote the full subcategory of $(\MHModT,\preceq)$ of semiring
 systems of hyper-type, and  $(\FHSrT,\preceq)$  denote the full subcategory of $(\MHModT,\preceq)$ for which $\tT =\mcA^\times.$
\begin{enumerate}
 \item \label{hypersys-i}   To every hyperring $\Hy$ is associated a system
$ ( \overline{\Hy} , \tT, (-),\preceq)$,
 where $\overline{\Hy}$ is the sub-monoid
of $ \nsets (\Hy)$ generated by $\Hy$
   for the additive law $\boxplus$, $\tT = \Hy \setminus \{ \hyzero \},$ identified as
 singletons,    $\zero$ is $\{\hyzero\}$,
   $(-)$ is the hypernegative $-$,
 $\tT$ acts on 
$\overline{\Hy}$
via $a S = \{ as: s \in S\}$, and  $\preceq$ is the set inclusion.
Moreover,  we have a faithful functor
   $\Psi:\mfrak{Hp}\to (\MHModT,\preceq_{\operatorname{hyp}})$ sending the hyperring
 $\Hy $ to $( \overline{\Hy} , \tT, (-), \subseteq)$, and hyperring morphisms to $\preceq$-morphisms.

 \item \label{hypersys-ii} $ \nsets(\Hy) $ satisfies the following partial version of double distributivity:
$$\left({\sum}_i^{\boxplus}a_i \right)\left({\sum}_j^{\boxplus} a_j' \right)\subseteq {\sum}_{i,j}^{\boxplus} a_i a_j' \quad \forall a_i, a_j'\in \Hy .$$

 \item \label{hypersys-iii} When the hyperring
    $\Hy$ is doubly distributive,
 $ ( \overline{\Hy} , \tT, (-), \preceq)$ is a semiring system.
In  this case the multiplication in
$\overline{\Hy} $ defined in \eqref{mult-doub}
coincides
    with
    \eqref{hypmul}, so the full subcategory of $\mfrak{Hp}$ whose objects are doubly distributive hyperrings is identified with a full  subcategory of  $(\MHSrT,\preceq_{\operatorname{hyp}})$.
      \item    The full subcategory of $\mfrak{Hp}$ whose objects are doubly distributive hyperfields is identified with a full  subcategory of~$(\FHSrT,\preceq_{\operatorname{hyp}})$.
     \item \label{hypersys-iv} Even when a hyperring $\Hy$ is not doubly distributive,
     it also gives rise to a semiring structure on    $ \nsets{\Hy}, $ by using the multiplication
\begin{equation}\label{mult-doub}
 \left( {\sum}_i^{\boxplus} a_i \right) \odot\left({\sum}_j^{\boxplus} a'_j\right) := {\sum}_{i,j}^{\boxplus} a_i a'_j\enspace ,
\end{equation}  as in
\cite[Theorem~2.45]{Row21} (also see \Cref{dc} ).
     The process also can be viewed as a faithful functor from  $\mfrak{Hp}$ to  $(\MHSrT,\preceq_{\operatorname{hyp}})$ \footnote{There is a subtle difficulty here: Given a hyperfield $\Hy$, $\overline{\Hy}$ need not be closed under the original multiplication, whereas the nd-semiring generated inside $ \nsets(\Hy)$ need not be $\Hy$-gen,
     so strictly speaking may not be a system. Thus we extend the multiplication on $\Hy$ to a different  multiplication on $\overline{\Hy}$}.
\end{enumerate}
 \end{theorem}
\begin{proof}
As in \cite[Theorem~3.21]{Row21}.
(i) We have to check that morphisms are preserved. Given a morphism $f: \Hy \to \Hy',$ we   extend $f$ to
$\nsets (\tT)$ by putting $f(S) = \cup\{f(a) :a\in S\}.$ This is a $\preceq$-morphism since $$f(S_1 \boxplus S_2) = \cup\{f(a_1 \boxplus a_2): a_i\in S_i\} \preceq \cup \{f(a_1)+ f(a_2): a_i\in S_i\} \subseteq f(S_1)+f(S_2).  $$

If $\zero \in S_1 (-) S_2 $, then there are $a_i\in S_i$ such that $\zero \in a_1 (-) a_2,$
so $\Phi (\Hy)$ is a system of hyper-type.

(ii)
Denote $S= {\sum}_j^{\boxplus}a_j'$.
Using  associativity, single distributivity, and then definition of sums of sets,  we have
\[ {\sum}_{i,j}^{\boxplus} a_ia_j'={\sum}_i^{\boxplus}({\sum}_{j}^{\boxplus} a_ia_j')={\sum}_i^{\boxplus}(a_i S)=\bigcup_{b_i\in S}({\sum}_i^{\boxplus}a_i b_i)\enspace . \]
Now using single distributivity again, we have that
\[({\sum}_i^{\boxplus} a_i )S=\bigcup_{b\in S}({\sum}_i^{\boxplus} a_i )b=\bigcup_{b\in S}({\sum}_i^{\boxplus} a_i b)
\subseteq \bigcup_{b_i\in S}({\sum}_i^{\boxplus}a_i b_i)\enspace .\]

(iii) By (ii), noting that all invertible elements of $ ( \overline{\Hy} , \tT, (-), \preceq)$ lie in $\Hy.$

(iv) First we identify $\tTz$ with its image under $\Phi $, since $ \{ a\in\tTz: a\preceq a'\}= \{a'\}$ when $a' \in \tTz.$
If $S_i = \{ a\in\tTz: a\preceq b_i\}$ then $a'S = \{ a' a: a\preceq b\}$, so when $f$ is a $\preceq$-morphism, $$\bar f(S_1+S_2) =  \{ f(a): a \in S_1+S_2\} =\{ f(a_1+a_2): a_i \in S_i\}   \subseteq  \{ f(a_1)+f(a_2): a_i \in S_i\} =\bar f(S_1) +\bar f(S_2) ,$$ proving $\Psi$ sends $\preceq$-morphisms to hyperring morphisms.

Clearly $(-)$ is preserved, since $(-) \{ a\in\tTz: a\preceq b\}=  \{(-) a\in\tTz: a\preceq b\} =  \{ a\in\tTz: a\preceq (-)b\}.$

If  $b_1 \preceq b_2$ then  $\{ a\in\tTz: a\preceq b_1\} \subseteq \{ a\in\tTz: a\preceq b_2\} ,$
by transitivity.
Hence $\Psi$ is a functor.

 $\Phi \Psi f$ sends
$a$ to $\bar f(\{a\}) = f(a).$

On the other hand, any hyperring morphism of the hyperring $\Hy$ corresponds to a $\preceq$-morphism of~$\Phi(\Hy).$

(v) Any hyperring morphism $f$ satisfies $$\begin{aligned} f\left(\left( {\sum}_i^{\boxplus} a_i \right) \odot\left({\sum}_j^{\boxplus} a'_j\right)\right)& = f\left({\sum}_{i,j}^{\boxplus} a_i a'_j\right) \\& \subseteq {\sum}_{i,j}^{\boxplus} f(a_i a'_j) =
 {\sum}_{i,j}^{\boxplus} f(a_i)f( a'_j) \\ &
=f\left( {\sum}_i^{\boxplus} a_i \right)f\left({\sum}_j^{\boxplus} a'_j\right)   \end{aligned},$$ so $f$ extends to a monoid morphism on $\mcA.$
\end{proof}

 \begin{definition}
 We call the system of \Cref{hypersys}~\eqref{hypersys-i}
the \textbf{hypersystem}  of the hyperring
$\Hy$.
\end{definition}

  Most researchers (as in \cite{BB2}) stick with
\eqref{hypmul}, requiring single distributivity but not double
distributivity.


\begin{note}\label{ddnote} $ $\begin{enumerate}
 \item Hypersystems are of hyper-type. Indeed, if $\sum S_i S_i' \succeq \zero$
 then $\zero \in \sum S_i S_i', $ so $\zero \in \sum a_i a_i' $
for suitable $a_i, a_i'\in S_i.$

\item Hypersystems are strong, since a singleton cannot contain a non-singleton.
\item Hypersystems also satisfy a stronger version of unique negation:
  If $b_1 \boxplus b_2 \in \mathcal A_{\Null}$, with
$\mathcal A=\overline{\Hy}$, then        $\hyzero \in b_1 \boxplus
b_2 ;$ consequently, for $b_1 \in \tT$ we get $(-)b_1 = b_2.$

\end{enumerate}
\end{note}

As observed by Krasner~\cite{krasner},
there is a natural way to construct some
important examples of hyperrings, see also
\cite[Proposition~2.6]{CC}, where applications of this construction
were given.  A similar construction
appeared in~\cite[\S 4, Examples~(ii)]{DrW}, in the setting
of fuzzy rings.
\begin{proposition}[{\cite{krasner}}]\label{prop-CC}
 Suppose $R$ is an associative ring
 having a normal multiplicative subgroup~$G$ (by which we mean $aG = Ga$ for all $a\in R$).
 Then the set of cosets $\mathcal H:=R/G=\{ [a] = aG: a\in R\}$,
 equipped  with the multivalued addition
$$[ a] \boxplus [{a'}] = \{ [{x+x'}] : x\in aG,\ x' \in
 a'G\},$$
 and the multiplication inherited from $R$,
 is a hyperring. In particular, $\mathcal H$
 is a hyperfield if $R$ is a field.\end{proposition}
\begin{remark}
The hyperrings and hyperfields $R/G$ arising in this way may or may not be
doubly distributive, seen as  follows.
The failure of the double distributivity
  property can be seen by considering the set
$$\begin{aligned}([a_0]\boxplus [a_1]\,)&([a_0']\boxplus [a_1']) =\\
& = \{[(a_0 g_1+a_1
g_2)(a_0'g_3+a_1'g_4)] : \quad g_i \in G\}\\ &  = \{[\, a_0a_0'
g_1g_3+a_1a_0'g_2g_3+a_0a_1'g_1g_4+a_1a_1'g_2g_4\, ] : \quad  g_i \in G\}
  \end{aligned}$$
  and comparing it with the one $\{[\, a_0a_0'g'_1+a_1a_0'g'_2+a_0a_1'g'_3+a_1a_1'g'_4\, ] : \quad g'_i\in G\}$  obtained by applying
  the double distributivity axiom.
  The first set is smaller in general, as it is obtained by imposing the
  restriction $g'_1g'_4 = g'_2g'_3$.
 \end{remark}
  \begin{remark}
Generalizing \Cref{prop-CC}, given a multiplicative monoid congruence
$\equiv$ on $R$, we can view
 $R/\equiv$ as a  hyperring by putting $$  \boxplus \left[a_i\right] = \left\{ \left[\sum _i b_i\right] : b_i \equiv a_i \right\}.$$
(Then  $-[a] = [-1][a]=[-a]$ is the unique hypernegative of $[a]$.) This generalizes \Cref{prop-CC} since any normal subgroup $G$ of a   monoid defines a monoid congruence
 by $a \equiv b$ when $aG = bG.$
\end{remark}


 Even though the flavor of the hypersystem is quite different from $\preceq_\circ$-systems, it is not so easy
  to find an example of a hypersystem  not of $\circ$-type, i.e., with $
 \mathcal A_{\Null}\neq \mathcal A^\circ $.
 \Crefrange{ex-krasner}{ex-viroc} are all of $\circ$-type.
 We next show several instances not of $\circ$-type.

\begin{example}
Consider Viro's multigroup $M=\{0,1,2\}$ of  \cite[\S 3.5]{Vi},
in which $1\boxplus 1=2$, $1\boxplus 2=2\boxplus 1=\{0,1\}$ and
$2\boxplus 2=\{1,2\}$, and let $\mathcal A=\nsets (M)$.
This is an example in which $ \mathcal A_{\Null}\neq \mathcal A^\circ $, since
$ \{ 0,2 \}\in  \mathcal A_{\Null}$ but $ \{ 0,2 \}\notin \mathcal A^\circ$.
However, the map $S\mapsto -S$
 is not an additive isomorphism, since $-1 = 2$, implying $(- 1)\boxplus 1
 = \{ 0,1 \} \ne  \{ 0,2 \}= \{ 0, - 1 \}$,
 whereas $1\boxplus ( - 1) = \{ 0,1 \}$, so
$-(1\boxplus ( - 1))= \{ 0,-1 \}$.
Therefore, $M$  is not a hyperring (by \Cref{rev1}).
Note also, that the additive monoid generated by~$M$ does
not contain $ \{ 0,2 \}$, and is not preserved by the $S\mapsto -S$
map.
\end{example}
\begin{example}
In 
the triangle hyperfield $\Hy=[0,\infty)$  of \cite[\S
 5.1]{Vi}, $a\boxplus b=[|a-b|,a+b]$,
so  $\mathcal A=\overline{\Hy}$,
the monoid generated by $\Hy$ for the additive law $\boxplus$,
 is the set of closed intervals $[a_1,a_2]$,
the hypernegation is the identity map, and
$\mathcal A_{\Null}=\mathcal A^\circ$ since
 $[0,a] = (\frac a 2)^\circ.$ But we can modify the
 latter to obtain formally our desired example.

 Take Viro's triangle hyperfield $\Hy$  of  \cite[\S
 5.1]{Vi}, but with $\Hy \setminus \{ \hyzero \} = \Int^+$, the set of positive natural numbers,
 instead of $\mathbb R^+$.
This is a sub-hyperring of the previous one, and it gives rise
to a hypersystem on $\mathcal A=\overline{\Hy}$, which is a
subsystem of the previous one. The elements of $\mathcal A=\overline{\Hy}$
are the closed intervals $[a_1,a_2]$, with $a_1,a_2\in \Hy$ and either $a_1=0$ or $a_2-a_1$ even.
The elements of $\mathcal A^\circ$ are the interval $[0,2a]$ with $a\in  \Hy$,
since $a^\circ=[0,a]^\circ=[0,2a]$.
Then $1\boxplus 1 \boxplus  1 = [0,3]$ is in $\mathcal A_{\Null}$ but it is not in $\mathcal A^\circ$ since $3$ is odd.
\end{example} 

 \begin{example}
In the set-up of \Cref{prop-CC},   taking  $R = \mathbb{Z}$ and $G$ containing the element  $-1,$
  leads us to the interesting arithmetical question of comparing  $
 \mathcal A_{\Null} $ and $\mathcal A^\circ .$
Consider $\Hy=R/G$, and the associated hypersystem $(\mathcal A=\overline{\Hy}, \tT := \Hy \setminus \{0\}  , (-),\subseteq)$ as in \Cref{hypersys}.
Write $[b]$ for the class $bG$, for $b\in R$ (as in \Cref{prop-CC}).
We have $-[a] = [a]$ since $-1\in G$. Let $\widetilde G = \{1 +g: g\in G\}$. Then $\tT^\circ  = \{ \{ [a +ag]: g\in G \} : [a] \in \tT\} = \{ \{ [a\tilde g]: \tilde g\in \widetilde G\} : 
 [a] \in \tT\}.$ (For example, if $G = \{\pm 1\}$ then $\widetilde G =
 \{ 0,2\}$ and $\tT^\circ = \{ \{[0], [2a]\}: [a]\in \tT\}.$)

 Recall that $\mathcal A_{\Null}$   consists of subsets of $\tTz$ containing~$0$. Let $b\in \mathcal A$, and suppose that $0 \in b$. Then there are $a_i \in R$ with $b= \{[\sum_{i=1}^{t} a_ig_i]: g_i\ \in G\}$ and $0 = \sum_{i=1}^t  a_ig_i $ for suitable $g_i\in G,$ so we can replace $a_i$ by $a_ig_i$ and have
   $ a_1+\dots + a_{t-1}+ a_t = 0,$ and thus
 $a_t = -a_1-\dots - a_{t-1},$ and then $$b=  \left\{\left[\sum_{i=1}^t a_i g_i \right] : g_i \in G\right\}=  \left\{\left[\sum_{i=1}^{t-1} a_i(g_i-g_t) \right]  : g_i \in G \right\}.
 $$
Clearly $b \subseteq   [a_1]^\circ \boxplus \dots \boxplus  [a_{t-1}]^\circ$, since $a_i(g_i-g_t) = a_ig_i-a_ig_t \in [a_i]^\circ $. But the other direction fails,
as we show now.

Let
$$b = [1]+[3]+[4] = \{[0], [2], [6], [ 8]\},$$
which is in $\mathcal A_{\Null}$.
If $b = \sum^\boxplus [a_i]^\circ$ for $a_i \in \mathbb{N},$ then  $\sum_i a_i = 4$ because the maximal element contained in an equivalence class belonging to $b$ is $\sum_i 2a_i$. Moreover, every  sum $S$
of a nonempty subcollection of $a_j$ must belong to  $\{1,3,4\}$ because $[a_i]^\circ=\{[0],[2a_i]\}$
and so the class $[2S]$ must belong to $b$.
So the only candidates left to verify are
$$ [4]^\circ = \{[0],[8]\},\qquad  [1]^\circ +[3]^\circ = \{[0], [2], [4], [6],  [8]\}.$$
  Thus $b\notin A^\circ.$

  We can modify the previous example so that it holds in a hyperfield  by taking $ R=\mathbb{Z}_p$ (for $p>23$) instead of $ \mathbb{Z}$.
\end{example}


\subsection{Examples of hyperfields and their triples and hypersystems}

We next discuss basic examples of hyperfields taken from \cite{krasner,CC,Vi} and relate their
corresponding hypersystems to several extensions or variants of
the tropical semifield which have appeared in the literature.

\begin{example}\label{ex-krasner}
  The \textbf{Krasner hyperfield} $\mathcal K$ is the set $ \Hy := \{ 0, 1 \}$, with the usual multiplication
law, and hyperaddition defined by  $ x \boxplus  0 = 0 \boxplus  x =
x $ for all $x ,$ and $1 \boxplus  1 =  \{ 0, 1\}$
\cite{krasner}.
It coincides with the quotient hyperfield $K/K^\times$ of \Cref{prop-CC} for any
field $K$ of order $>2$.

The hypersystem of $\mathcal K$
  is isomorphic to the supertropical algebra
  $\{ \hyzero, \one, \one^\nu\},$
which is of the first kind, with $\tT =\{ \one\}$, seen by sending $
 0  \to \hyzero$,  $ 1 \to \one $, $ \{ 0, 1\} \to
\one^\nu$.
\end{example}
\begin{example}
  The \textbf{hyperfield of signs} $\mathcal S$
  is the set $ \Hy := \{ 1 , 0, -1\}$,
  with the usual multiplication
law, and hyperaddition defined by $1 \boxplus  1 = 1 ,$\ $-1
\boxplus  -1 = -1  ,$\ $ x \boxplus  0 = 0 \boxplus  x = x $ for all
$x ,$ and $1 \boxplus  -1 = -1 \boxplus  1 = \{ 0, 1,-1\} $
\cite{CC,Vi,BL2}.
It coincides with the quotient hyperfield $K/K_{>0}$ for every linearly
ordered field $(K,\leq )$,
where $K_{>0}$ is the group of positive elements of $K$.

The  triple of $(\mathcal S, \{(\pm)\one\},(-))$ $\mathcal S$ is isomorphic to the triple $\{L, \{\pm 1\}, -)$ of the sign semiring, defined as follows:
\begin{enumerate}
    \item $L = \{0,1,-1,\infty\}$, which we recall has the usual
multiplication law,  where $0\cdot\infty=1\cdot\infty =\infty,$
and with addition defined by $1 +1 = 1 ,$\ $-1 + (,-1) = -1  ,$\ $ x
+ 0 = 0+ x = x $ for all $x ,$   $1 + ( -1) = \infty ,$ and $\infty$ is
an additively absorbing element.

 \item The natural negation map is the natural negation  $(-)1 = -1$.

 \item $\{L, \{\pm 1\}, -)$   is of the second kind with $\tT = \{\pm
1\}$ and $L^\circ = \{0, \infty = 1 ^\circ\}$
\end{enumerate}
The    isomorphism from~$(\mathcal S, \{(\pm)\one\},(-))$ to $\{L, \{\pm 1\}, -)$,
  sends $ \{ 0, 1,-1\} $ to $ \infty$.

$(\mathcal S, \{(\pm)\one\},(-))$ is also
isomorphic to the symmetrization, defined as in
\Cref{semidir39}, of the Boolean
semifield $\{ \hyzero, \one\}$ (with $\one + \one = \one$), seen by
sending $1 \mapsto (\one, \hyzero)$, $-1\mapsto ( \hyzero, \one)$
and $ \{ 0, 1,-1\} $ to $(\one,\one)$ (see also~\cite{AGG2}).
\end{example}
\begin{example}\label{ex-phase}
  The \textbf{phase hyperfield}. Let $S^1$ denote the complex
  unit circle, and take $\Hy = S^1\cup \{ 0 \}$.
  Nonzero points $a$ and $b$ are \textbf{antipodes} if $a = -b.$
Multiplication is defined as usual (so corresponds on $S^1$ to
addition of angles). We denote  an open arc of less than 180 degrees
connecting two distinct points $a,b$ of the unit circle by
$\myarc{a\,b}$. The hypersum is given, for $a,b\neq 0$, by
$$a \boxplus b=
\begin{cases} \myarc{a\,b} \text{ if } a \ne b \text{ and } a\neq -b\enspace ;\\
  \{ -a,0,a \} \text{
if } a = -b \enspace , \\   \{  a \} \text{ if }   a = b\enspace .
\end{cases}$$
The phase hyperfield coincides with the quotient hyperfield
$\C/\R_{>0}$. It is not doubly distributive. The hypersystem
$\overline{\Hy}$ contains only the points of $\Hy$, the elements of
the form $a\boxplus b$ with $a,b\in S^1$,
 and subsets  $C$ of $\Hy$,
where either $C$ is an open half circle connecting an element $a\in
S^1$ to $-a$ (which is obtained as the sum $a\boxplus b\boxplus
(-a)$, where $b$ is inside the open arc  $C$), or $C=\Hy$ (which is obtained as $a\boxplus b\boxplus (-a)\boxplus (-b)$).
For example, $\{ \pm 1 \} \not\in \overline{\Hy}$.

The phase hyperfield is not doubly distributive. One can consider however
the hypersystem of $\Hy$, that is
$(\overline{\Hy},\tT=\Hy\setminus\{0\}, (-),\preceq)$, with the multiplication $\odot$ defined
in~\eqref{mult-doub}, which makes it a semiring system. This
semiring $\overline{\Hy}$ is isomorphic to the phase
semiring defined in~\cite[Ex.\ 2.22]{AGG2}. Indeed, this semiring consists in the set
$\mathcal C$ of closed convex cones of $\C$ (which are either the
  singleton $\{0\}$, the closed angular sectors
  between two half-lines with angle less or equal to $\pi$, the
  lines passing through $0$, or the set $\C$), in which
  the addition of two elements $C_1$ and $C_2$
  is the convex hull of $C_1\cup C_2$, and their multiplication
    is the convex   hull of the set of products of elements
  in $C_1$ and $C_2$.
  The isomorphism between $\Hy$ and $\C$ is obtained by the bijective map
  $i:\overline{\Hy}\to {\mathcal C},\;  p\mapsto c$,
  where, for any element $p$ of $\overline{\Hy}$, $i(p)$ is
  the closed convex cone generated by $p$, and any element $c$ of
  $\mathcal C$,  $i^{-1}(c)$ is the intersection
  of the relative interior of $c$ with $\Hy$.
\end{example}

   Recall that a \textbf{(non-archimedean) valuation} from a ring $\ring $   to  an
ordered monoid $( \hgroup,+,0)$ is   a monoid homomorphism $\val:
(R,\cdot,1) \to \hgroup\cup\{\top\}$ (where $\top$ is a top element  not belonging to $\hgroup$),   satisfying
$\val(0)=\top$, and $$\val(a+b) \ge \min \{ \val(a),
\val(b)\}, \qquad \forall a,b \in R.$$

It is well known that $\val(\pm 1) = 0,$ and if $ \val(a)> \val(b)$   then
$\val(a+b) = \val(b)$.

  For instance, we may take $R= \C[[t^{\hgroup}]]$ the field of
  Hahn series~\cite{higman}, which is the set of formal series  with complex
  coefficients, and a well ordered
support of exponents in an ordered group $({\hgroup},+,0)$, i.e., every Hahn series can be written
  as a formal sum $f=\sum_{\lambda \in \Lambda} f_\lambda t^\lambda$
  where $\Lambda$ is a well ordered subset of $\hgroup$
  and $s_\lambda \in \C$.
The \textbf{valuation} of a series is defined by
\begin{equation}
  \operatorname{val}(f) = \min \{\lambda : f_\lambda \neq 0\}\label{e-def-val}
\end{equation}
with the convention that $\operatorname{val}(0) =+\infty$.
The operator
$\operatorname{val}$ is often referred to as \textbf{tropicalization}.

When $\hgroup$ is a subgroup of $\R$, one can also consider the
smaller field  $\C\{\{t^{\hgroup}\}\}$ of Hahn series such that
the  set of exponents is a subset of $\hgroup$ which is either
finite, or the set of elements of a sequence converging to
$+\infty$. When $\hgroup=\R$, this is the field of (generalized)
Puiseux series with complex coefficients, which is a popular choice
in tropical geometry.
  If  $\hgroup$ is a  divisible subgroup of $(\R,+)$, then $R$
  is algebraically closed. In particular,  one can choose  ${\hgroup}=(\Q, +)$.
  In this way we obtain the Levi-Civita field $\C\{\{t^\Q\}\}$, which is the Cauchy closure
   of the more familiar field of Puiseux
  series, which is the subfield of $\C\{\{t^\R\}\}$ consisting
  of series whose exponents are rational
  and included in an arithmetic progression.


\begin{example}\label{ex-tropicalhyperfield}
  The \textbf{tropical hyperfield}
  consists of the set $\Hy=\R \cup \{-\infty\}$, with $-\infty$
  as the zero element and $0$ as the unit, equipped
  with the addition $a\boxplus b = \{a\}$ if $a>b$,
  $a\boxplus b = \{b\}$ if $a<b$,
  and $a\boxplus a= [-\infty,a]$.
  It coincides with the quotient hyperfield $\kfield/\hgroup$,
  where $\kfield$ denotes a field with a surjective non-archimedean
  valuation $\val: \kfield \to \R\cup\{+\infty\}$, and
  $\hgroup:=\{f\in \kfield: \val f =0\}$, the equivalence
  class of any element $f$ having valuation $a$ being
  identified with the element $-a$ of $\Hy$.

The hypersystem of the tropical hyperfield
is isomorphic to the supertropical system where $\tT=\R$,
seen by sending
$a\mapsto a$, and $[-\infty,a]\mapsto a^\nu$, for $a\in \R \cup \{-\infty\}$.
It is also isomorphic to the layered extension system
  $L \rtimes \R$ of the ordered monoid $(\R,+,0,\leq)$ with the supertropical algebra
  (or the hypersystem of the Krasner
  hyperfield) without zero
  $L=\{1,1^\nu\}$,
using \Cref{AGGGexmod1} of layered systems below.
   \end{example}
\begin{example}\label{ex-signedtrop}
  The \textbf{signed tropical hyperfield} $\Hy$
  is the union of two disjoint copies of $\R$,
  the first one being identified with $\R$, and denoted by $\R$,
  the second one being denoted by $(-) \R$,
  with a zero element of $\Hy$ adjoined, denoted  $-\infty$.
  We denote by ``$a$'' with $a\in \R$ the elements of the first copy, and by
  ``$(-) a$'' with $a\in \R$ the elements of the second copy $(-) \R$.
We put $(-) \zero=\zero$.
  The addition is defined by
  \[
  a \boxplus b =\{\max(a,b)\},
  \qquad a \boxplus ((-) b) =
  \begin{cases}
    \{a\} & \text{ if }a>b\\
    \{\ominus b\} & \text{ if } a<b\\
    \{\zero\}\cup \{x: x\leq a\}\cup
    \{\ominus x: x\leq a\}
    & \text{ if } a=b \enspace ,
    \end{cases}
  \]
  and $\zero\oplus a=a\oplus \zero=a$, for $a,b\in \R$ and satisfies
  $(\ominus a) \boxplus (\ominus b)=\ominus (a\boxplus b)$.
  The multiplication is obtained by taking the usual addition on $\R$
  and using
  the rule
  of signs, i.e., $ab$ is the usual sum of $a$ and $b$, $(\ominus  a)b=a(\ominus b):=\ominus (ab)$, and $(\ominus a)(\ominus b)=ab$, for all $a,b\in \R\cup\{\zero\}$ and the absorbing property of the zero $\zero a=a\zero=\zero$.

  The signed tropical hyperfield is otherwise known as the ``real tropical numbers,'' in analogy to \Cref{ex-tropicalhyperfield},
  or the ``tropical real hyperfield'', see \cite{Vi}.
  This hyperfield coincides with the quotient hyperfield $\kfield/\hgroup$
  where $\kfield$ is a linearly ordered non-archimedean field with a surjective
  valuation $\val: \kfield \to \R\cup\{+\infty\}$, and
  $$G:=\{f\in \kfield: f>0, \; \val f =0\}.$$
  For instance, we may take for $\kfield$ the set of Hahn series
   $\R[[t^{\R}]]$
  or the set of generalized Puiseux series $\R\{\{t^\R\}\}$,
  with real exponents and real coefficients,
  with the lexicographic order on the sequence of coefficients
  of a series, where lexicographic order is with respect to the
  order of the exponents~\cite{ribenboim_ordered,MR1289092}.

  The hypersystem of the signed tropical hyperfield is isomorphic
  to the symmetrized max-plus semiring, see for instance~\cite{AGG2},
  and it is also isomorphic to the layered extension system
  $L \rtimes \R$ of the ordered monoid $(\R,+,0,\leq)$ with the sign semiring (or the hypersystem of the
  hyperfield of signs) without zero
  $L=\{1,-1,\infty\}$,
using \Cref{AGGGexmod1} of layered systems below.
\end{example}

\begin{example}\label{ex-phasealt}
  A variant of the phase hyperfield, which we call
  here the \textbf{weak phase hyperfield},
  can be obtained by taking
  the quotient $\kfield/\hgroup$, where $\kfield=\C\{\{t^{\R}\}\}$, and $\hgroup$
  is the group of (generalized) Puiseux series with  positive real
  leading coefficient, where the leading coefficient is the coefficient
  $f_\lambda$ of the series $f=\sum_{\lambda \in \Lambda} f_\lambda t^\lambda$
  such that $\lambda$ is the minimal element of
  $\{\lambda \in \Lambda : f_\lambda \neq 0\}$.
  The non-zero elements of $\kfield/\hgroup$ can still
  be represented by elements of the unit circle $S^1$. However,
  the hyper-addition differs, for $a,b\neq 0$, we now have
  $$a \boxplus b=
  \begin{cases} \myarc{[a\,b]}
    &
    \text{ if } a \ne b \text{ and } a\neq -b ,\\
  S^1 \cup \{0\} & \text{
    if } a = -b
  , \\   \{  a \} & \text{ if }   a = b .
\end{cases}$$
where $\myarc{[a\,b]}$ denotes the {\em closed} arc of the angle
inferior to $180$ degrees joining $a$ and $b$ (compare with the open
arc $\myarc{a\,b}$ in \Cref{ex-phase}).
As the phase hyperfield, the weak phase hyperfield is
not doubly distributive, but is $\subseteq$-distributive.
  \end{example}

\begin{example}\label{ex-tropicalcomplex}
  By analogy with the construction of the tropical hyperfield and
  of the signed tropical hyperfield, obtained as quotients
  of non-archimedean fields by groups of elements of valuation $0$,
  we take $L$ to be the \textbf{phased tropical hyperfield}, viewed
  as the quotient $\kfield/\hgroup$, where $\kfield=\C\{\{t^{\R}\}\}$ (as in \Cref{ex-phasealt}),
  and $\hgroup$ is the group of (generalized) Puiseux series with valuation
  $0$ and positive leading coefficient.  A non-zero
  element of $\kfield/\hgroup$ is the coset
  consisting of all series $f= rut^\alpha  + o(t^\alpha)$,
  with $r\in \R_{>0}$, where $u\in S^1$ and $\alpha\in \R$ are fixed,
    and $o(t^\alpha)$ denotes the set of series having valuation smaller than $\alpha$.     So, this element can be uniquely represented by the couple
    $(u,\alpha)\in S^1\times \R$. The multiplication of non-zero elements of
    $\kfield/hgroup$ corresponds to the multiplication on the $S^1$-component,
    and to the addition on the $\R$-component.
    Negation and hyper-addition of non-zero elements of
    $\kfield/hgroup$ is defined as follows: for $a=(u,\alpha)$ and $b=(v,\beta)
    \in S^1\times R$, $-a=(-u,\alpha)$ and
\[
a \boxplus b=
\begin{cases}
  \{  a \} & \text{ if }   \alpha  > \beta \text{ or } a =b\enspace ,\\
  \{  b \} & \text{ if }   \alpha  < \beta  ,\\
 \myarc{u\,v}\times \{\alpha\} &  \text{ if } \alpha= \beta \text{ and } u\neq -v\enspace ;\\
  \{ -a,a \} \cup \left(S^1\times (-\infty,\alpha)\right)\cup\{\hyzero\}& \text{ if } a = -b \enspace .
\end{cases}
\]

\end{example}
\begin{example}[Viro's complex hyperfield]\label{ex-viroc}
In~\cite{Vi}, Viro considered a variant, which he called the
\textbf{complex hyperfield}. The elements of this hyperfield are the
complex numbers, so that non-zero elements $a$ of this hyperfield
can still be represented as $(u,\alpha)\in S^1\times \R$, by taking
$u=a/|a|$ and $\alpha=\log(|a|)$. Using this representation,
negation and hyper-addition of non-zero elements is defined as
follows: for $a=(u,\alpha)$ and $b=(v,\beta)
    \in S^1\times R$, $-a=(-u,\alpha)$ and
\[
a \boxplus b=
\begin{cases}
  \{  a \} & \text{ if }   \alpha  > \beta \text{ or } a =b\enspace ,\\
  \myarc{[u\,v]}\times \{\alpha\} &  \text{ if } \alpha= \beta \text{ and } u\neq -v\enspace ;\\
 \left(S^1\times (-\infty,\alpha]\right)\cup\{\hyzero\}& \text{ if } a = -b \enspace .
\end{cases}
\]
Consider the field of
Hahn series with exponents in the group $\R\times \R$
with lexicographic order, or equivalently the field
of series of the form $f=\sum_{\lambda \in \Lambda} f_\lambda t^{\lambda_1}(\log\, t)^{\lambda_2}$,  where $\Lambda$ is a well ordered subset of $\R\times \R$ and
and $f_\lambda \in \C$. The valuation of a non-zero series is
an element of $\R\times \R$ and has thus two coordinates $\alpha_1,\alpha_2$.
Then,  the complex hyperfield of Viro
coincides with the quotient hyperfield $\kfield/\hgroup$,
where $G$ is the group of series
with  positive leading coefficient and first coordinate of its valuation
equal to $0$. Using the representation as series in $t$ and $\log\, t$,
a non-zero element of $\kfield/\hgroup$ is the coset
  consisting of all series $f= rut^{\alpha_1}(\log\, t)^{\alpha_2}  + o(t^{\alpha_1}(\log\, t)^{\alpha_2})$,
  with $r\in \R_{>0}$ and $\alpha_2\in \R$,
  where $u\in S^1$ and $\alpha_1\in \R$ are fixed,
  and $o(t^{\alpha_1}(\log\, t)^{\alpha_2})$ denotes the set of series of valuation smaller than $(\alpha_1,\alpha_2)$.
Again, this hyperfield is not doubly distributive.
\end{example}



\subsection{Obtaining hypersystems  through balancing}$ $

Conversely to \Cref{hypersys}, given a system
$(\mathcal A,  \tT, (-), \preceq)$, we would like $\mathcal A$ to
act like a power set for which we  identify $\tT$ as the hyperring
of its singletons. The difficulty is in determining the set defining
the sum of tangible elements.
\cite[Theorem~7.4]{Row21} provides a
partial converse to \Cref{hypersys} for metatangible systems;
we aim for a more general construction, based on ``balancing''.
Recall that for a triple $(\mathcal A,  \tT, (-))$, given a $\tT$-sub-bimodule $\mathcal{I}$ of $\mathcal{A}$,
containing $\mathcal A^\circ$, 
the balance relation was defined by $b_1\nabla_{\mathcal I} b_2$
if $b_1(-)b_2\in \mathcal I$, see \Cref{def-tripbal}.
Moreover, we write $\nabla$ instead of $\nabla_{\mcA_{\Null}}$.



\begin{definition}

Two elements $b_1,b_2$ of $(\mathcal A, \tT, (-))$ are \textbf{tangibly balanced} if there is $a\in \tTz$ such that $b_i \nabla a$ for $i=1,2.$

The triple $(\mathcal A, \tT, (-))$ is \textbf{tangibly balanced} if any balanced pair
of elements
is tangibly balanced.


Conversely,   $\mathcal A$ has \textbf{tangible balance elimination}
if tangibly balanced implies balanced, i.e., for any $a \in \tTz$,
$b_i \in \mathcal A$, $ b_1   \nabla  a$ and $a \nabla b_2$ implies
$b_1 \nabla b_2$.

$\mathcal A$ has \textbf{left $\nabla$-inversion} if
$ab \nabla a_1$
for $a,a_1\in \tTz$ and $b \in \mathcal A$ implies there is $b' \in \tTz$
such that $ab' = a_1$ and $b \nabla b'$. \textbf{Right $\nabla$-inversion} is defined analogously.
 \textbf{$\nabla$-inversion} means left and right $\nabla$-inversion.
\end{definition}

\begin{remark}\label{nabl} \  \begin{enumerate}
   \item Any two elements surpassing $-\infty$ are tangibly
 balanced, seen by taking $a = \zero.$

  \item Tangible balance elimination is a weak form of
transitivity of $\nabla.$
\end{enumerate}\end{remark}

\begin{lem}\label{flf}
For any two tangible elements $a_1,a_2$ of a metatangible triple, there is $a_3\in \tTz$ such that $a_1+a_2(-)a_3 \in {\mathcal A}_{\Null}.$
\end{lem}
\begin{proof}
If $a_1+a_2 \in {\mathcal A}_{\Null},$ take $a =\zero.$
If $a_1+a_2 \in \tT,$ take $a = (-)(a_1+a_2).$
\end{proof}

The condition of \cref{flf} is analogous to ``field-like'' for fuzzy rings,
given in \cite{GJL}.

\begin{lem}\label{ugly} Assume $\mathcal A$ is tangibly balanced. We have
 \begin{enumerate}
\item \label{ugly-i} If   $b\in \mathcal A$ is tangibly balanced with   $a\in\tTz$, then
$a\nabla b$.
  \item  \label{ugly-ii}
 If $a\in \tTz$ with   $a \nabla (b+c)$, for $b,c\in \mathcal A$, then there exists $a' \in \tTz$ balancing $b$, with $a \nabla (a'+c)$.
 \item  \label{ugly-iii} Left and right $\nabla$-inversion holds whenever $\tT$ is a group.
\end{enumerate}\end{lem}
\begin{proof}
\eqref{ugly-i} If $a\nabla a'$ for $a'\in \tTz$ then $a = a'$.

\eqref{ugly-ii}  If
    $a\nabla (b+c)$ then $(a(-)c) \nabla b$, implying $a'\nabla (a(-)c)$ as well as  $a' \nabla b$, for some $a' \in
    \tTz$, and we get $a \nabla (a'+c)$.

 \eqref{ugly-iii} Let $a,a_1\in \tTz$ and $b \in \mathcal A$ such that $ab \nabla a_1$.

If $a\neq \zero$, then $b' = a^{-1}a_1\in \tTz$ satisfies
$a_1 = ab'$ and $b  (-) b' =     a^{-1}(ab (-) a_1) \in {\mathcal A}_{\Null}.$

If $a=\zero$, then $a_1\in \tTz$ and $a_1\nabla \zero$,
implying $a_1=\zero$, so the assertion is vacuous.
\end{proof}

\begin{proposition}\label{tangbal} $\mathcal A$ is both tangibly balanced and satisfies tangible balance elimination under either of the
following conditions:
\begin{enumerate}
   \item \label{tangbal-i}  $\mathcal A$ is  shallow with unique negation
   (thus including the  supertropical situation and \Cref{semidir39}), and $\preceq$ coincides with
   $\preceq_\zero$.
 \item  \label{tangbal-ii} $\mathcal A$ is a hypersystem.
   \end{enumerate}
\end{proposition}
\begin{proof} \eqref{tangbal-i} In this case $ {\mathcal A}_{\Null} = \mathcal A^\circ.$ Suppose $b_1 (-) b_2 \in \mathcal A^\circ.$ The
requirement for tangibly balanced is clear if either $b_i \in \tTz$,
so we may assume that $b_1 , b_2 \in \mathcal A^\circ,$ in which
case we just take $a = \zero.$

For tangible balance elimination, suppose $b_i (-) a  \in {\mathcal A}_{\Null}$
for $i=1,2.$ If either $b_i \in \tT$, then $a = b_i$ by unique
negation, so again the assertion is immediate.

\eqref{tangbal-ii} Suppose $b_1\nabla b_2$,  meaning that $\zero \in b_1 (-) b_2 .$ By hypothesis there is
tangible $a \in b_i$ for $i=1,2,$ implying $a \nabla b_1$ and $a
\nabla b_1,$  proving $\mathcal A$ is tangibly balanced.

For tangible balance elimination, suppose $ b_i \nabla a  $ for
$i=1,2.$ Thus $a \in b_1$ and $-a \in - b_2$, implying $\zero
\preceq a (-)a \subseteq b_1 (-)b_2.$
\end{proof}

Thus a counterexample to tangibly balanced cannot be supertropical,
symmetrized, or a hypersystem. Nevertheless, one can obtain such
counterexamples when the conditions do not hold.

\begin{example}\label{someex} $ $\begin{enumerate}
  \item \label{someex-i} (Essentially \cite[Example~4.1]{GJL}) If $a\notin \mathcal A_\Null$   and does not balance any tangible element, then we
  have a counterexample, by \eqref{ugly-i} of \Cref{ugly}. For example, let $\mathcal A = \mathbb Z ,$
  viewed as a hyperfield under the usual addition, and usual negation $-$, and let $\tT = \{ \pm 1 \}$. Then $2 \nabla
  2,$ but $2$ does not balance $\pm 1,$
  for the inclusion surpassing relation. $(\mathcal A, \tT, -,\subset)$ is a semiring system but not a hyperfield  system.

\item  \label{someex-ii} One can get a counterexample to tangible balance elimination by starting with $ \Net  [\Lambda],$ the
polynomial semiring in a commuting set of indeterminates $\Lambda$
over $ \Net $ (viewed either classically or as a max-plus algebra),
and imposing relations which provide some balancing, such as
$$\lambda_1 + \lambda_2  = \lambda_1 + \lambda_3= \lambda_4 (-) \lambda_4.$$ Then
$ \lambda_2 \nabla \lambda_1 \nabla  \lambda_3, $ but $\lambda_2  $
does not balance $\lambda_3$ because of degree considerations in the
$\lambda_i$.

 \item \label{someex-iii}Following
 \cite[Example 4.2]{GJL}, in the set-up and notation of \Cref{prop-CC}, we take $H: = R/G$ with $G =\{\pm 1\}$.
 Tangible balancing holds if we take the inclusion surpassing relation, because this is a hypersystem.
 However, one can take $\mathcal I \supset  \mathcal A_{\Null}$ such that tangible balancing fails with respect to  $\mathcal I$. For example, take  $R$ to be the finite field $F_{11}$;
 we can view $ H $ as the classes  $\{ 0,1,2,3,4,5\},$  with
 $H^\times =   \{  1,2,3,4,5\}, (-))$.
 When we take $$\mathcal I = \mathcal A_{\Null} \cup \{\text {subsets  of order} \ge 4  \},$$  tangible balancing fails for $b_1 = 1+1 = \{ 0,2\}$ and $b_2 = 2+3 = \{1,5\},$
 since $b_1+b_2 = \{ 1,3, 4,5\} \in \mathcal I  $ whereas the only way for $a +b_1\in \mathcal I $ is for $a = 0$
 or $a=2$ and neither choice has $a+ b_2\in \mathcal I .$
 Notice that if $b_1
 $ and $b_2$ each are sums of two or more elements from $\tT$, then
 $b_1 + b_2 \in \mathcal I,$ from which it follows that $\mathcal A$ still satisfies
 $\mathcal I $-balance elimination.
\end{enumerate}
\end{example}



In analogy to \Cref{hyp}~\eqref{hyp1}, we identify elements of $\mcA$ to singletons, and for subsets $S_1,S_2$ of $\mcA$ we define
$S_1 + S_2 = \{ s_1 + s_2  : s_ i \in S_i\} = \cup_{s \in S_1}\{ s
+ S_2\}.$ For a subset $S$ of $\mathcal A$ we write
$a \nabla S$ if $a \nabla b$ for some $b \in S.$

Also,  we can define the new operation $\boxplus_{\nabla}$ on the set of subsets of $\tTz$ given, for all $S_i\subseteq \tTz$, by

   \begin{equation}\label{hypadd} \mathop{\boxplus_{\nabla}}_{i=1\;}^{t\quad} S_i
= \left\{a  \in \tTz :\ a \nabla \sum_{i=1}^t S_i
\right\}.\end{equation}

If $\mathcal A$ is tangibly balanced, then any element $a$ of $\mathcal A$ balances an element of $\tTz$, see \Cref{ugly}, and therefore the above set is nonempty.
Moreover, since any system is uniquely negated, we get that $a\boxplus_{\nabla} \zero=a$, for all $a\in \tTz$,
where again singletons are identified to their unique element.

For $t =2,$ $a_1 \boxplus_{\nabla}\, a_2 = \{c  \in
\tTz :  c \nabla (a_1 + a_2)   \},$ a condition we came across in \Cref{flf}.
%

\begin{proposition}\label{closed}
 Any    system $(\mathcal A,  \tT, (-),\preceq)$
satisfying tangible balance elimination also satisfies $$(b_1
\boxplus_{\nabla} b_2)  \boxplus_{\nabla} b_3\subseteq \,
b_1\boxplus_{\nabla} b_2 \boxplus_{\nabla} b_3, \quad \forall b_i\in \tTz\enspace . $$
\end{proposition}
\begin{proof}
 By definition $b_1 \boxplus_{\nabla} b_2 = \left\{c \in \tTz
 : c \nabla (b_1+b_2)
\right\},$ and
$$(b_1 \boxplus_{\nabla} b_2) \boxplus_{\nabla} b_3 = \{  a  \in \tTz :  c \nabla  (a(-)b_3),
\text{ for some } c\in \tTz  \text{ where } c \nabla (b_1+b_2)\} .$$
But then $(b_1+ b_2) \nabla (a(-)b_3)$, i.e., $a \nabla ((b_1+ b_2)
+ b_3) =  (b_1+ b_2 + b_3) .$
\end{proof}

\begin{proposition}\label{closed2} In a  tangibly balanced
  system $(\mathcal A,  \tT, (-),\preceq)$,
  we have $$ a_1 \boxplus_{\nabla} a_2
\boxplus_{\nabla} a_3 \subseteq  (a_1 \boxplus_{\nabla} a_2)
\boxplus_{\nabla} a_3 , \quad \forall a_i\in \tTz.$$\end{proposition}
\begin{proof}   We are given
$ a \nabla (a_1 +a_2 +a_3),$ so $ a_1 (-)a_3 \nabla a_1 +a_2$. By
assumption we have $c\in \tTz$ with $(a_1 (-)a_3) \nabla c$ and $ c
\nabla  (a_1 +a_2)$, so $c \in  a_1 \boxplus_{\nabla} a_2$ and $a
\nabla (a_3 \boxplus_{\nabla}(a_1 \boxplus_{\nabla} a_2))= (a_1
\boxplus_{\nabla} a_2)\boxplus_{\nabla}a_3.$
\end{proof}

\Cref{someex}~\eqref{someex-iii} shows why the tangibly balanced hypothesis  is needed in \Cref{closed2}.

\begin{definition}
    Let $(\TBS,\preceq)$ be  the subcategory of $(\MBT,\preceq)$ whose objects are tangibly balanced systems satisfying tangible balance elimination and $\nabla$-inversion,
with $\preceq$ of hyper-type. We call such systems \textbf{TBS}.
\end{definition}

 \begin{theorem}\label{hyprec} For any   tangibly balanced
 system $(\mathcal A,  \tT, (-),\preceq)$ satisfying tangible balance elimination,  $\tTz$ becomes
 a
 multiring
under $ \boxplus_{\nabla}$ of \eqref{hypadd} and the multiplication of $\tTz$. Furthermore, when
$(\mathcal A,  \tT, (-),\preceq)$ satisfies $\nabla$-inversion,
$\tTz$~becomes
 a
 hyperring.

 Let $\Psi$ be the functor  from  $(\TBS,\preceq)$ to $\mfrak{Hp}$, sending  $(\mathcal A,  \tT, (-),\preceq)$ to the hyperring $\tTz$, in the above correspondence, and sending a $\preceq$-morphism to a  morphism of hyperrings.
Letting $\Phi$ be the functor from  $\mfrak{Hp}$ to $(\TBS,\preceq)$  given by \Cref{hypersys},
 then
 $\Psi\Phi$ is the identity on $\mfrak{Hp}$, so $\Phi$ demonstrates $\mfrak{Hp}$ as a full reflective subcategory of $(\TBS,\preceq)$.
\end{theorem}
\begin{proof}  Associativity is by \Cref{closed,closed2}. Namely, $(a_1 \boxplus_{\nabla} a_2)  \sumnabla a_3\subseteq \, a_1 \boxplus_{\nabla} a_2 \sumnabla a_3$
  by \Cref{closed} and $a_1 \sumnabla
  a_2 \boxplus_{\nabla} a_3 \subseteq(a_1 \boxplus_{\nabla}
  a_2) \boxplus_{\nabla} a_3$ by \Cref{closed2}.
  A symmetrical argument shows that
  $a_1 \boxplus_{\nabla}(\boxplus a_2 \boxplus_{\nabla} a_3) =
  a_1\sumnabla a_2\sumnabla a_3$.
  It follows that $a_1 \sumnabla( a_2 \sumnabla a_3) =
(a_1  \sumnabla a_2) \sumnabla a_3$.

 To obtain unique hypernegatives in $\Hy$,
note that $a_1 \boxplus_{\nabla} a_2 = \{a \in \Hy : a_1 + a_2 (-)a
\in \mathcal A^\circ\}.$ Thus $\zero \in a_1 \boxplus_{\nabla} a_2$
iff $ a_1 + a_2  \in {\mathcal A}_{\Null},$ iff $ a_1 =(-) a_2 $.

 Reversibility is
 by \Cref{reversibility}, but also is seen immediately:  $ a \in b\boxplus_{\nabla}
c$ iff $(b+c)(-)a \in \mathcal A^\circ$, iff $(a(-)b)(-)c \in
\mathcal A^\circ$, iff $c \in a \boxplus_{\nabla}  ((-)b)$.

To verify \eqref{mul} for $\boxplus_{\nabla}$, note that if $a,
a',s_i \in \tT$ and $a' \nabla \sum_{i=1}^t  s_i,$ then $aa'\,
\nabla \, a\sum_{i=1}^t s_i = \sum_{i=1}^t a s_i,$ so $aa' \in
{\boxplus_{\nabla}}_{i=1}^t  (a s_i )$, so $a ({\boxplus_{\nabla}}
_{i=1}^t s_i)\subset {\boxplus_{\nabla}}_{i=1}^t a s_i. $
 Conversely, if $\nabla$-inversion holds, and
 $ a' \nabla \sum_ {i=1}^t  a s_i = a \sum_ {i=1}^t   s_i , $ then writing $a' = a b'$
 with $b' \nabla \sum _{i=1}^t   s_i$ and $b'\in\tTz$,
 we have $b' \in {\boxplus_{\nabla}}_{i=1}^t \,  s_i,$
so $a' \in a ({\boxplus_{\nabla}}_{i=1}^t \, s_i),$ showing
that ${\boxplus_{\nabla}}_{i=1}^t a s_i\subset a ({\boxplus_{\nabla}}
_{i=1}^t s_i)$.

 We have just shown that  every system in  $\TBS$ gives rise to a  hyperring.
   In the other direction, if  $\Hy $ is a hyperring, it gives rise
   by \Cref{hypersys}  to a system $ (\mathcal A = \langle \tT \rangle , \tT, (-), \preceq)$ in $\TBS$, where $\tTz =\Hy $ is identified
 with the singletons of $\mathcal A$, cf., \eqref{hypmul}.

 If $\Hy '$ is the
 hyperring obtained from \Cref{hyprec}, then $(\Hy, \boxplus)$ and
 $(\Hy', \boxplus_{\nabla}\,)$ are identified as hyperrings, since the hyperring additions match, since by definition
(Equation \eqref{hypadd}),
we have, for all $b_1,b_2\in \Hy'$,
$b_1 \boxplus_{\nabla} b_2 = \left\{c \in \Hy'
 : \zero\in (-c)\boxplus  (b_1\boxplus b_2)
\right\}=b_1\boxplus b_2$ .

We need to show that every  $\preceq$-morphism $f$ of nd-semiring systems of hyper-type yields a morphism $ \tilde f$ of hyperrings. Indeed, for sets $S_1,S_2$ we have
$a \nabla(S_1 + S_2) $ if and only if $a \nabla c$ for some $c\nabla (S_1 + S_2) $, if and only if $a \nabla  (S_1 + S_2) $
by tangible balance elimination. But if $a \nabla (a_1+a_2)$
then $f(a) \nabla f(a_1+a_2) \preceq f(a_1)+f(a_2),$
and thus $\tilde f(a) \subseteq  f(a_1)\boxplus f(a_2)$.
Hence
$$ \tilde f ( S_1  \boxplus  S_2 ) = \{ f(a): a \in \tT_0,\  a \nabla(S_1 \boxplus S_2)\}
=  \{ f(a): a \in \tT_0,\  a \nabla(S_1 + S_2)\}\subseteq  f(a_1)\boxplus f(a_2). $$
\end{proof}

%

 \begin{remark}\label{dich0} As a consequence of Theorem~\ref{hyprec}, every TBS is of hyper-type. We would like an equivalence of categories, but the theorem does not tell
 us that the image of~$\Phi$ obtained in
   \Cref{hypersys} is $(\TBS,\preceq)$. We say that a system $ (\mathcal A = \langle \tT \rangle , \tT, (-), \preceq)$ is \textbf{faithfully balanced} if for any $b,b'\in \mathcal{A}$, $\{a \in \tTz: a \nabla b\}=\{a \in \tTz: a \nabla b'\}$ implies $b=b'.$ (One already gets the weaker
   conclusion that $b \nabla b'$ from ``tangibly balanced.'')

   Obviously hyperrings are faithfully balanced since $\{a\}\nabla b$  means $a\in b.$ The restriction of the action of~$\Psi$ to  faithfully balanced members of $\TBS$  yields a 1:1 correspondence with the objects in $\mfrak{Hp}$.

 \end{remark}

\subsection{Obtaining hypersystems  from systems, using $\preceq$}$ $

Alternatively to \eqref{hypadd}, we can define a different candidate.
Define
   \begin{equation}\label{hypadd6} \mathop{{\boxplus_{\preceq}}}_{i=1\;\;}^{t\quad} S_i
= \left\{a  \in \tT :\ a \preceq \sum_{i=1}^t S_i
\right\}.\end{equation}

In particular, for $t =2,$ $a_1 \boxplus_{\preceq}\, a_2 = \{c  \in
\tT :  c \preceq (a_1 + a_2)   \}.$

We shall also define the following $\preceq_{\mathcal I}$ variant of
 \eqref{hypadd6}:
   \begin{equation}\label{hypadd1}
   \mathop{\boxplus_{\mathcal I}}_{i=1\;\;}^{t\quad} S_i
= \left\{a  \in \tT :\ a \preceq_{\mathcal I}  \sum_{i=1}^t S_i
\right\}.\end{equation}

\begin{remark}\label{hyper1}
We shall use  the following conditions:
\begin{enumerate}
  \item \label{hypadd11}
For all $a_0, a_1 \in \tT$ there exists  $a_2 \in \tT$ with  $a_2
\preceq_{\mathcal I} a_0+ a_1 .$ Equivalently, for all $b \in \mathcal A$ there exists  $a_2 \in \tT$ with  $a_2
\preceq_{\mathcal I} b .$

\item \label{hypadd12} If $a ' \preceq a_1+ a_2 +a_3$  for $a', a_i \in \tT,$ then there is an element $a_2' $ in $\tT$, satisfying
     $a_2 ' \preceq_{\mathcal I} a_1+ a_2$ and $a '  \preceq_{\mathcal I}   a_2' +a_3.$

\item \label{hypadd13} (Symmetrically):  If $a ' \preceq_{\mathcal I} a_1+ a_2 +a_3$
for $a', a_i \in \tT,$ then there is an element $a_3' $ in $\tT$,
satisfying
     $a_3 ' \preceq_{\mathcal I} a_2+ a_3$ and $a '  \preceq_{\mathcal I}   a_1 +a_3'.$
\end{enumerate}
 \end{remark}

\begin{theorem}\label{sysassoc} \begin{enumerate}
 \item \label{sysassoc-i} The operation $\boxplus_{\mathcal I}$ of \eqref{hypadd1}
always satisfies  $$(a_1 \boxplus_{\mathcal I} a_2)
\boxplus_{\mathcal I} a_3 \subseteq a_1 \boxplus_{\mathcal I} a_2
\boxplus_{\mathcal I} a_3.$$

 \item \label{sysassoc-ii} The operation $\boxplus_{\mathcal I}$ defines a hyperring
 when \Cref{hyper1} holds.
\end{enumerate}
\end{theorem}
 \begin{proof} $(a_1 \boxplus_{\mathcal I} a_2) \boxplus_{\mathcal I} a_3 =
\{ c' \in \tT: c'  \preceq_{\mathcal I} c+a_3, \text{ where } c
\preceq_{\mathcal I} a_1 + a_2$\}.

Similarly,  $a_1 \boxplus_{\mathcal I} (a_2 \boxplus_{\mathcal I}
a_3) = \{ c' \in \tT: c'  \preceq_{\mathcal I} c+a_1, \text{ where }
c \preceq_{\mathcal I} a_2 + a_3$\}.

Condition \eqref{hypadd11} of \Cref{hyper1} says $a_1 \boxplus a_2$ and $a_2 \boxplus a_3$ are
nonempty. Condition \eqref{hypadd12} of \Cref{hyper1}
says $a_1 \boxplus_{\mathcal I} a_2
\boxplus_{\mathcal I} a_3 \subseteq  (a_1 \boxplus_{\mathcal I} a_2)
\boxplus_{\mathcal I} a_3 $,
and likewise Condition \eqref{hypadd13} of \Cref{hyper1} says $a_1 \boxplus_{\mathcal
I} a_2 \boxplus_{\mathcal I} a_3 \subseteq  a_1 \boxplus_{\mathcal
I} (a_2 \boxplus_{\mathcal I} a_3) ,$
so we get the equality $$ (a_1 \boxplus_{\mathcal I} a_2)
\boxplus_{\mathcal I} a_3  = a_1 \boxplus_{\mathcal I} a_2
\boxplus_{\mathcal I} a_3 =  a_1 \boxplus_{\mathcal
I} (a_2 \boxplus_{\mathcal I} a_3).$$

 To obtain unique hypernegatives in $\Hy$,
note that $a_1 \boxplus_{\preceq} a_2 = \{a \in \Hy : a \preceq a_1
+ a_2\}.$ Thus $\zero \in a_1 \boxplus_{\preceq} a_2$ iff $ a_1 +
a_2 \in {\mathcal A}_{\Null},$ iff $ a_1 =(-) a_2 $.

Reversibility is
 by \Cref{reversibility}.
Double distributivity is by \Cref{ddnote}.
\end{proof}
 \Cref{sysassoc}
raises the question of how \eqref{hypadd1} compares with
\eqref{hypadd}. We have that $a\nabla_{\mathcal I}\, b$ implies $a\preceq_{\mathcal I} b$ if $a\in \tT$ and $b\in \mathcal A$, for rings $\mathcal A$ (with ${\mathcal I}=\{\zero\}$), for supertropical systems and symmetrized systems, cf.~\Cref{supert}, and for hypersystems.

\begin{proposition}\label{tancom3}
In any hypersystem over a hyperring $\mathcal{H},$ $a \nabla S$ iff  $\{a\} \preceq S$ for $a\in \mathcal H.$
\end{proposition}
 \begin{proof} 
 If $a \nabla S = \{ a_i: i \in I\},$ then $a \nabla
 a_i$ for some $i$, implying $a =
 a_i$, and thus $a \in S$, i.e., $\{a\} \preceq S$.
\end{proof}

\subsection{Special kinds of triples}$ $

Here are two properties which relate to hyperring theory.

 \subsubsection{$(-)$-Regular triples}$ $

We also weaken metatangibility,  for some hyperring examples.

\begin{definition}\label{reg1} A triple of the second kind is \textbf{$(-)$-regular} if    $(-) a_1 + a_2 + a_3
\in {\mathcal A}_{\Null}$ and $a_1 + a_2 +a_3 \in {\mathcal A}_{\Null}$ for $a_i \in \tT$
together imply $a_2 = (-)a_3.$
\end{definition}


 \begin{proposition}\label{hyp0}
Any    metatangible triple of the second kind is
$(-)$-regular.
 \end{proposition}
\begin{proof} If  $(-) a_1 + a_2 + a_3
\succeq\zero$ and $a_1 + a_2 +a_3 \in {\mathcal A}_{\Null}$ and $a_2 + a_3 \in
\tT,$ then $a_1 = a_2 + a_3 = (-)a_1,$ which contradicts
the assumption that the system is of the second kind. Thus
$a_2 = (-)a_3$.
\end{proof}

\begin{definition} A hyperring $\Hy$ is \textbf{regular} when for any
$S \in \langle \Hy \rangle ,$ if both $a,-a \in S$ then $\zero \in
S.$
\end{definition}

Note that  $(-)$ cannot be of first kind in a regular hyperring,
since one just takes $S = \{a\} = \{ (\pm) a \}.$

 \begin{proposition}\label{hyp3b}
The  hypersystem of any regular hyperring is $(-)$-regular.
 \end{proposition}
\begin{proof}
Let $\mathcal A=\langle \Hy \rangle $. Suppose $(-) a_1 \boxplus a_2
\boxplus a_3 \succeq\zero$ and $a_1 \boxplus a_2 \boxplus
a_3\succeq\zero$ for $a_i \in \Hy$. By assumption, $a_2 \boxplus
a_3$ contains both $-a_1$ and $a_1$, and thus $\zero$, implying $a_2
= (-)a_3.$
\end{proof}

 \subsubsection{Geometric triples}$ $

Here is a different sort of triple, arising from lattice theory but
also applicable to various hyperfields.
%

 We define an element $c
\in \mathcal A$ to be \textbf{(additively) $\tT$-irreducible} if $c
= a+b $ for $a $ in $\tT$ implies $b=c$ or $a = c$.

\begin{proposition}\label{geom1} Every $\tT$-irreducible element $c$ is in $\tT$.
\end{proposition}\begin{proof} Write $c = \sum_{i=1}^t a_i$ for $a_i \in
\tT,$ with $t$ minimal. We are done if $t = 1,$ so assume that
$t>1$. Then $c = a_t + \sum_{i=1}^{t-1} a_i$ implies $c =  \sum
_{i=1}^{t-1} a_i$, contrary to the minimality of $t$.
\end{proof}

\begin{definition}\label{geom} A triple $(\mathcal A, \tT, (-))$ is
\textbf{geometric} if every element of $\tT$ is $\tT$-irreducible.
\end{definition}

 \begin{lem}
A  metatangible triple is geometric iff it  is $(-)$-bipotent.
 \end{lem}
 \begin{proof}
 For $a_1\ne (-)a_2$ in $\tT,$ $a_1+a_2$ is tangible; this yields both directions.
 \end{proof}

 Thus geometric triples are a generalization of $(-)$-bipotent triples,
 which are relevant for hyperrings, as we see in the next example.
 \begin{example}$ $
\begin{enumerate}
     \item The triples of   the phase hyperfield,    the hyperfield of
   signs $\mathcal S$,  and the
   triangle hyperfield triple of \cite[\S
 5.1]{Vi} are geometric and regular.
 \item The triples of   the  Krasner hyperfield $\mathcal K$   and the
   triangle hyperfield triple of \cite[\S
 5.1]{Vi} are geometric but not regular.
   \item  The hyperfield constructed in \Cref{prop-CC} often is a geometric triple, but not
   always;
 in the complex hyperfield, one has
$(\theta, b)+ (\text{any angular sector}, b')=(\theta, b)$ if
$b>b'.$ 
\end{enumerate}
\end{example}

%

\begin{lem}\label{firstc3}  If $\one $ is irreducible and $\tT$ is a group, then $(\mathcal A', \tT', (-))$ is a geometric triple.
\end{lem}\begin{proof}
 Suppose   $a = c_1 + c_2$ with $c_i\in \tT$.  Then $\one =a
a^{-1}= c_1 a^{-1}+ c_2a^{-1},$  so
by hypothesis  $a = c_1$ or $a = c_2$.
\end{proof}

\begin{lem}\label{rev122}
The sum of elements $b = \sum_{i=1}^t a_i$ (for $a_i \in \tT$) in a
geometric triple cannot be in~$\tT$ unless some $a_i = b$.
\end{lem}\begin{proof} For any $i$, write $b =( \sum_{j \ne i} a_j
 )+
a_i.$

 Either  $a_i = b$ or  $\sum_{j\ne i} a_j =b $ and we continue
inductively on $t$.\end{proof}

%

 \section{Comparison between systems and other constructions}\label{Other}

\subsection{Fuzzy rings as systems}\label{fuzz}$ $

Fuzzy rings where introduced by Dress.

\begin{definition}[{\cite[Definitions~2.1,2.8]{Dr},
\cite[Definition~2.14]{GJL}}]\label{fuzzy0}  $\mathcal A := (\mathcal A, +, \cdot, \zero, \one )$ is a
\textbf{fuzzy ring} if it is an $\mathcal A^\times$-gen bimodule and has a
distinguished element $\vep$ and a proper  ideal $\mathcal I$
satisfying the following axioms for $a\in\mcA^\times,$  $b_i\in \mathcal A$:
\begin{enumerate}
\item\label{fuzzy1} $\vep^2 = \one;$
\item\label{fuzzy2}  $a = \vep $, iff $a\in \mathcal A^\times$  with  $\one + a \in \mathcal I$;
\item\label{fuzzy3} If   $b_1 + b_1,\ b_3 + b_4 \in I,$ then
$b_1 b_3 + \vep b_2 b_4 \in \mathcal I;$
\item\label{fuzzy4} If   $b_1 +
b_2( b_3 + b_4 ) \in \mathcal I,$ then $b_1 +  b_2  b_3  +  b_2
b_4 \in \mathcal I.$
\end{enumerate}

The fuzzy ring is \textbf{coherent} if $\mathcal A^\times$
spans $(\mathcal A,+)$.

A \textbf{weak fuzzy morphism}\cite[\S1.4]{DrW}, \cite[Definition~2.16]{GJL} of coherent fuzzy rings  $f$ from $( \mcA,\mathcal I)$ to $(\mcA,\mathcal I')$ is a   multiplicative map $f: \mcA^\times\to {\mcA'}^\times$, having the property that if $\sum a_i \in \mathcal I$ for $a_i\in \mathcal A^\times$ then
 $\sum f(a_i)\in \mathcal I'.$ This defines the category $\mfrak{wFR}$ ($\operatorname{FuzzRing}_{\operatorname{wk}}$ in \cite{BL2}).

 We also define a subcategory $\mfrak{FR}$ ($\operatorname{FuzzRing}_{\operatorname{str}}$   in \cite{BL2}) of $\mfrak{wFR}$, whose morphisms, called \textbf{fuzzy morphisms} \cite[\S1.4]{DrW},  \cite[Definition~2.17]{GJL}, are bimodule multiplicative  maps $f: (\mcA,\mcA^\times)\to (\mcA',(\mcA')^\times)$ satisfying
$\sum b_i c_i \in  \mathcal I $ implies $\sum f(b_i)f(c_i) \in \mathcal I',$  for all $b_i,c_i\in \mcA.$
 \end{definition}

\begin{lem}\label{fuz}
    There is a functor from    $\mfrak{wFR}$ to $(\WMBT,\preceq)$.
\end{lem}
\begin{proof}
    It is easy to see that for a coherent fuzzy ring $\mathcal A$, $(\mathcal A,\mathcal A^\times, (-))$ is a triple, where $\tT$ is the group~$\mathcal A^\times$ and $(-)a = \vep a.$ 
Condition~\eqref{fuzzy2} implies
$\mathcal A^\circ \subset \mathcal I,$
and $\mathcal A^\times$ is uniquely quasi-negated over $\mathcal I$,
moreover
$\mathcal A^\times \cap \mathcal I = \emptyset$ since $\mathcal I $ is a proper ideal.
Then, by \Cref{newsys}, $\mathcal I$ defines a  surpassing relation $\preceq_{\mathcal I}$.
This sends objects from $\mfrak{wFR}$ to $\WMBT $.
Weak fuzzy morphisms clearly are the same as  weak morphisms of the corresponding system.
\end{proof}

\begin{rem} We cannot generalize \Cref{fuz} to  $\mfrak{FR}$, since although  fuzzy morphisms preserve the balance relation $\nabla_{\mathcal I}$,
 the fuzzy morphism  condition does preserve $\preceq_{\mathcal I}.$
\end{rem}

 To go the other direction, we have a generalization of \cite[Theorem~3.3]{GJL}, with similar proof.

\begin{theorem}\label{fuzz07}
There is a   functor from  $(\FHSrT,\preceq_{\operatorname{hyp}})$ to  $\mfrak{FR}$, sending an nd-semiring system  of hyper-type $(\mcA,\tT,(-),\preceq)$ to the fuzzy ring $\mcA$ with the distinguished element $\vep=(-) \one$, and the proper ideal $\mathcal I=\mcA_{\Null}$.
\end{theorem}
 \begin{proof} Conditions \eqref{fuzzy1} and \eqref{fuzzy2} of \Cref{fuzzy0} are automatic for $\eps = (-)\one,$ so
 we need to verify \eqref{fuzzy3} and~\eqref{fuzzy4} of \Cref{fuzzy0}. Take $b_i\in \mcA.$

 If $b_1+b_1, b_3+b_4 \succeq \zero,$ then taking $a_i\in\tTz$, with $a_i \preceq b_i$ such that $a_1 +a_2 \succeq \zero$ and $a_3+a_4 \succeq \zero,$ we have $a_2 = (-) a_1$ and $a_4 = (-)a_3,$ so $a_1a_3 (-) a_2 a_4 = a_1 a_3 (-) a_1 a_3\succeq \zero.$ Hence $b_1 b_3 (-)b_2 b_4 \succeq a_1 a_3 (-) a_2 a_4\succeq \zero,  $
 yielding \eqref{fuzzy3}.

 If $b_1 + b_2(b_3+b_4) \succeq \zero,$ then
 take $a_1,a_2,a\in \tTz$ such that $a_i\preceq b_i$, $a\preceq b_3+b_4$ and $a_1 + a_2a \succeq \zero,$
 and then $a_1 + a_2b_3 + a_2b_4 =a_1 + a_2(b_3+b_4)\succeq a_1 + a_2a
 \succeq \zero,$
 implying $b_1 + b_2b_3 + b_2b_4 \succeq \zero,$
  yielding \eqref{fuzzy4}.

  Likewise, if $f$ is a morphism and $\sum b_i b_i'\succeq \zero$ then taking $a_i,a_i'\in \tTz$, such that $a_i \preceq b_i,$ $a'_i \preceq b'_i$ and $\sum_i a_i a_i' \succeq \zero$, we have  $\zero \preceq f(\sum_i a_i a_i')\preceq \sum_i f(a_i)f(a_i') \preceq \sum_i f(b_i)f(b_i').$
  \end{proof}

 \begin{theorem}\label{fuzz071}
  There is a dominant functor sending $(\mfrak{wFBT},\preceq_{\operatorname{hyp}})$ (having weak morphisms)
      to $\mfrak{wFR}$.
 \end{theorem}\begin{proof}
      The same proof, now in conjunction with \Cref{fuz}.
 \end{proof}

\begin{example}\label{fuzzytrop}
 The supertropical semiring is isomorphic to the construction in \cite[Theorem~4.3]{DrW}, where we identify $a^{\circ}$
 with the subset $\{a' : a' \le a \}.$
 Thus
\cite[Definition~4.1]{IKR} corresponds to its valuation.
\footnote{As a special case, by \cite[Example~2.33~(ii)]{AGR1}, the
Puiseux series system has a surjective non-archimedean valuation. We
follow the notation of \cite[Example~2.33~(iii)]{AGR1}, extending
$\Phi$ element wise to $\nsets (\mathcal A)$, and writing $S_K$
for $\Phi^{-1}(S)$ for any $ S\in \nsets (\mathcal A)$.  The same proof as
\cite[Lemma~4.5]{DrW} shows that $\Phi$ preserves the systemic
operations (with the exception of when the residue field has order
2, which requires special treatment). Hence the supertropical
semiring, viewed as the fuzzy ring of \Cref{fuzzytrop}, is a
fuzzy ring, as explained after \cite[Lemma~4.5]{DrW}.}
\end{example}

 \begin{remark}\label{fuzz072}$ $
     \begin{enumerate}
         \item \Cref{fuzz07} extends to the category $\MBT$, when one   generalizes \Cref{fuzzy0}, replacing $A^\times$ by any monoid $\tT$, with Axiom~(ii) replaced by:

         $a_1 = \varepsilon a_2$, for $a_i\in \tT,$ if and only if $a_1 + a_2 \in \mcI.$

         \item Fuzzy rings are called
  \textbf{field-like} in \cite{GJL} if for any $a_1,a_2\in \mcA^\times$ there is $a\in \mcA^\times \cup \{ \zero\}$ such that $a_1+a_2+a \in \mathcal I.$ The corresponding systemic axiom is given in \Cref{flf}.

    \item A natural functor from the subcategory of metatangible systems in  $(\FHSrT;\preceq_{\operatorname{hyp}})$ to $\mfrak{FR}$
 was given in \cite{Row21,JuRo}. In this case,
   Condition \eqref{fuzzy4} is a special case of associativity, so we are left with \eqref{fuzzy3}, which
 holds for metatangible systems by \cite[Theorem~6.54]{Row21}.

 But metatangible semiring systems need not be of hyper-type (cf.~\Cref{supert}).
     \end{enumerate}
 \end{remark}

 \subsection{Blueprints and systems}\label{bluesy}$ $

Let us review blueprints~\cite{Lor1,Lor2} and how they are developed, putting them
in the context of semiring systems.
In~\cite[Def.~1.1]{Lor1}, Lorscheid starts with a multiplicative monoid $\tT$
embedded in the monoid semiring $\Net[\tT]$, and defines a blueprint
as a pair $(\Net[\tT], \mathcal{R})$ where $\mathcal{R}$ is a congruence
on $\Net[\tT]$. A blueprint is said to be
proper if the restriction of this congruence to $\tT$ is trivial.
By taking $\mathcal{A}=\Net[\tT]/\mathcal{R}$,
we see that blueprints yield a special class of semirings,
that are equipped with an mgen bimodule structure and are additively
generated by the image of $\tT$ by the quotient map modulo $\mathcal{R}$. In particular, when the blueprint
is proper, $\mathcal{A}$ contains $\tT$ and is additively generated
by $\tT$. Such a blueprint is almost a semiring triple: unlike triples,
it is not equipped with a negation map.
More recently, in~\cite{BL},
there are notions of ordered blueprint and of $\mathbb{F}_1^{\pm}$-algebra,
which are closer to the notion
of a semiring system: an ordered blueprint is already a semiring $\mathcal{A}$ additively generated by $\tT$, with a $\tT$-gen bimodule pre-order, but this pre-order
is not required to satisfy the additional properties of a surpassing relation \Cref{precedeq07}
(the notation $B^\bullet$ corresponds to $\tT$,
 $B^+$ corresponds to $\mathcal A$); a $\mathbb{F}_1^{\pm}$-algebra
 is an ordered blueprint equipped with a negation
 (the distinguished element $\epsilon$ of~\cite{BL} plays
 the role of $(-)1$ here), satisfying Condition~\eqref{precedeq07-1}  of \Cref{precedeq07} (the order is a $\tT$-pre-order), however it
 does not satisfies all the properties of semiring system (Condition~\eqref{triple3} of~\Cref{int1},
nor Condition~\eqref{surp-5} of \Cref{precedeq07}, or the
unique negation property of \Cref{syst}).
The subclass of ``pasteurized ordered blueprints''
satisfies the unique negation property.


\begin{remark} $\Net[\tT]$ is generic in the sense that for any
$\tT$-gen bimodule $\mathcal A$, there is a  map
$\varphi_{\mathcal A}: \Net[\tT]\to \mathcal A$ which is the
identity on $\tT$ and sends the formal sum $\sum n_a a$ to its sum
taken in $\mathcal A$. We write $\bar a$ for $\varphi_{\mathcal
A}(a).$ A certain subtlety is involved here: If $\mathcal A$ is not
a distributive semiring, which is the case for certain hyperfields, then
$\varphi_{\mathcal A}$ cannot be a semiring homomorphism, but
$\varphi_{\mathcal A}$ is indeed a homomorphism when $\mathcal A$ is
 distributive.\end{remark}

\begin{remark}
  Any partial order $\preceq$ of a $\tT$-gen semiring   $\mathcal A$
 additively generated by $\tT$,
induces a pre-order $\preceq_{\mathcal A}$  on
$\Net[\tT]$, given by $b \preceq b'$ if writing $b = \sum n_a a$ and
$b' = \sum n_a' a$ we have  $ \sum n_a \bar a\preceq \sum n_a' \bar
a$ in $\mathcal A$. However, even when $\preceq$ is a surpassing relation, the induced
pre-order $\preceq_{\mathcal{A}}$ may not be a surpassing relation
(property (c) of~\Cref{precedeq07} may fail).
\end{remark}

\subsection{Tracts, idylls and systems}\label{trac}$ $

The same sort of
 argument as in \Cref{bluesy} shows that the tracts of \cite{BB2}
 and idylls of \cite{BL} are closely related to triples and systems in the special case that
 $\mathcal A = \mathbb N[G]$,
 for a group $G$. 
In \cite{BL}, a tract is a group $G$ with the monoid semiring $\Net[G]$, a subset $N_G$ of $\Net[G]$ stable by the multiplication of elements of $G$, and a negation satisfying $1 + (-)1 \in N_G$
and unique quasi-negation over $N_G$
 (again the distinguished element $\epsilon$ of~\cite{BL} plays
 the role of $(-)1$ here).
 In \cite{BL} tracts such that $N_G$ is a $G$-sub-bimodule
 of $\Net[G]$ are called idylls.


\begin{theorem}\label{trac1}   Any tract gives rise to a
  triple of the form $( \mathcal A = \mathbb N[G] , G, (-))$
with a group $\tT = G$, satisfying  unique quasi-negation over $N_G$.
When the tract is an idyll, i.e.\ $N_G$  is a $G$-sub-bimodule of $\mathbb N[G]$, then it also gives rise to a system
$( \mathcal A = \mathbb N[G] , G, (-),\preceq_{\mathcal I})$ where
$\mathcal I = N_G$. This gives a faithful functor from such tracts to $(\MSr,\preceq)$.

Conversely, any system of the form $( \mathcal A = \mathbb N[G] ,G,
(-),\preceq)$, for  $G = \tT$ a group, gives rise to an idyll where $N_G = \{
b : \zero \preceq b\}$.
\end{theorem}

\begin{proof} $G$ additively generates $ \mathbb N[G]$. Since $\varepsilon$ is the negative of $\one,$ the negation map on
$\mathcal A$ is given by $a \mapsto \varepsilon a.$  Thus we have a
triple. If $b \in N[G]$ then $\zero + b = b$, implying $\zero
\preceq b$. If $ \zero \preceq b$ then $  b =  \zero +c = c \in
N_G$.

For the converse, we recall that $N_G=\{ b : \zero \preceq b\}$ is a
$\tT$-gen bimodule, which is disjoint from $\tT$ so does not
contain~$\one.$
\end{proof}

 \begin{remark}\label{dich0b}$ $\begin{enumerate}
 \item In the first assertion of \Cref{trac1} we have ``forgotten'' $N_G$ of the tract,
which could be any subset of $\mathcal A$ not containing
$\one,$ having a $G$-action. If we want to replace $N_G$ in forming the system,
 let us consider a $G$-submodule $N_G$ must sit inside the
triple $( \mathcal A = \mathbb N[G] , G, (-))$. $\tT ^\circ
\subseteq N_G $ since $\one ^\circ \in N_G$. In
 the hypersystem of \Cref{hypersys}, the intuition stated after \cite[condition
(T3)]{BB2} indeed indicates that all elements of~$N_G $ contain
$\zero$, which is reinforced in \cite[\S 1.4]{BB2}.  Thus it would
be natural to take $N_G$ to be any $G$-sub-bimodule of $\{ b : \zero
\preceq b\}$ containing $G^\circ $.
\item The definition of $\preceq$ in \Cref{trac1} is  analogous to
 \Cref{precex}~\eqref{precex-i}, of $\preceq_\circ$, and in fact we could
 take $\preceq = \preceq_\circ$ and $N_G = \mathcal A^\circ$.

\item   In the usual hyperfield
examples, $N_G = \mathcal A^\circ,$  as explained in~Footnote 3.


\end{enumerate}
\end{remark}

\section{Induced algebraic constructions}
\label{Induced}

Some of the standard algebraic constructions induce systems based on
the original systems.

\subsection{Function systems:  polynomials and matrices}$ $

 Here is a convenient way for building up triples and systems.
\label{polmat}
We  consider a structure 
 generalizing the construction of
semigroup algebras. For any set $S$, the \textbf{support} $\supp{f}$ of a function $f:S\to \mathcal{A}$
is $\{ s\in S : f(s) \ne \zero\}.$

\begin{definition} [The \textbf{convolution endofunctor}]\label{conv}
 Consider a \distributed $\tT$-gen nd-semiring $\mcA$  
 and a multiplicative monoid $S$.
If $S$ has an absorbing element $0_S$, we set $S^*:=S\setminus \{0_S\}$, and $S^*:=S$ otherwise.

We form the  \distributed $\tT_S$-nd-semiring
 $\mathcal A^{(S^*)}$ to be the set of functions $f:S^*\to \mathcal A$ with finite support, endowed with the pointwise addition and negation,
$\tT_S$ is defined as the set of functions   $f:S^*\to  \tTz$ with $|\supp{f}|=1$,
and the  action of $f\in (\tT_S)_\zero$ on $g\in \mathcal A^{(S^*)} $ is
denoted $f*g$ and defined, for $s\in S^*$, as
$(f * g)(s) =  \sum_{u,v\in S^*: \, uv = s}f(u) g(v)$,
where the empty sum is $\zero$.
Since $S$ has a unit $\one_S$, considering the element
of $\tT_S$ with support $\{\one_S\}$ and value $\one_{\mathcal A}$ for $\one_S$, we get the unit of the monoid
$ (\tT_S)_\zero$.
(A more general situation is given in \Cref{conv1}.)

In particular, the multiplication in $ \tT_S$ is
such that the support of $f * g$ is equal to
$\{uv\}\cap S^*$ and
$(f * g)(uv)=f(u)g(v)$ if
$uv\in S^*$, $\supp{f}=\{u\} $ and
$\supp{g}=\{v\} $.

Any morphism $\varphi:\mcA\to \mcA'$ induces a morphism $\varphi^S$ of $\mcA^{(S^*)}$,  by $\varphi^S(f)(b)= \varphi(f(b)),$ for $f\in \mcA^{(S^*)}$, $b\in \mcA.$
This provides an endofunctor $^S$ of $\MMod$.
which we call the \textbf{convolution endofunctor}.

If we started with a triple $(\mathcal A, \tT, (-))$, then we get a triple
$$(\mathcal A^{(S^*)} , \tT_{S}, (-)),$$  called  the \textbf{convolution triple}, where $((-)h)(s) =(-)(h(s)).$ This again provides a \textbf{convolution endofunctor} of $\MBT$.

If we started with a system, then we get a system
$$(\mathcal A^{(S^*)} , \tT_{S}, (-),\preceq),$$  called  the \textbf{convolution system}, where
we write $f\preceq g$ if and only if $f(s)\preceq g(s)$ for all $s\in S.$. This provides a \textbf{convolution endofunctor} of $(\MB,\preceq)$.

If in addition, $\mathcal A$ is a \distributed $\tT$-gen nd-semiring,
 the set $\mathcal A^{(S^*)} $, endowed with the above addition and convolution product given by $(f * g)(s) = \sum_{uv = s}f(u)g(v)$, is a \distributed $\tT$-gen nd-semiring, denoted $\mathcal A[S]$, and $\tT_S$ is also denoted
 $\tT_{\mathcal A[S]}$.
\end{definition}

 As examples, any system $(\mathcal A, \tT, (-),\preceq)$ naturally gives rise to two important secondary constructions of convolution systems:
   \begin{itemize}  \item[-] If $(\lambda_i)_{i\in I}$ is a family of variables, and $\Lambda$ is the commutative monoid consisting of formal products of these variables, we obtain the \textbf{polynomial system} $(\mathcal A[\Lambda], \tT_{\mathcal A[\Lambda]}, (-),\preceq)$,
   for $\tT_{\mathcal A[\Lambda]}$ the set of monomials with coefficients in
   $\tT$.


  \item[-] The other main construction, matrices, is
  subtler, since it involves a smigroup $S$ which is not a monoid, see \Cref{conv1}.
  Matrix theory is pursued in \cite{AGR1}.
   \end{itemize}


 \begin{rem}
 $ $
  \begin{itemize}
  \item
  For the polynomial system, one might be tempted to take as set of tangible elements, the set of  polynomials with tangible
   coefficients, but this is not a monoid if $\one(-)\one\ne \zero$ since then $(\one+\lambda)(\one(-)\lambda) = \one+(\one(-)\one)\lambda + \lambda^2$ and $\one(-)\one$ is not tangible.

\item Here is an advantage of the structure theory of systems over hyperrings. The usual construction  $\mathcal{H}[\Lambda]$ of polynomials over a hyperring $\mathcal{H}$
 is equivalent to considering the set  of
 polynomials with tangible or zero coefficients.
However, if $\mathcal{H}$ is not a ring, then
 the multiplicative law is multivalued, and so $\mathcal{H}[\Lambda]$  is not a hyperring.

\item
An alternate hyperring construction
of polynomials over a hyperring $R$ could be
attempted by means of \Cref{prop-CC}, where one
respectively takes
 $R[\la]/G$ , using the original $G$ as constants embedded as a subgroup, to define the congruence.
Note that this does not provide enough elements in the product given in \eqref{mul}, since one would want to take cosets of each coefficient.
 On the other hand, taking cosets for  each coefficient modulo $G$ would make  the coefficient of $\la$ in $(\la+\one)(\la +2) = \la^2 + 3\la +2$ too large a set.

\end{itemize}
 \end{rem}

 \subsection{Layered systems}\label{layered}

 \begin{construction}\label{AGGGexmod1}
  We are given a multiplicative  monoid $(\tT_L, \cdot, 1)$ without a zero element, an ordered monoid $(\tG,\cdot,\one)$
     without a bottom element,
     and a  $\tT_L$-gen module $(L,+)$,  perhaps lacking  a zero element.
   \begin{enumerate}  \item\label{AGGGexmod1-i} We equip $L\times \tG$ with the following addition
   \begin{equation}
\label{basicex17}(\ell_1,a_1) + (\ell_2,a_2) = \begin{cases}
(\ell_1,a_1) \text{ if } a_1 > a_2,
 \\ (\ell_2,a_2) \text{ if } a_1 < a_2,  \\  (\ell_1 + \ell_2,\, a_1)
 \text{ if }   a_1 =  a_2  ,
\end{cases}\end{equation}
and the multiplicative action of $\tT = \tT_{
L \rtimes \tG}:=\tT_L \times \tG$
(defined component-wise). Since $\tG$ has no bottom element,
the above addition has no zero element. So, we adjoin
formally an absorbing zero element $\zero$ to $L\times \tG$.
$ L \rtimes \tG$
consists of the set $(L\times \tG) \cup \{\zero\}$,
equipped with the above operations.



    \item \label{AGGGexmod1-ii}
We now assume that $L$ has a negation map $(-)$,
i.e., $(-)\one\in \tT_L$. We define the negation map $(-)$ on $ L \rtimes \tG$ by
$(-)(\ell,a) =((-)\ell,a)$, and the kind of $(-)$ on $ L \rtimes \tG$ is the same kind
as $(-)$ on~$L$.

 We call $ L \rtimes \tG$ the \textbf{layered extension} of $\tG$ by $L$. It
 becomes a
triple $( L \rtimes \tG,   \tT_{ L \rtimes \tG}, (-))$, 
 in view of the 1:1 morphism
 $ L \to L\times \{\one\}$.
The quasi-zeros will be the elements having first coefficient in
$L^\circ$.
\item \label{AGGGexmod1-iii}
If $L$ has no zero, we formally adjoin a  multiplicatively absorbing element $\zero$ to get a monoid $L_\zero:= L\cup \{\zero\}$.
If $(L,\tT_L,(-)$ is a triple, then $( L \rtimes \tG,   \tT_{ L \rtimes \tG}, (-))$ is a triple.

If $L_\zero$ has a surpassing relation $\preceq$, we get the \textbf{layered system} by defining a surpassing
relation on $L_\zero \rtimes \tG$ by $(\ell_1,a_1) \preceq (\ell_2,a_2)$
when  $\zero\preceq l_2$ with $a_1< a_2$, or  $\ell_1 \preceq \ell_2$ with $a_1 = a_2$, and $\zero\preceq (\ell_2,a_2)$ when $\zero\preceq l_2$. Then $(L \rtimes \tG)_{\Null}=L_\Null\rtimes \tG$ and $L \rtimes \tG$ is a system if $L$ or $L_\zero$ is a system.
\item  \label{AGGGexmod1-iv} If $L$ or $L_\zero$ is a semiring, then we define multiplication on  $ L \rtimes
\tG$ component-wise for elements in $L\times \tG$, and such that $\zero$ is the absorbing element.
  \end{enumerate}
\end{construction}

\begin{rem}
 This
type of extension was already considered in~\cite[Proposition
2.12]{AGG2}. It provides a category   of layered extensions, with the morphisms being pairs $(f,h):L \rtimes \tG \to L' \rtimes \tG'$ where $f:L\to L'$ is a semigroup homomorphism sending $ \tT_{ L \rtimes \tG}$ to $ \tT_{ L '\rtimes \tG'},$ and $h: \tG \to  \tG '$ is a monoid homomorphism.

\end{rem}

When $\overline{\Hy}$ is the hyperfield system of a hyperfield  $\Hy$,  and $L=\overline{\Hy}\setminus \{0\}$ or $\overline{\Hy}$, we must modify addition in \Cref{AGGGexmod1} since the reversibility
 condition of  \Cref{hyp}\eqref{hyp4} fails. For example, if $a_1
 < a_2$ then $(\ell_1,a_1) + (\ell_2,a_2) = (\ell_2,a_2)$ but
 $(\ell_2,a_2)-(\ell_2,a_2)$ does not necessarily contain $(\ell_1,a_1)$
since $a_2$ might have no connection to $a_1 .$ As a remedy, we consider the following extension for hyperfields.
 \begin{definition}\label{AGGGexmod2} Let $\overline{\Hy}$ be the hyperfield system of the hyperfield  $\Hy$,  and let $L=\overline{\Hy}\setminus \{0\}$ or $\overline{\Hy}$ and $\Hy'$ be equal to ${\Hy}\setminus \{0\}$ or $\Hy$ respectively. We define the extension $L\rtimes \tG$ as the  hyperfield system of the hyperfield $\Hy'\rtimes \tG$ with the operations defined as in \Cref{AGGGexmod1}, where we modify
the third line of \eqref{basicex17}
by
\[
(\ell_1,a_1) + (\ell_2,a_1)
= \begin{cases}
 (\ell_1+\ell_2)\times \{a_1\}&\text{
 when}\; \ell_1\neq -\ell_2;\\

 ((\ell_1-\ell_1)\setminus\{0\})\times \{a_1\}\cup
( \Hy'\times \{a\in \tG,\; a<a_1\})\cup\{\zero\}
& \text{
otherwise.}\end{cases}
\]
\end{definition}

 Under this alternate definition,   reversibility holds in
 $L\rtimes \tG$, which thus indeed  is  a hyperfield system associated to the hyperfield $(\tT_L\times \tG)\cup\{0\}$.

\begin{example}\label{trythis1} When  $(\mathcal A, \tT, (-))$ is a $(-)$-bipotent triple with $(-)$ of the first kind
(i.e., identity map),
then $\tTz$ is ordered by $a_1 < a_2$ iff $a_1+a_2 = a_2,$ so when $(\tT,\cdot)$ is a monoid
we can perform the construction of
\Cref{AGGGexmod1}, for   $\tG = \tT$ or $\tG = \tTz$.
Note that, using  \Cref{circint3},
we can put $L\rtimes \tTz$ in bijection
with the quotient of $L\times \mathcal{A}$
by the equivalence relation such that
$(m \ell,a)\equiv (\ell,ma)$ for all positive
integers~$m$, and $(\ell,\zero)\equiv (\ell',\zero)$
for all $\ell,\ell'\in L$ and $a\in\mathcal{A}$.
We denote by $L\rtimes \mathcal{A}$ the quotient
set, equipped with the induced laws.
\end{example}

The $\Net$-layered semiring  is a good source of examples, taking the negation map on $\Net$ to be the identity.
Another family of examples arises when taking for the layering triple the symmetrization $\hat{\Net}$ of the semiring of natural numbers discussed in \Cref{netneg}.
Alternatively, we can \textbf{truncate} $\Net$ at $m$ by applying the
congruence $\{ (\ell,\ell), (\ell',\ell''): \ell',\ell''\ge m\},$
i.e., defining the new sum $\ell_1 + \ell_2 = m$ whenever $\ell_1 +
\ell_2 \ge m$ in $\Net$ and $\ell_1  \ell_2 = m$ whenever $\ell_1
\ell_2 \ge m$ in $\Net$. The cases $m=1$ and $m=2$ are respectively
called \textbf{Boolean} and \textbf{superBoolean} in \cite{RhS}. We
get supertropical algebra for $m=2$.
 Here  are specific   examples, where   the choice of the surpassing relation $\preceq$ can
 lead to unusual properties.

 \begin{example}\label{semidir287b}$ $
  \begin{enumerate}    \item \label{semidir287b-i}   $L = \Net$, $(-)$ is the identity map,
 and $\tT_{ L \rtimes \tG} : = \{1\} \times \tG$.

\begin{itemize}
     \item  We define  $ (k_1,
 a_1) \preceq (k_2, a_2)$ whenever  $ a_1 = a_2 $ and either $ k_1= k_2=1$ or $ 2 \le k_1 \le
 k_2.$ Then for $k_1 = k_2 =2,$ we cannot have $(1,a_1) \preceq
 (2,a_1)$, and thus the $\preceq$ analog of tangible balance elimination fails.

 \item  We define  $ (k_1,
 a_1) \preceq (k_2, a_2)$ whenever  $ a_1 = a_2 $ and $k_1 =  k_2 $ or $k_1 = 1,\ k_2 \ge
 4$.
Then  $ (1,
 a) \preceq (2, a)$ and $ (1,
 a) \preceq (3, a)$, but $(2, a)$    does not balance $ (3, a)$, so again the    $\preceq$ analog of  tangible balance elimination
 fails.
\end{itemize}
   \item \label{semidir287b-ii}Here is an example we promised earlier after \Cref{geom1}.
    $L = \Net_{\max}$, $(-)$ is the identity map,  and $\tT_{ L \rtimes \tG} : = \{1\} \times \tG$.
 The elements of $\{2\} \rtimes \tG$ are $\tT$-irreducible although not tangible, since 2 is $\tT$-irreducible in $L$.
  But $ L \rtimes \tG$
 is  not a triple, since it is not spanned by $\tT_{ L \rtimes \tG}.$  \end{enumerate}
\end{example}

\begin{example}\label{trythis}
We can do the same sort of construction as in
\Cref{AGGGexmod1}~\eqref{AGGGexmod1-ii}, for $(-)$-bipotent triples
 $(\mathcal A, \tT, (-))$ instead of the monoid $\tG$, where we want to
 bring the triple structure of $(\mathcal A, \tT, (-))$ into the picture.
 In case $(-)$ is of the first kind we already have
 \Cref{trythis1}, so assume $(-)$ is of the second kind. For example, we
 could take a symmetrized system. In particular,
$\mathcal A$ is idempotent, so $\mathcal A = \tT_{\zero} \cup
\tT^\circ.$ In place of~\eqref{basicex17}, using
 \Cref{netneg}, writing
 $b_i = (m_i(-)n_i) {b_i}_\tT$ for $i=
 1,2,$ where either $m_i = n_i =1$ or ($m_i=1$ and $n_i = 0$), we take
 $$(\ell_1,b_1) + (\ell_2,b_2) = \begin{cases}
     ((m_1 (-)n_1)\ell_1 + (m_2 (-)n_2) \ell_2,\,  {b_1}_\tT)
   \text{ if }   {b_1}_\tT = {b_2}_\tT ,\\
        ((m_1 (-)n_1)\ell_1 + (n_2 (-)m_2) \ell_2,\,  {b_1}_\tT)
   \text{ if }   {b_1}_\tT = (-){b_2}_\tT ,\\
        (\ell_1,b_1) \text{ if } {b_1} + {b_2} =
  {b_1}, \text{ and } {b_{1}}_\tT \neq (\pm){b_{2}}_\tT ,
  \\ (\ell_2,b_2) \text{ if } {b_1} + {b_2} = {b_2}, \text{ and } {b_{1}}_\tT \neq (\pm){b_{2}}_\tT . \end{cases}$$
 In order for this addition to be independent of the choice
 of ${b_{i}}_{\tT}$ in the representation $b_i = (m_i(-)n_i) {b_i}_\tT$,
 we need to quotient by the equivalence relation
 $(\ell,(-)b)\equiv ((-)\ell,b)$
 and $((1 (-) 1) \ell, b)\equiv (\ell,b^\circ)$.


 We denote the quotient set by $L\rtimes \mathcal{A}$,
 and we equip it with the
 multiplicative action of $\tT_L \times \tT$
 (defined component-wise), and with the
 negation $(-)(\ell,a)=((-)\ell,a)=(\ell,(-)a)$.
 In this way, $L\rtimes \mathcal{A}$ becomes a triple.


 So we almost have replicated \Cref{AGGGexmod1},
the difference being that, for $(-)$ of the second kind on~$\mathcal
A$,
 strictly speaking $\tT $ is not ordered, since $a$  and~$(-)a$ are not comparable.
%
%
\end{example}


A classical algebra of characteristic~2 provides an example
 of a
triple of first kind with $e' = \one,$ which is not $(-)$-bipotent.

\begin{example}\label{semidir41}   These ideas lead to
several different ways of extending tropical mathematics to the
field~$\mathbb C$ of complex numbers, providing alternative
structures to the hyperfields considered
in~\Cref{ex-tropicalcomplex} and~\Cref{ex-viroc}.

\begin{enumerate}\label{Cond1}

\item
  \label{semidir4}
   The semiring
   $\mathbb C \rtimes \mathbb R$, where $\mathbb R$
   is equipped with its structure of an additive ordered group,
   was considered in~\cite[Example 2.13]{AGG2}
   (with a different notation). There it was called
   ``the complex semiring extension of the tropical semiring.''
  It constitutes a layered system with $\mathcal{T}_{\mathbb{C}}=\mathbb{C}^*$, the usual negation and
  the equality surpassing relation in $\mathbb C$. Then $\mathbb C \rtimes \mathbb R$ is
  equipped with the induced negation and
  the surpassing relation $\preceq_{\circ}$.

As explained in \cite[Example 2.13]{AGG2}, an element $(\ell; a) \in
\mathbb C \rtimes \mathbb R$ encodes the
``asymptotic series'' $\ell t ^{-a} + o(t^{-a}),$ when $t$ goes to
$0^+$.  Indeed, the ``lexicographic'' rule in the addition of
$\mathbb C \rtimes \mathbb R$ corresponds
precisely to the addition of asymptotic series, and the
component-wise product of  $\mathbb C \rtimes \mathbb
R$ corresponds to the product of asymptotic
series.

Another way to get this layered system is to consider the hyperfield
defined as follows. Take   $K: =\mathbb{C}\{\{t^{\mathbb{R}}\}\}$
the field of generalized Puiseux series with real exponents and
complex coefficients, and $G$ the subgroup consisting of series of
leading exponent $0$ and leading coefficient $1$. Then, a non-zero
element of $\kfield/hgroup$ is the coset of an element $\ell t^{-a}$ with
$\ell\in \mathbb{C}^*$ and $a\in \mathbb{R}$,
which can be identified to the element $(\ell,a)\in \mathbb C\rtimes \mathbb R$.
Then the hypersystem of the hyperfield $\kfield/hgroup$ is isomorphic to
$\mathbb C\rtimes \mathbb R$, by sending the coset of $\ell t^{-a}$
to $(\ell,a)$, when $\ell\neq 0$, and by sending the set
$t^{-a}\hgroup-t^{-a}\hgroup$, which is the set of series with valuation
strictly greater than $a$, to $(0,a)$.





\item
  Let $\phase=\overline{\Hy}$ denote the hypersystem
  of the phase hyperfield $\Hy$, defined in~\Cref{ex-phase}, and consider
  the  layered system $(\phase\setminus\{0\})\rtimes \R$
  obtained with the
ordered monoid $(\R,+,\leq)$.
Using \Cref{AGGGexmod2}, we recover the hyperfield system of  the phased tropical hyperfield (see \Cref{ex-tropicalcomplex}).

Recall that $\phase$ is not a semiring with the hyperfield multiplication. The same holds for
 the  layered hyperfield system $(\phase\setminus\{0\})\rtimes \R$ obtained in this way.

\item  Let  $\phase^{\textrm{w}}=\overline{\Hy}$ denote the hypersystem
  of the weak phase hyperfield $\Hy$, of \Cref{ex-phasealt}.
  We can  consider the extension
  $(\phase^{\textrm{w}}\setminus\{0\}) \rtimes \R$,
   with the ordered monoid $(\R,+,\leq)$.
   Using \Cref{AGGGexmod2},
   we recover the hyperfield system of the hyperfield of complex numbers
  defined by Viro~(see \Cref{ex-viroc}).

  \item  A  variant of the phased tropical hyperfield
  and the hyperfield of complex numbers of Viro
  is \textbf{the phased extension of the tropical semiring},
  proposed in \cite[Ex.\ 2.22]{AGG2}.
Its definition coincides with the one of $L\rtimes \R$
in  \Cref{AGGGexmod1}, with $L=\mathcal A\setminus\{\zero\}$, in which
 $\mathcal A$ is the
phase semiring defined in~\cite[Ex.\ 2.22]{AGG2}.
Since this semiring coincides with the hypersystem
obtained from $\phase$ by using the  multiplication $\cdot$ in~\eqref{mult-doub}, we get that
 the variant of \cite[Ex.\ 2.22]{AGG2}
coincides with the layered system $(\phase\setminus\{0\})\rtimes \R$
defined as in \Cref{AGGGexmod1}.

It is also similar to the hypersystem extension $(\phase^{\textrm{w}}\setminus\{0\}) \rtimes \R$, defined using \Cref{AGGGexmod1},
except for the addition of two opposite elements.

   This variant has the advantage of being an semiring and a layered
system.



\end{enumerate}
\end{example}

\section{Matroids over systems} \label{sec-mat-sys}$ $
Dress  introduced matroids over fuzzy rings in
 \cite{Dr}, see  \cite[Definition~1.1]{DrW}, as a
vehicle for unifying and studying matroids.
 M.~Baker and O.~Lorscheid \cite{BL}
 developed a more general setting of matroids over ``blueprints''.
Here we discuss the notion of matroids in conjunction with systems, while generalizing the notion already in the case of fuzzy rings or blueprints as follows, in order to obtain more representable matroids.

\subsection{$\mathcal{A}$-matroids and Grassmann-Plücker maps}
 \begin{definition}\label{mat0}
A \textbf{Grassmann-Pl\"ucker map} of degree $m$ with ground set $E$ and values in a semiring
system $(\mathcal A, \tT,(-), \preceq)$
is a map $b:E^m\to\mathcal{A}$
 satisfying the following
conditions:
\begin{enumerate}
 \item \label{mat0-i} There exist $e_1, \dots, e_m \in E$ with  $b(e_1, \dots, e_m)
 \in \tT;$
 \item  \label{mat0-ii} $b(e_1, \dots, e_m) \in {\mathcal A}_{\Null}$ whenever
   $e_i = e_j$ for  some $i\ne j$;
   \item  \label{mat0-iii}
     $b(e_{\pi(1)}, \dots, e_{\pi(m)})= (-)^{\sgn(\pi)} b(e_1, \dots,
 e_m)\quad \text{ for every permutation } \pi \in S_m.$

  \item  \label{mat0-iv}Any $e_0,e_1, \dots, e_m, e_2', \dots, e_m'$ satisfy the
 \textbf{Grassmann-Pl\"ucker relations}
 \begin{equation}\label{Plucker}
  \sum_{i=0}^m  (-)^i  b(e_0,e_1, \dots, e_{i-1}, e_{i+1},
 e_m)b(e_i,e_2',  \dots,
 e_m') \in {\mathcal A}_{\Null},
\end{equation}
\end{enumerate}
An \textbf{$\mathcal{A}$-matroid} is an equivalence class of such maps, with two maps $b_1,b_2$ called
\textbf{equivalent} if there are $a_1,a_2 \in \tT$ such that $a_1
b_1 = a_2b_1,$

  A \textbf{base} for an $\mathcal{A}$-matroid is a set $\{e_1, \dots, e_m \}$
  such that $b(e_1, \dots, e_m) \in \tT.$
    \end{definition}
 In contrast with~\cite[Definition~4.2]{DrW} and \cite{GJL,BL}, we allow
 $b$ to take values in all of ${\mathcal A}$ whereas in~\cite{DrW},
 the values of $b$ are required to be invertible or zero;
 i.e., the definition of~\cite{DrW} requires
 the values of $b$ to be in $\mathcal{T}\cup\{\zero\}$.
 The same kind of restriction is made in~\cite{GJL},
 and also in the definition of matroids over tracts, which includes
 matroids over idylls~\cite{BL}.

For any fuzzy ring $K$ with distinguished ideal $I$,  define $K' =
\Net[K^\times].$ Then the identity map on $K^\times$ induces an
additive semigroup map $f: K' \to K$, and one defines $I' =
f^{-1}(I)$. This defines a new fuzzy ring having the same
matroids.    Baker and Lorscheid \cite[Prop.~3.6]{BL} have proved that all
matroids represented over a fuzzy ring can be recovered in this way.

 Explicit examples of $\mathcal{A}$-matroids arise from the following construction. Given a $m\times n$ matrix $A = (a_{i,j})$, and
 a $m$-uple $(i_1,\dots,i_m)$ of elements of $\{1,\dots,n\}$, we define  $p_A(i_1,\dots,i_m)$ to be the value of the determinant of the $m\times m$ matrix consisting of the columns $i_1,\dots,i_m$ of $A$.
 Note that the usual formula defining the determinant
 makes sense in any semiring system.


\begin{proposition}\label{Lap7}
  Suppose $(\mathcal A, \tT,(-),\preceq)$
  is a semiring system. Let $E:=\{1,\dots,n\}$,
  and let $A$ be a $m\times n$ matrix with entries
  in $\mathcal{A}$, with at least one maximal minor
  with values in~$\tT$.
  Then the map $p_A:E^m\to\mathcal{A}$ is a Grassmann-Pl\"ucker map
  of degree $m$.
\end{proposition}
\begin{proof}
  The strong transfer principle established in~\cite[Th.~3.4]{AGG1},
  based on an idea of Reutenauer and H.~Straubing~\cite[Lemma 3]{ReS},
  states that every classical polynomial identity, valid over rings, admits
  an analogue, valid over semirings, in which the equality relation
  is replaced by a surpassing relation. Hence, since in the special case of rings,
  the map $b$ satisfies the formula~\eqref{Plucker},
the same is true in the present generality.
\end{proof}

The $\mathcal{A}$-matroids provided
by \Cref{Lap7} could be called ``systemically representable.'' Some
interesting aspects of representability of matroids can be found in
\cite{IRh} in conjunction with \cite[Theorem~5.1.1]{RhS}. \cite{IRh}
shows that a significant class of simplices are not
Boolean-representable, cf.~\cite[Example~5.1.3]{RhS}, and
cf.~\cite[Example~5.2.12]{RhS} is a matroid which is not
Boolean-representable.

The following two examples illustrate the notion of matroids over semiring systems. \Cref{mat0} allows for Grassman-Pl\"ucker coordinates
to be {\em nontangible}, and this is helpful to carry information about tropically degenerate intersections.
\begin{example}
  Consider the following matrix over the tropical semiring
\begin{align} A = \left(\begin{array}{cccc}
0 & 0 & 2 & 3\\
0 & -1 & 0 & 3\\
0 & 0 & -1 & 0
  \end{array}\right)
  \label{e-deg-trop}
\end{align}
It describes the arrangement of tropical lines $L_1,\dots,L_4$
shown in~\Cref{fig-arrang-deg}, in which $L_i$ is determined
by the $i$th row of $A^\top$.
For instance, the fourth row of $A^\top$
encodes the tropical line
\[
L_4: \; \max(3+x_0,3+x_1,x_2) \; \text{achieved twice.}
\]
The tropical Grassmann-Pl\"ucker vector of $A$, interpreted
over the standard supertropical system, is shown on the figure.
Owing to the Cramer theorem over the supertropical semiring (see~\cite[Lemma~5.1]{theobald_first_steps}, \cite{IzhakianRowen2009TropicalRank} or~\cite[Theorem~6.2]{AGG2}),
the tangible character of $p_A(1,2,3)$ (the fact that $\{1,2,3\}$ is a basis)
indicates that the three
lines $L_1,L_2,L_3$ do not intersect. Likewise for the tropical lines
$L_1,L_3,L_4$ and $L_2,L_3,L_4$. However, the ghost nature
of the Pl\"ucker coordinate $p_A(1,2,4)$ indicates that $L_1,L_2,L_4$ do intersect.
\end{example}

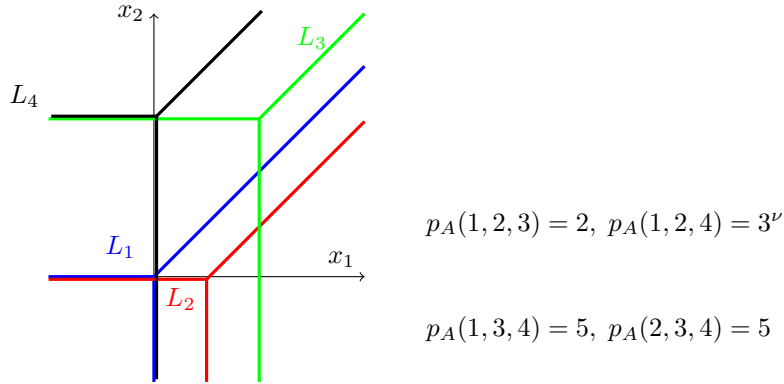
\begin{figure}
\begin{center}
  \begin{tikzpicture}[scale=0.7]
        \coordinate (xaxisleft) at (-2,0);
    \coordinate (xaxisright) at (4,0);
    \coordinate (yaxisbot) at (0,-2);
    \coordinate (yaxistop) at (0,5);
    \draw[->] (xaxisleft) --(xaxisright);
    \draw[->] (yaxisbot) --(yaxistop);
    \node[above left=0pt of {(4,0)}] {$x_1$};
         \node[left=0pt of {(0,5)}] {$x_2$};
    \coordinate (A) at (0,0);
    \coordinate (Ab) at (0,-2);
    \coordinate (Al) at (-2,0);
    \coordinate (Au) at (4,4);
    \node[above left=5pt of {(0,0)}] {\color{blue}$L_1$};
    \coordinate (B) at (1,0-0.05);
    \coordinate (Bb) at (1,-2);
    \coordinate (Bl) at (-2,0-0.05);
    \coordinate (Bu) at (4,3-0.05);
        \node[below left=1pt of {(1,0)}] {\color{red}$L_2$};
    \coordinate (C) at (2,3);
      \coordinate (Cb) at (2,3-5);
    \coordinate (Cl) at (3-5,3);
    \coordinate (Cu) at (2+2,3+2);
        \node[above=3pt of {(2+1,3+1)}] {\color{green}$L_3$};
    \coordinate (D) at (0+0.05,3+0.05);
    \coordinate (Db) at (0+0.05,3-5+0.05);
    \coordinate (Dl) at (0+0.05-2,3+0.05);
    \coordinate (Du) at (0+0.05+2,3+2+0.05);
        \node[above left=0pt of {(0-2,3+0.05)}] {$L_4$};
    \draw[very thick,draw=blue] (A) --(Ab);
    \draw[very thick,draw=blue] (A) --(Al);
    \draw[very thick,draw=blue] (A) --(Au);
    \draw[very thick,draw=red] (B) --(Bb);
    \draw[very thick,draw=red] (B) --(Bl);
    \draw[very thick,draw=red] (B) --(Bu);
    \draw[very thick,draw=green] (C) --(Cb);
    \draw[very thick,draw=green] (C) --(Cl);
    \draw[very thick,draw=green] (C) --(Cu);
    \draw[very thick] (D) --(Db);
    \draw[very thick] (D) --(Dl);
    \draw[very thick] (D) --(Du);
    \node[right=0pt of {(5,1)}] {$p_A(1,2,3) = 2 , \; p_A(1,2,4)= 3^\nu$};
     \node[right=0pt of {(5,-1)}] {$p_A(1,3,4) = 5, \; p_A(2,3,4)=5$};
  \end{tikzpicture}
\end{center}
\caption{The degenerate arrangement of tropical lines associated to the matrix~\eqref{e-deg-trop}, and the corresponding Grassman-Plücker coordinates over the supertropical system.}
\label{fig-arrang-deg}
\end{figure}
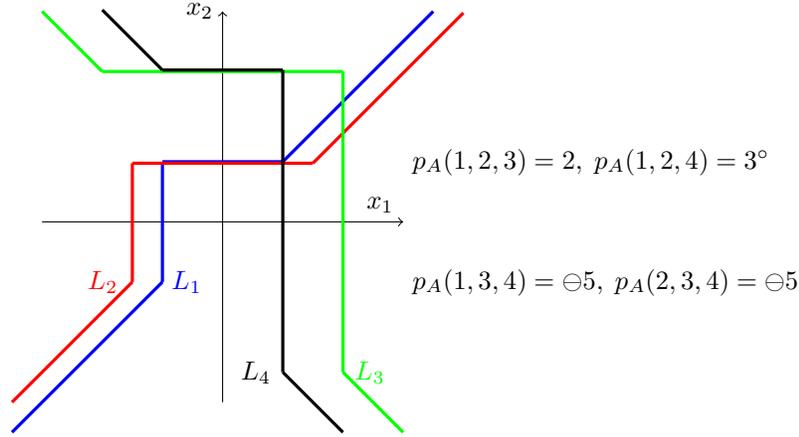
\begin{figure}
\begin{center}
  \begin{tikzpicture}[scale=0.4]
    \coordinate (xaxisleft) at (-6,0);
    \coordinate (xaxisright) at (6,0);
    \coordinate (yaxisbot) at (0,-6);
    \coordinate (yaxistop) at (0,7);
    \draw[->] (xaxisleft) --(xaxisright);
    \draw[->] (yaxisbot) --(yaxistop);
        \node[above left=0pt of {(6,0)}] {$x_1$};
         \node[left=0pt of {(0,7)}] {$x_2$};

    \coordinate (A) at (2,2);
    \coordinate (Al) at (-2,2);
    \coordinate (Au) at (7,7);
    \coordinate (Aau) at (-2,-2);
    \coordinate (Aaau) at (-7,-7);
    \node[right=0pt of {(-2,-2)}] {\color{blue}$L_1$};
    \coordinate (B) at (3,2-0.05);
    \coordinate (Bl) at (-3,2-0.05);
    \coordinate (Bu) at (8,7-0.05);
        \coordinate (Bau) at (-3,-2);
    \coordinate (Baau) at (-7,-6);

        \node[left=2pt of {(-3,-2)}] {\color{red}$L_2$};
    \coordinate (C) at (4,5);
      \coordinate (Cb) at (4,-5);
      \coordinate (Cl) at (-4,5);
      \coordinate (Cla) at (-6,7);
       \coordinate (Cba) at (6,-7);
        \node[right=1pt of {(4,-5)}] {\color{green}$L_3$};
    \coordinate (D) at (2,5+0.05);
    \coordinate (Db) at (2,-5);
    \coordinate (Dl) at (-2,5+0.05);
    \coordinate (Dba) at (4,-7);
    \coordinate (Dla) at (-4,7+0.05);
        \node[left=1pt of {(2,-5)}] {$L_4$};
        \draw[very thick,draw=blue] (A) --(Al);
        \draw[very thick,draw=blue] (Al) --(Aau);
                    \draw[very thick,draw=blue] (Aau) --(Aaau);
    \draw[very thick,draw=blue] (A) --(Au);
    \draw[very thick,draw=red] (B) --(Bl);
    \draw[very thick,draw=red] (B) --(Bu);
    \draw[very thick,draw=red] (Bl) --(Bau);
            \draw[very thick,draw=red] (Bau) --(Baau);
    \draw[very thick,draw=green] (C) --(Cb);
    \draw[very thick,draw=green] (C) --(Cl);
    \draw[very thick,draw=green] (Cl) --(Cla);
        \draw[very thick,draw=green] (Cb) --(Cba);
    \draw[very thick] (D) --(Db);
    \draw[very thick] (D) --(Dl);
        \draw[very thick,draw=black] (Dl) --(Dla);
        \draw[very thick,draw=black] (Db) --(Dba);


    \node[right=0pt of {(6,2)}] {$p_A(1,2,3) = 2 , \; p_A(1,2,4)= 3^\circ$};
     \node[right=0pt of {(6,-2)}] {$p_A(1,3,4) = \ominus 5, \; p_A(2,3,4)= \ominus 5$};

    \end{tikzpicture}
\end{center}
\caption{The degenerate arrangement of signed tropical lines associated to the matrix~\eqref{e-deg-trop-signed}, and the corresponding Grassman-Plücker coordinates over the system associated to the signed tropical hyperfield.}
\label{fig-arrang-deg-signed}
\end{figure}

\begin{example}
  Consider now the matrix
  \begin{align}
    A= \
 \left(\begin{array}{cccc}
 \ominus 0 & \ominus 0 & \ominus 2 & \ominus 3\\
 \ominus 0 & \ominus (-1) &  0 & 3\\
 0 & 0 & (-1) & 0
\end{array}\right) \enspace .
\label{e-deg-trop-signed}
\end{align}
with entries in the signed tropical semiring
(\Cref{ex-signedtrop}).
It determines the arrangement of signed tropical lines:
\[
L_1 : \ominus x_0 \ominus x_1 \oplus x_2 \nabla \zero,
\quad
L_2 : \ominus x_0 \ominus (-1)x_1 \oplus x_2 \nabla \zero,
\quad
L_3 : \ominus 2x_0 \oplus x_1 \oplus (-1)x_2 \nabla \zero,
\quad
L_4 : \ominus 3x_0 \oplus 3x_1 \oplus x_2 \nabla \zero \enspace,
\]
shown in~\Cref{fig-arrang-deg-signed}. Here, we draw only
the ``signed'' part of the tropical line, i.e., the signed
vectors satisfying the above relations.

The tropical Grassmann-Pl\"ucker vector of $A$, now interpreted
the semiring system associated to the signed tropical hyperfield,
is shown on the figure.
Owing to the Cramer theorem over the symmetrized tropical semiring
(see~\cite{Pl}, and~\cite{AGG1,AGG2} for more recent discussions),
the signed character of $p_A(1,2,3)$ indicates that the three
lines $L_1,L_2,L_3$ do not intersect. So do the tropical lines
$L_1,L_3,L_4$ and $L_2,L_3,L_4$. However, the balanced nature
of the Pl\"ucker coordinate $p_A(1,2,4)$
indicates that $L_1,L_2,L_4$ do intersect.
\end{example}

\begin{rem}\label{vals}
If $\mathbf A$ is a matrix over the field of complex
Puiseux series, then we can consider the matrix  $A:= -\val \mathbf A$ obtained by applying the valuation map~\eqref{e-def-val} to $\mathbf A$ entrywise, and interpret it as a matrix over the supertropical system.
We define the usual Grassman-Pl\"ucker map ${\mp}_{\mathbf{A}}$ and
the supertropical Grassman-Pl\"ucker map $p_A$,
as in \Cref{Lap7}.
Then, for all $(i_1,\dots,i_m)\in E^m$,
we have
\[
-\val (\mathbf p_{\mathbf A}(i_1,\dots,i_m)) \preceq p_A(i_1,\dots,i_m) \enspace.
\]
More generally, if $\mathbf A$ is a matrix over a field $\kfield$
with multiplicative subgroup $\hgroup$, and if $\pi$ denotes the canonical map
from $\kfield$ to $\kfield/\hgroup$ thought of as a hypersystem, in which
$\preceq$ is interpreted as inclusion, we have
\[
\pi (\mathbf p_{\mathbf A}(i_1,\dots,i_m)) \preceq p_A(i_1,\dots,i_m) \enspace.
\]
\end{rem}
\begin{rem}
In the definition of valuated matroid, \cite[Definition~1.1]{DrW}
replaces Condition \eqref{mat0-iii} of \Cref{mat0} by
$b(e_{\pi(1)}, \dots, e_{\pi(m)})=  b(e_1, \dots,
 e_m)\quad \text{ for every permutation } \pi \in S_m.$
 Let $\binom {E}m$ denote the set of (unordered) subsets of
    $m$ elements of $E$. Now  the  maps $ b:
E^m \to  {\mathcal A} $ can be viewed more concisely as maps $ b:
\binom {E}m \to  {\mathcal A} $. This is the same as
\Cref{mat0}~\eqref{mat0-iii}  when $(-)$ is of the first
 kind.
\end{rem}
\subsection{A useful $\tT$-pre-order on $(-)$-bipotent triples}\label{preo12}$ $

In order to interpret~\Cref{mat0} in terms of {\em exchange relations},
and to compare it with~\cite{DrW}, we  equip $\mathcal{A}$ with
a new pre-order.

\begin{definition}\label{circord0}
  Suppose  $(\mathcal{A},\tT,(-))$ is a $(-)$-bipotent cancellative triple. Recall by \Cref{circint3} that for any $b\in\mcA$, $b=m b_\tT $ is a uniform presentation  of $b$ (where $b_\tT$ is unique when $m\neq 2$ and $b_\tT^\circ=b$ when $m=2$). The $\circ$-\textbf{pre-order} is defined as   $b_1\le_\circ b_2$ if $(b_1)_\tT ^\circ \preceq _\circ(b_2)_\tT^\circ $, where $\preceq_\circ$ is the $\circ$-surpassing relation.
\end{definition}

\begin{lem}\label{leq-order00}
Assume that $(\mathcal{A},\tT,(-))$ is a $(-)$-bipotent cancellative
triple, and $c \leq_\circ
c'$ for $c,c' \in \mcA$.
 If  $c_\tT^\circ \ne  {c'_\tT}^\circ ,$ then $c+ c' = c',$ and $(-)c +c' = c'.$
 \end{lem}
\begin{proof}   By assumption $c_\tT ^\circ\ne  {c'_\tT}^\circ$,
and  $c_\tT^\circ \preceq_\circ {c'_\tT}^\circ ,$
so  $c_\tT^\circ +d^\circ={c'}_\tT^\circ$  for some $d\in \mcA\setminus\{\zero\}$.
If   $d_\tT ^\circ = c_\tT ^\circ,$ then
${c'}_\tT^\circ$ is a sum of copies of $c_\tT^\circ$, so by
uniqueness of the presentation of  ${c'}_\tT^\circ$,
see \Cref{prop-uni}, we get ${c'_\tT} ^\circ = c_\tT ^\circ,$ a
contradiction. Thus $d_\tT ^\circ \neq c_\tT ^\circ$, so in particular $d_\tT \neq (\pm) c_\tT.$
By \Cref{neg3}, we have
 $c_\tT+ d_\tT$ is either $c_\tT$ or $d_\tT.$ Hence  $c_\tT+ d$ is either $c_\tT$ or $d$, and so ${c_\tT'}^\circ = c_\tT^\circ+ d^\circ$ is either $c_\tT^\circ$ or $d^\circ.$ In the
former case we have
a contradiction.
 In the latter case 
 ${c'}^\circ_\tT = d^\circ,$ so
$(c_\tT+ {c'}_\tT)^\circ = {c'}_\tT^\circ.$
Then, applying the same arguments as before with $c'_\tT$ instead of $d$, we deduce that $c_\tT+ c'_\tT=c'_\tT$, and therefore $c+c'=c'$.
The second equality is by symmetry.
\end{proof}

\begin{lem}\label{leq-order1}
Assume that $(\mathcal{A},\tT,(-))$ is a $(-)$-bipotent cancellative triple. Then $\leq_\circ$ is ``almost''
a total order on $\mathcal{A}$ in the following sense.
  For all $c,c' \in \mcA$, \begin{enumerate}
    \item  \label{leq-order1-1}
 $c\leq_\circ c'$ or $c'\leq_\circ c$;
  \item  \label{leq-order1-2} we have both $c\leq_\circ c'$ and
  $c'\leq_\circ c$, if and only if $c_\tT  = (\pm) {c'_\tT} $.
\end{enumerate}
Moreover, on~$\mathcal{A}^\circ$, $\leq_\circ$ is a total order  which coincides with $\preceq_\circ$.
\end{lem}
\begin{proof} Let $c,c'\in \mathcal A$ and take their uniform presentations $c=mc_\tT$ and  $c'=m'c'_\tT$. If $ c_\tT \ne (\pm) c_\tT'$, then \Cref{neg3} shows that $c_\tT+c'_\tT=c_\tT$ or $c'_\tT$, which implies  $c'_\tT\preceq_\circ c_\tT$ or $c_\tT\preceq_\circ c'_\tT$, that is
$c'\leq_\circ c$ or $c\leq_\circ c'$. Otherwise, we get that $c_\tT^\circ={c'}_\tT^\circ$, so we have both
$c'\leq_\circ c$ and $c\leq_\circ c'$.
This verifies \eqref{leq-order1-1} and the if part of \eqref{leq-order1-2}.

Let us verify the only if part in \eqref{leq-order1-2}. 
If $c\leq_\circ c'$ and $c'\leq_\circ c$, then $ c_\tT^\circ= {c'}_\tT^\circ$ by definition, and
by \Cref{prop-uni}, we get $c_\tT= (\pm) c'_\tT$. 
This finishes the proof of \eqref{leq-order1-2}.
Applying these properties to the elements of $\mathcal{A}^\circ$,  this shows that $\leq_\circ$ is a total order  on $\mathcal{A}^\circ$.
\end{proof}

\subsection{Exchange relations} $ $

Now we  equip $\mathcal{A}$ with
the $\circ$-pre-order of the previous subsection.

\begin{proposition}\label{Lap700} (Compare with \cite[Theorem~4.3]{DrW}).
Suppose that $(\mathcal A, \tT,(-), \preceq)$ is a $(-)$-bipotent cancellative system, with the $\circ$-pre-order $\le_\circ$ of  \Cref{circord0}.
Then any $\mathcal{A}$-matroid  over $E$ satisfies:

   For $e_0, \dots e_m, f_2, \dots , f_m \in
    E$ with $b(e_1, \dots , e_m) b(e_0, f_2, \dots   f_m)\in \tT,$ there exists some
    $i$,
$1 \leq  i \leq m$, with \begin{equation}\label{Plucker1} b(e_1, \dots ,
e_m) b(e_0, f_2, \dots f_m) \le_\circ b(e_0,  \dots , \hat e_i, \dots,
e_m) b(e_i, f_2, \dots ,f_m).
\end{equation}
\end{proposition}
\begin{proof}
The  proof parallels \cite[Lemma~4.6]{DrW}. Namely, let $c$  run
over the elements $$c_i:=b(e_0,e_1, \dots, e_{i-1}, e_{i+1},
 e_m)b(e_i,f_2,  \dots,
 f_m) , \qquad i=0,\dots,m.$$
By \Cref{leq-order1}, there exists $i \in \{0,\ldots , m\}$ such
that $c_i^\circ$ is maximal for this order. This implies that $c_j\leq_\circ
c_i$ for all $j=0,\ldots , m$.

If there exists such an $i$ which is $\geq 1$, then \eqref{Plucker1}
is satisfied, since the left hand side of~\eqref{Plucker1} is $c_0$ and the right hand side is $c_i$.

Otherwise all $i\geq 1$ are such that ${{c_i}^\circ}$ is not
maximal. This implies that $c_0^\circ>_\circ c_j^\circ$,
so $c_0>_\circ (-)^j c_j$ 
 for all $j\geq 1$. Since $c_0\in\tT$, \Cref{leq-order00}
implies that $c_0+(-)^j c_j=c_0$ for all $j\geq 1$. By induction, we
obtain that the sum  in \eqref{Plucker} is equal to $c_0\in \tT$ and
is thus not in ${\mathcal A}_{\Null}$ (since $\preceq$ is a surpassing relation), a contradiction to
\eqref{Plucker}.
\end{proof}
When the values of $b$ belong to $\tTz$, the properties stated
in~\Cref{Lap700} define a valuated matroid in the sense
of~\cite{DrW}.

\begin{rem} Conversely, the property stated in \Cref{Lap700} implies the Pl\"ucker relation~\eqref{Plucker} when the negation map is of the first kind, supposing that $\mathcal{A}$ is strongly $(-)$ bipotent and cancellative (\Cref{circidem}). Indeed, we are in the same situation as  \cite[Theorem~4.3]{DrW}):
if \eqref{Plucker1} holds for all $e_0, \dots e_m, f_2, \dots , f_m \in
    E$, then,
    choosing a maximal summand in~\eqref{Plucker},
    either this element is in $\tT$, and then by  \eqref{Plucker1},
    there must be another maximal summand in~\eqref{Plucker},
    since $\leq$ is a total order (see \Cref{leq-order00}), these two summands are equal,
    and then the  sum in~\eqref{Plucker} is in ${\mathcal A}_\Null$,
    since $\mathcal A$ is strongly bipotent.
    Otherwise, this element is not in $\tT$, and since necessarily
    $\mathcal A$ is shallow (see \Cref{ht2}), then again  the  sum in~\eqref{Plucker} is in ${\mathcal A}_\Null$.

In general, the property stated in~\Cref{Lap700} does not imply the Pl\"ucker
relation. Indeed, the Plucker relations involve the sign, which does not appear
any more in~\eqref{Plucker1}. For instance, an oriented matroid may be thought
of as a matroid over the hypersystem of the hyperfield $F/F_{>0}$
where $F$ is an ordered field, and~\eqref{Plucker1} does not capture
the oriented matroid axioms.
\end{rem}

\appendix
\section{Abstract $\tTz$-modules}

The definition of systems in \cite{Row21} does not require the monoid $\tT$ to be contained in $\mcA,$ so we review that more general set-up and show how it relates to what we have done until now.
Towards that end, we redefine $\tTz$-modules.

\subsection{$\tTz$-modules}$ $

We assume now that  $(\tTz,\cdot)$ is a
semigroup,
with an absorbing element $\zero_\tT.$

\begin{definition}\label{modu13}
A (left) $\tTz$-\textbf{action} on a monoid $(\mathcal A,+,\zero)$
is a scalar multiplication
$\tTz\times \mathcal A \to \mathcal A$ (denoted also either as $\cdot$ or
concatenation), which satisfies the following properties:
\begin{enumerate}
\item $\zero_{\tTz} \mcA = \zero$ (in $\mcA$).
  \item Compatibility with the multiplication on $\tTz$:
$$(a_1 a_2)b =  a_1 (a_2b),\qquad \text{for all}\; a_i \in
\tTz, \; b   \in \mathcal A \enspace.$$
\item If $\tTz$ has a unit $\one_{\tTz}$, then
$$\one_{\tTz} b = b ,\qquad \text{for all}\; b   \in \mathcal A \enspace.$$
  \item Distributivity,  in the sense that
$$a(b_1+b_2) = ab_1 +ab_1,\qquad \text{for all}\; a \in \tTz,\; b_i \in \mathcal A\enspace .$$
\item
$a\zero=\zero$,
for all $a\in\tTz\enspace .$
\item  $\zero_{\tTz} a = \zero,$ for all $a\in {\mathcal A}.$
\end{enumerate}

A (left) $\tTz$-\textbf{module} 
is a monoid
$(\mathcal A,+,\zero)$ together with a $\tTz$-action.
(We do not assume now that $\tTz \subseteq \mcA.$)
As before, the left action of $\tTz$ on $\mathcal A$ is \textbf{cancellative} when
    $a c_1 = a c_2$ iff $ c_1 = c_2$, for all $a\in \tTz \setminus \{\zero\},$
$ c_i \in \mathcal A$.
A right  $\tTz$-\textbf{module} is defined analogously. 
A  $\tTz$-\textbf{bimodule} is a left and right $\tTz$-module
satisfying  $(a_1 b)a_2 = a_1(b a_2)$ for all $a_i\in \tTz,$ $b\in \mcA.$
\end{definition}

 \subsection{Obtaining triples from $\tTz$-bimodules}

  Towards the end of making  $(\mathcal A,+, \zero_{\mathcal{A}})$  a semiring,
we shall use the following notion.

\begin{definition}\label{embedding0}
 A \textbf{preunit} of a
$\tTz$-bimodule $(\mathcal A,+, \zero)$ is an element $u \in \mathcal
A$ satisfying the following properties $\forall a,a' \in \tTz$:
\begin{enumerate}
 \item $au = ua;$
\item $au = a'u$ implies $a=a';$
\item $\mathcal A \setminus \{ \zero_{\mathcal{A}} \}$ is included in the additive span of $\tTz u= u\tTz$.
\end{enumerate}
\end{definition}

Identifying $a\in \mathcal{T}$ with $au\in \mathcal{A}$, we can think of
$\mathcal{A}\setminus\{\zero_{\mathcal{A}}\}$ as the additive span
of  a distinguished subset~$\tTz$, which
additively spans $\mathcal A$.


Extending Proposition~\ref{dc} to $\tT$-bimodules, we have

\begin{proposition}\label{dc1} Suppose that $\mathcal A$ is a $\tTz$-bimodule with a preunit $u$. We can define a semiring structure on $\mathcal A$, whose  unit element will be $u$,
extending the given $\tTz$-bimodule structure,
 defining multiplication  by  \begin{equation}\label{eq2.1}
\left(\sum_i
 a_i u\right)\left(\sum_j a_j'u\right) = \sum_{i,j}
 (a_ia_j')u.\end{equation} If $\mathcal A$ already is a semiring,
 with a multiplication consistent with the
action of $\tTz$ and with unit element $u=\one_{\mathcal A}$,
then the multiplication in $\mathcal A$ coincides with \eqref{eq2.1}.
\end{proposition}
\begin{proof} First notice that the formula is well-defined, since if $\sum_j a'_j u =
\sum_k a''_k u$ then $$\sum_{i,j} (a_i a'_j)u   =
\sum_{i,j} a_i (a'_ju )
=\sum_i a_i
(\sum_j  a'_j u  )
=\sum_i a_i
(\sum_k  a''_k u  )
=
\sum_{i,k} (a_i a''_k)u ,$$ and likewise if we change the $a_i$.
Now, double distributivity is clear.

 Next, $u$ is the unit element
since for any nonzero element $b=\sum_i a_i u$ of $\mathcal A$,
we have $b u= b (\one_{\tTz} u)= \sum_i (a_i\one_{\tTz})  u=b$
 and similarly $ub=b$. Associativity follows from \eqref{eq2.1} and
the associativity of multiplication in $\tTz$.

To show the second part of the proposition, we assume that $\mathcal A$ is a semiring with a multiplication $\times$
consistent with the action of $\tTz$,  and unit $\one_{\mathcal A}$, and that $u=\one_{\mathcal A}$.
Let $a,a'\in \tTz$.
From the consistency of multiplication with the action and the properties of a
preunit, we have $(a u)\times (a'u)=a (u\times (ua'))=a ((u\times u) a')$.
Since $u=\one_{\mathcal A}$, we have $u\times u=u$, which leads to
$(a u)\times (a'u)= a (u a')= a (a'u)=(aa') u$.
Then,  from the distributivity of $\times$ over addition, we
obtain that
$\left(\sum_i  a_i u\right)\times \left(\sum_j a_j'u\right) = \sum_{i,j}
(a_i u)\times (a_j' u)= \sum_{i,j} (a_i a'_j)u$, which means that
the multiplication in $\mathcal A$ coincides with \eqref{eq2.1}.
\end{proof}

We shall also need the following notion.
\begin{definition}\label{modrel}
A subset $\mathcal B$ of the (left) $\tTz$-\textbf{module} $\mathcal A$
is a (left) $\tTz$-\textbf{sub-module} of $\mathcal A$ if it is a sub-semigroup of $\mathcal A$ and if, for all $a\in \tTz$ and $b\in \mathcal B$, we have $a b\in {\mathcal B}$.

\end{definition}

The next definition transports $\tTz$ into $\mcA.$

\begin{definition}\label{int1}
 A  \textbf{pseudo-triple}
is a collection $(\mathcal A, \tTz, (-)),$  where  $\mathcal A$ is a
$\tTz$-bimodule and $(-)$ is a
negation map on  $\mathcal A $.

A $\tTz$-\textbf{pseudo-triple} is a  pseudo-triple where $\tTz \subseteq \mathcal A.$

 A \textbf{triple} $(\mathcal A, \tTz, (-))$ is a $\tTz$-pseudo-triple satisfying:
\begin{enumerate}
 \item \label{triple2}$\tTz$
generates $( \mathcal A,+),$
\item \label{triple3}$ \tTz \cap \mathcal A^\circ = \{ \zero
 \}.$
\end{enumerate}

\end{definition}

\begin{remark}\label{circint1}$ $\begin{enumerate}

\item The basic property \eqref{triple3} of \Cref{int1} holds in uniquely negated $\tTz$-nontrivial $\tTz$-pseudo-triples $(\mathcal A, \tTz, (-))$   when the action
is cancellative, by \cite[Proposition~2.21]{Row21}.

   \item  If $\mathcal A$ is a $\tTz$-(bi)module, then $\mathcal A^\circ$ is a
$\tTz$-sub(bi)module of $\mathcal A$.

\end{enumerate}
\end{remark}

\begin{remark}\label{dd} Let  $\mathcal A$ be a  $\tTz$-bimodule with a preunit $u$.
 If $(-)$ is a negation map on $\mathcal A$, then
 $(\mathcal A, \tTz, (-))$ is a $\tTz u$-pseudo-triple.
 Replacing $\mathcal A$ by the semigroup
 spanned by $\tTz u$ yields condition \eqref{triple2}
 of \Cref{int1}. The (internal) multiplication on $\mathcal{A}$ given by
\Cref{dc} makes $\mathcal{A}$  a semiring. It is a triple if and only if it also satisfies \eqref{triple3}, which is automatic when $\tTz \setminus \{\zero\}$ is a group (barring the degenerate case $\mathcal A^\circ = \mathcal A$;
 if on the contrary $au\succeq \zero$ then $u \succeq \zero,$ so $\mcA \subseteq \mcA_0.$). Moreover,
 in this case,
$(\mathcal A, \tTz, (-))$ is automatically a semiring triple.
\end{remark}

\begin{definition}\label{su}$ $
 \begin{enumerate}
 \item A relation $\mathcal R$ on a (left) $\tTz$-module $\mathcal A$
is called a (left) $\tTz$-\textbf{module relation} if it satisfies,
for $a\in \tTz$ and $b_i,b_i' \in \mathcal A$:
\begin{enumerate}
 \item If $b_i {\mathcal R} b_i'$ for $i= 1,2$, then  $b_1 + b_2 {\mathcal R} b_1' + b_2'.$
    \item   If  $b_1 {\mathcal R} b_1'$ then $a b_1 {\mathcal R} ab_1'.$
\end{enumerate}
Right $\tTz$-module relations are defined analogously.

     \item  A \textbf{surpassing relation} on a $\tTz$-pseudo-triple is a pre-order $\tTz$-module relation which reduces
to equality on the set of tangible elements, and for which all quasi-zeros surpass zero.

   \item  A \textbf{system}  $(\mathcal A, \tTz, (-), \preceq)$
is a triple  $(\mathcal A, \tTz, (-))$ together with a
surpassing relation~$\preceq$ satisfying the following property:

$\tTz$ is uniquely quasi-negated  over the submodule   $$
{\mathcal A}_{\Null}: = \{ b \in \mathcal A : b \succeq
\zero\}.$$
\end{enumerate}
\end{definition}

\begin{remark}$ $
   Note that this definition of surpassing relation does not relate to negation, so we have a category $(\MMod;\preceq)$ of $\tTz$-modules with surpassing relation $(\preceq)$ and their $\preceq$-morphisms. This idea is developed in \cite{AGR1}.
\end{remark}

We can  extend notions from the body of the paper. In analogy to \Cref{func1} we have

\begin{theorem}\label{dou}
Symmetrization provides the  doubling functor
  from the category $\MMod$ of $\tTz$-bimodules  to the subcategory of pseudo-triples, which restricts to the
 doubling functor
  from the category $(\MMod;\preceq)$   to the subcategory of pseudo-triples with surpassing relation. These subcategories are reflective.
\end{theorem}
\begin{proof}
Just as in the proof of \Cref{func1} (i),(ii),(iv),(v). One checks that the symmetrization of a \distributed left $\tT$-bimodule is a pseudo-triple,
and if there is a surpassing relation, it is carried over.

The morphisms of the doubled category are pairs $\hat f =(f_1,f_2):(\widehat{\mcA},\widehat{\mcA_0})\to (\widehat{\mcA'},\widehat{\mcA'})$, where $f_1,f_2:\mcA \to \mcA'$ are morphisms, and we define  $$\hat f(b_1,b_2) = (f_1(b_1)+f_2(b_2),f_2(b_1)+ f_1(b_2)),$$ which is checked routinely to be a morphism.

Given a morphism $f:(\mcA,\mcA_0) \to (\mcA',\mcA'),$
        we define $\hat f:(\widehat{\mcA},\widehat{\mcA_0})\to (\widehat{\mcA'},\widehat{\mcA'})$, by putting
        $\hat f = (f,f).$ In the other direction, one can project onto the first coordinate.
\end{proof}

This   approach also permits us to generalize
\Cref{AGGGexmod1}.
Now the monoid $\tT_L$ only needs to act on $L$,  and $L$ could be a semiring without $\zero.$

 This more general approach also gives a way to view matrices as systems.

 \begin{definition}\label{conv1}$ $
  \begin{enumerate}
      \item If $\mcA$ is an nd-semiring and $S$ is a semigroup, define the \textbf{semigroup nd-semiring}   $\mcA[S] $ to be $\mcA^{(S)}$ endowed with the \textbf{convolution product}  as in \Cref{conv}.
         \item In particular, taking $S = \{ e_{i,j}: 1\le i,j \le n \}\cup \{\zero\}, $ made into a monoid via matrix unit multiplication $$ e_{i,j} e_{k,\ell} = \delta_{j,k}e_{i,\ell},$$
         one can define the matrix nd-semiring $M_n(\mcA).$
         This produces the dominant matrix functor
         (where morphisms, negation maps, and surpassing relations are defined componentwise).
  \end{enumerate}

 \end{definition}

 The polynomial system and matrices can be viewed in terms of direct sums, to be treated more generally in \cite{AGR1}.


\bibliography{bibliolouis}
\bibliographystyle{alpha}
\end{document}

%% file: tropMacro1.tex
\usepackage{epsfig,latexsym,amsfonts,amssymb,amsmath,amscd,graphics}
\usepackage{amsfonts,amssymb,amsmath,amscd,amsthm}
\usepackage[mathscr]{eucal}
\usepackage{mathrsfs}

\usepackage{mathbbol}

\usepackage{oldgerm,units}
\usepackage{wrapfig,epsfig}

\usepackage{ifthen}
\usepackage{mathbbol}

\usepackage{amsthm}
\usepackage[mathscr]{eucal}
\usepackage{mathrsfs}
\usepackage{mathbbol}
\usepackage{oldgerm,units}
\usepackage{wrapfig}

\newtheorem{theorem}{Theorem}[section]
\newtheorem{proposition}[theorem]{Proposition}
\theoremstyle{definition}
\newtheorem{definition}[theorem]{Definition}

\theoremstyle{plain}

\theoremstyle{remark}

\newtheorem{example}[theorem]{Example}
\newtheorem{remark}[theorem]{Remark}
\theoremstyle{plain}

\theoremstyle{remark}
\newtheorem{note}[theorem]{Note}

\newcommand{\Real}{\mathbb R}

\newcommand{\Int}{\mathbb Z}
\newcommand{\Net}{\mathbb N}

\newcommand{\one}{\mathbb{1}}
\newcommand{\zero}{\mathbb{0}}

\newcommand{\trop}[1]{\mathcal{#1}}

\newcommand{\tG}{\trop{G}}

\newcommand{\tT}{\trop{T}}

\newcommand{\eps}{\varepsilon}

    \ifx\proof\undefined
    \newenvironment{proof}{
    \smallskip
    \noindent\emph{Proof.}}{\hfill\(\Box\)
    \bigskip
    } \fi

\newcommand{\ifdef}[3]{\ifthenelse{\equal{#1}{true}}{#2}{#3}}